\RequirePackage{fix-cm}
\documentclass[12pt]{article}       % twocolumn
\usepackage{graphicx}
\usepackage{mathptmx}      % use Times fonts if a\newcommand{\tHP}{\tilde {\mathcal H}_P}vailable on your TeX system
\usepackage{multirow}
\usepackage{makecell}
\usepackage{pdflscape}
\usepackage{tikz}
\usetikzlibrary{cd}
\usepackage{multicol}
\usepackage{amssymb,amsmath,amsthm}
\newcommand{\R}{\mathbb{R}}
\newcommand{\Z}{\mathbb{Z}}
\newcommand{\RP}{{\mathcal R}_P}
\newcommand{\PP}{\mathbb{P}}
\newcommand{\HP}{{\mathcal H}_P}
\newcommand{\RPc}{{\mathcal R}_{P,c}}
\newcommand{\RPcs}{{\mathcal R}_{P,s,c}}
\newcommand{\A}{\mathbb{A}}
\newcommand{\tHP}{\tilde {\mathcal H}_P}
\newcommand{\RF}{{\mathcal R}_F}

\allowdisplaybreaks
\usepackage{geometry}
\numberwithin{equation}{section} 
\newtheorem{theorem}{Theorem}[section] 
\newtheorem{proposition}[theorem]{Proposition} 
\newtheorem{lemma}[theorem]{Lemma} 
\newtheorem{corollary}[theorem]{Corollary}

\newtheorem{problem}{Problem} 
\theoremstyle{definition}

\newtheorem{example}[theorem]{Example}
\newtheorem{remark}[theorem]{Remark} 
\theoremstyle{remark}
\newtheorem{claimnumbered}{Claim}

\begin{document}

\title{Projective geometry of Wachspress coordinates%\thanks{Grants or other notes
}


\author{Kathl\'e{}n Kohn         \and
        Kristian Ranestad %etc.
}

\maketitle

\begin{abstract}
We show that there is a unique hypersurface of minimal degree passing through the non-faces of a polytope which is defined by a simple hyperplane arrangement.
This generalizes the construction of the adjoint curve of a polygon by Wachspress in 1975.
The defining polynomial of our adjoint hypersurface is the adjoint polynomial introduced by Warren in 1996.
This is a key ingredient for the definition of Wachspress coordinates, which are barycentric coordinates on an arbitrary convex polytope.
The adjoint polynomial also appears both in algebraic statistics, when studying the moments of uniform probability distributions on polytopes,
and in intersection theory, when computing Segre classes of monomial schemes. 
We describe the Wachspress map, the rational map defined by the Wachspress coordinates, and the Wachspress variety, the image of this map.
The inverse of the Wachspress map is the projection from the linear span of the image of the adjoint hypersurface.
To relate adjoints of polytopes to classical adjoints of divisors in algebraic geometry,
we study irreducible hypersurfaces that have the same degree and multiplicity along the non-faces of a polytope as its defining hyperplane arrangement.
We list all finitely many combinatorial types of polytopes in dimensions two and three for which such irreducible hypersurfaces exist. 
In the case of polygons, the general such curves< are elliptic.
In the three-dimensional case, the general such surfaces are either K3 or elliptic.
\end{abstract}

\section{Introduction}
Barycentric coordinates on convex polytopes have many applications, such as mesh parameterizations in geometric modelling, deformations in computer graphics, or polyhedral finite element methods.
Whereas barycentric coordinates are uniquely defined on simplices, there are different versions of barycentric coordinates on more general convex polytopes.
For instance, \emph{mean value coordinates} and \emph{Wachspress coordinates} are both commonly used in practice. 
A nice overview on different versions of barycentric coordinates, their history and their applications is~\cite{Flo15}.

This article provides an algebro-geometric study of Wachspress coordinates,
which were first introduced for polygons by Wachspress~\cite{Wachs75} and later generalized to higher dimensional convex polytopes by Warren~\cite{War96}.
Wachspress defined the \emph{adjoint curve} of a polygon 
as the minimal degree curve passing through the intersection points of pairs of lines containing non-adjacent edges of the polygon;
see Figure~\ref{fig:wachspress}.
Warren defined the \emph{adjoint polynomial} for any convex polytope $P$ in $\R^n$:
he first fixes a triangulation $\tau(P)$ of $P$ into simplices such that the vertices of each occurring simplex $\sigma$ are vertices of $P$,
and then defines the \emph{adjoint} to be the polynomial
\begin{equation}
\label{eq:warrenAdjoint}
	\mathrm{adj}_{\tau(P)}(t) := \sum_{\sigma \in \tau(P)} \mathrm{vol}(\sigma) \prod_{v \in V(P) \setminus V(\sigma)} \ell_v(t),
\end{equation}
where $t = (t_1, t_2, \ldots, t_n)$, the set of vertices of a polytope is denoted by $V(\cdot)$, and $\ell_v(t)$ denotes the linear form $1-v_1t_1-v_2t_2 - \ldots - v_nt_n$ associated to a vertex $v$.
Warren shows that the adjoint is independent of the chosen triangulation, so we set $\mathrm{adj}_{P} := \mathrm{adj}_{\tau(P)}$.
Moreover, he observes that the adjoint of $P$ vanishes on the intersections of pairs of hyperplanes containing non-adjacent facets of the dual polytope $P^\ast$.
If $P$ is a polygon, this implies that $\mathrm{adj}_P$ coincides with Wachspress' adjoint of the dual polygon $P^\ast$.
Finally, Warren defines the \emph{Wachspress coordinates} of a polytope $P$ as follows:
\begin{equation}
\label{eq:barycentric}
	\forall u \in V(P) : \quad \beta_u(t) := \frac{\mathrm{adj}_{F_u}(t) \cdot \prod \limits_{F \in \mathcal{F}(P) : \, u \notin F} \ell_{v_F}(t) }{\mathrm{adj}_{P^\ast}(t)},
\end{equation}
where $\mathcal{F}(P)$ denotes the set of facets of $P$,
$F_v$ denotes the facet of the dual polytope $P^\ast$ corresponding to the vertex $v \in V(P)$,
and $v_F$ denotes the vertex of the dual polytope $P^\ast$ corresponding to the facet $F \in \mathcal{F}(P)$.

Warren's adjoint also appears in algebraic statistics and intersection theory.
\begin{figure}
\centering
\includegraphics[width=0.33\textwidth]{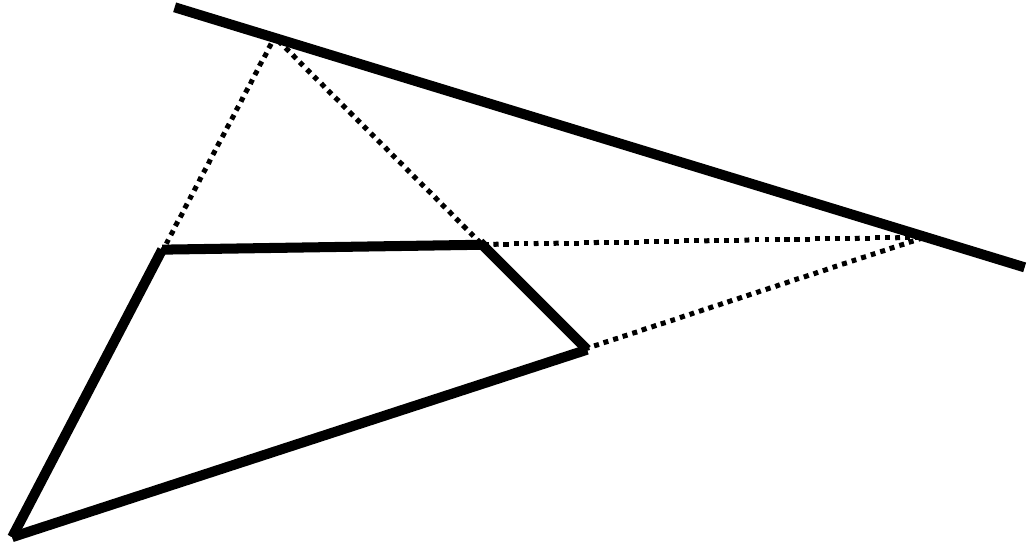}
~~~~~~~~~~
\includegraphics[width=0.33\textwidth]{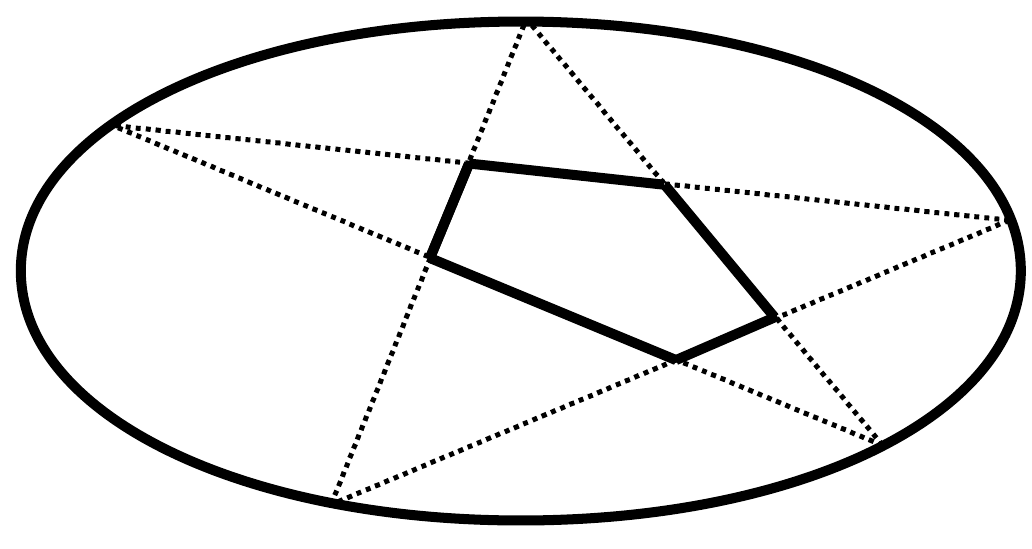}
\caption{Adjoint curves of a quadrangle and a pentagon.}
\label{fig:wachspress}
\end{figure}

\paragraph{\textbf{Intersection theory.}} 
The adjoint is the central factor in Segre classes of monomial schemes.
This was observed by Aluffi~\cite{aluffi1,aluffi2}, but without the language of adjoints. 
To explain Aluffi's setup, we consider a smooth variety $V$ with smooth hypersurfaces $X_1, X_2, \ldots, X_n \subset V$ meeting with normal crossings in 
$V$.
For an index tuple $I = (i_1, i_2, \ldots, i_n) \in \Z_{\geq 0}^n$,
we write $X^I$ for the hypersurface obtained by taking $X_{j}$ with multiplicity $i_j$.
Any finite subset $\mathcal{A} \subset \Z_{\geq 0}^n$ of index tuples defines a \emph{monomial subscheme} $S_\mathcal{A} := \bigcap_{I \in \mathcal{A}} X^I$
as well as a \emph{Newton region} $N_\mathcal{A}$ which is the complement in $\R^n_{\geq 0}$ of the convex hull of the positive orthants
translated at the points in $\mathcal{A}$,
i.e., $N_\mathcal{A} = \R^n_{\geq 0} \setminus \mathrm{convHull} \left(\bigcup_{I \in \mathcal{A}} (\R^n_{>0} + I) \right)$.
From Aluffi's results~\cite{aluffi1,aluffi2} we deduce the following in Section~2:
\begin{proposition}
\label{prop:segreClass}
If the Newton region $N_\mathcal{A}$ is finite, the Segre class of the monomial subscheme $S_\mathcal{A}$ in the Chow ring of $V$ is
\begin{equation}
\label{eq:segreClass}
\frac{n! \cdot X_1 \cdot X_2 \cdots X_n \cdot\mathrm{adj}_{N_\mathcal{A}}(-X_1, -X_2, \ldots, -X_n)}{\prod \limits_{v \in V(N_\mathcal{A})} \ell_v(-X_1, -X_2, \ldots, -X_n)}.
\end{equation}
\end{proposition}

\begin{remark}
If $N_\mathcal{A}$ has vertices at infinity, these are points at infinity in the direction of the standard basis vectors $e_1, e_2, \ldots, e_n$.
The Segre class of  $S_\mathcal{A}$ is the limit of~\eqref{eq:segreClass}.
This limit is a rational function of the same form as~\eqref{eq:segreClass}
where the linear form $\ell_{v_i}$ associated to a vertex $v_i$ at infinity in the direction of $e_i$ is $\ell_{v_i}(t) = -t_i$;
see Remark~\ref{rem:infiniteNewtonRegionExtended}.
 \hfill$\diamondsuit$
\end{remark}

\begin{multicols}{2}
\begin{example}
We are using the running example in~\cite{aluffi1}.
Here $n = 2$ and $\mathcal{A} = \lbrace (2,6), (3,4), (4,3), (5,1), (7,0) \rbrace$.
The Newton region $N_\mathcal{A}$ is a hexagon with a vertex at infinity in the direction of the second standard basis vector. 
We will see in Remark~\ref{rem:infiniteNewtonRegionExtended}
that the Segre class of the monomial scheme $S_\mathcal{A}$ is
\end{example}
\columnbreak
\centering
\includegraphics[width=0.3\textwidth]{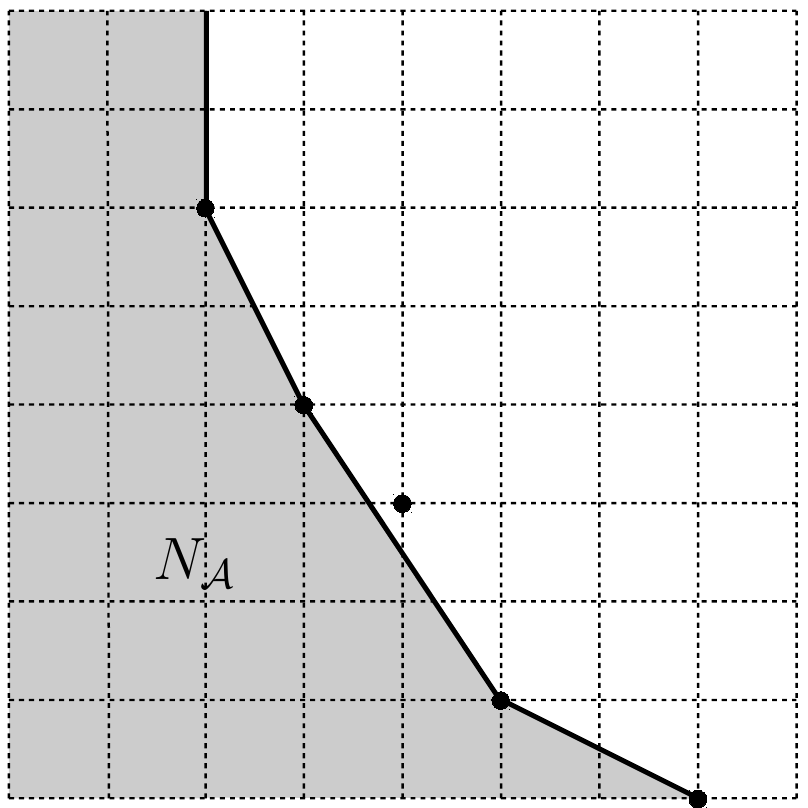}
\end{multicols}
\vspace*{-6mm}
$$
\frac{2X_1X_2 \, \mathrm{adj}_{N_\mathcal{A}}(-X_1,-X_2) }{X_2(1+2X_1+6X_2)(1+3X_1+4X_2)(1+5X_1+X_2)(1+7X_1)},
\quad\text{ where }
$$
$$
\hspace*{13mm}
\mathrm{adj}_{N_\mathcal{A}}(t) = 1-15  t_1 -22 t_2 +71 t_1^2+212 t_1 t_2+95  t_2^2 -105 t_1^3 -476 t_1^2 t_2-511 t_1 t_2^2-84 t_2^3. 
\hspace*{12mm}
\diamondsuit
$$

\paragraph{\textbf{Algebraic statistics.}}
Here the adjoint appears when studying the moments of the uniform probability distribution on a polytope.
For a probability distribution $\mu$ on $\R^n$ and an index tuple $I = (i_1, i_2, \ldots, i_n) \in \Z_{\geq 0}^n$, the \emph{$I$-th moment of $\mu$} is
$$
m_I(\mu) := \int_{\R^n} w_1^{i_1} w_2^{i_2} \ldots w_n^{i_n} d \mu.
$$
If $\mu_P$ is the uniform probability distribution on a polytope $P$, we shortly write
 $m_I(P):=m_I(\mu_P)$.
A suitably normalized generating function over all moments of a fixed simplicial polytope $P$ in $\R^n$ with $d$ vertices $v_1, v_2, \ldots, v_d$ is a rational function whose numerator is Warren's adjoint of $P$~\cite{moments}:
\begin{equation*}
\sum_{I \in \Z_{\geq 0}^n} c_I \, m_I(P) \, t^I = \frac{\mathrm{adj_P(t)}}{\mathrm{vol}(P) \prod \limits_{i=1}^d \ell_{v_i}(t)}, \quad\quad\quad \text{where } c_I := \binom{i_1 + i_2 + \ldots + i_n + n}{i_1, i_2, \ldots, i_n, n}.
\end{equation*}
In~\cite[p. 6]{moments} it was also observed that the adjoint of the simplicial polytope $P$ vanishes on all linear spaces
which do not contain faces of the dual polytope $P^\ast$ but are intersections of hyperplanes spanned by facets of $P^\ast$.
Moreover, the authors conjectured that the adjoint is uniquely characterized by this vanishing property.  In this article, we prove this conjecture. 

\paragraph{\textbf{Unique adjoint hypersurfaces.}}
We generalize Wachspress' construction of the adjoint from polygons to %simple 
polytopes
%(recall that a polytope is simple if and only if its dual polytope is simplicial)
and show that our definition of the adjoint coincides with Warren's definition.
In particular, we answer a question by Wachspress~\cite{Wachs11}:
in 2011 he asked for a geometric construction of the unique adjoint associated to a tesseract (i.e., a four-dimensional hypercube); see Example~\ref{ex:tess}.

To avoid dealing with degenerate situations such as parallel faces, we typically work in complex projective $n$-space $\PP^n$.
With a polytope $P$ in $\PP^n$ we denote the collection of the Zariski closures of the faces of a convex polytope in $\R^n$.
If the polytope is full-dimensional, 
the union of these projective subspaces is a hyperplane arrangement $\HP$. 
We study the \emph{residual arrangement} $\RP$ of $P$, which consists of all linear spaces
that are intersections of hyperplanes in $\HP$ and do not contain any face of $P$.

\begin{example}\label{firstex}
If $P$ is a polygon in the plane with $d$ edges, the residual arrangement $\RP$ consists of $\binom{d}{2}-d$ points. 
These are exactly the intersection points of the extended edges outside of the polygon $P$, as studied by Wachspress.

If $P$ is a triangular prism (i.e. the second polytope in Table~\ref{tab:polytopes}),
the planes spanned by the two triangle facets of $P$ intersect in a line
and the planes spanned by the three quadrangle facets intersect in a point. 
The residual arrangement consists of this line and this point.

If $P$ is a cube, each pair of planes spanned by opposite facets of the cube intersects in a line. 
Hence, the residual arrangement $\RP$ consists of three lines.
If $P$ is a regular cube, the three lines lie in a common plane (the ``plane at infinity'').
If the cube $P$ is perturbed, the three lines are skew.
\hfill$\diamondsuit$
\end{example}

The concept of residual arrangements allows us to generalize Wachspress' adjoint. 
To be more precise we prove the following assertion in Section~\ref{sec:adjoint} (see Theorem~\ref{thm:adjoint}).

\begin{theorem}
\label{thm:adjointIntro}
Let $P$ be a full-dimensional polytope in $\PP^n$ with $d$ facets.
If the hyperplane arrangement $\HP$ is simple (i.e. through any point in $\PP^n$ pass at most $n$ hyperplanes in $\HP$),
there is a unique hypersurface $A_P$ in $\PP^n$ of degree $d-n-1$
which vanishes along the residual arrangement $\RP$.
\end{theorem}

If $\HP$ is simple, we call $A_P$ the \emph{adjoint} of the polytope $P$.
Another way of stating this theorem is:
for a polytope $P$ with a simple hyperplane arrangement $\HP$, the ideal sheaf $\mathcal{I}_{\RP}$ of the residual arrangement twisted by $d-n-1$ has a unique global section up to scaling with constants,
i.e. $h^0 (\PP^n, \mathcal{I}_{\RP}(d-n-1)) = 1$.

\begin{example}
\label{ex:firstAdjoints}
If $P$ is a polygon in the plane, then the adjoint $A_P$ is exactly the adjoint curve described by Wachspress.
If $P$ is a triangular prism, then the adjoint $A_P$ is the plane spanned by the line and the point in the residual arrangement $\RP$.

If $P$ is a perturbed cube such that the three lines in the residual arrangement $\RP$ are skew, then the adjoint $A_P$ is the unique quadric passing through the three lines.
If the cube $P$ is regular and the three lines lie in a common plane, then the adjoint $A_P$ degenerates to that plane doubled.
However, in this case the plane arrangement $\HP$ is not simple and %the adjoint is not unique:
there is a three-dimensional family of quadrics passing through the three lines, 
as such a quadric consists of the common plane of the three lines together with any other plane.
\hfill$\diamondsuit$
\end{example}

To compare Warren's adjoint with our definition, we recall Warren's vanishing observation stated above:
his adjoint $\mathrm{adj}_P$ of a polytope $P$ vanishes on the codimension two part of the residual arrangement $\mathcal{R}_{P^\ast}$ of the dual polytope $P^\ast$~\cite[Thm.~5]{War96}.
We generalize this assertion to the whole residual arrangement (for a proof, see Section~\ref{sec:adjoint}).

\begin{proposition}
\label{prop:WarrenCompare}
For a full-dimensional convex polytope $P$ in $\R^n$, Warren's adjoint $\mathrm{adj}_P$ vanishes along the residual arrangement $\mathcal{R}_{P^\ast}$ of the dual polytope $P^\ast$.
In particular, for a simple hyperplane arrangement $\mathcal{H}_{P^\ast}$, the zero locus of the adjoint polynomial $\mathrm{adj}_P$ is the adjoint hypersurface~$A_{P^\ast}$.
\end{proposition}

This shows that, for every convex polytope in $\R^n$ with $d$ facets, there is a hypersurface of degree $d-n-1$ passing through the residual arrangement (namely the zero locus of $\mathrm{adj}_{P^\ast}$).
If the hyperplane arrangement $\HP$ in $\PP^n$ is simple, this hypersurface is unique due to Theorem~\ref{thm:adjointIntro}.
Otherwise, this hypersurface might not be unique as Example~\ref{ex:firstAdjoints} demonstrates.
Nevertheless, the adjoint hypersurface is well-defined, independently of the simplicity of the hyperplane arrangement:
we can deduce from
Proposition~\ref{prop:WarrenCompare}
that the adjoint hypersurface of a polytope $P$ with a non-simple hyperplane arrangement
is the unique limit of the adjoint hypersurfaces of all perturbations of $P$ with a simple hyperplane arrangement (for a proof, see Section~\ref{sec:adjoint}).

\begin{corollary}
\label{cor:limit}
Let $P$ be a full-dimensional polytope in $\PP^n$ with $d$ facets.
If $P_t$ and $Q_t$ are continuous families of polytopes with $d$ facets and simple hyperplane arrangements for $t \in (0,1)$
such that $P = \lim \limits_{t \to 0} P_t$ and $P = \lim \limits_{t \to 0} Q_t$,
then the limits $\lim \limits_{t \to 0} A_{P_t}$ and $\lim \limits_{t \to 0} A_{Q_t}$ of their adjoint hypersurfaces coincide.
\end{corollary}

\begin{remark}
The unique limit in Corollary~\ref{cor:limit} is the adjoint hypersurface of $P$, independently of the simplicity of the hyperplane arrangement $\HP$.
In a real affine chart where $P$ is convex, this is the (possibly non-reduced) zero locus of the adjoint polynomial $\mathrm{adj}_{P^\ast}$.
 \hfill$\diamondsuit$
\end{remark}

\paragraph{\textbf{The Wachspress coordinate map.}}
The study of the map defined by the Wachspress coordinates \eqref{eq:barycentric}, particularly in dimensions two and three, was proposed by Garcia and Sottile~\cite{GarciaSottile}
and Floater and Lai~\cite{FloaterLai}. 
The two-dimensional case was thoroughly analyzed by Irving and Schenck~\cite{IS14}.
We extend their results to three and higher dimensions.
In particular, we will see that already the three-dimensional case is considerably more complicated than the case of polygons. 

The dual polytope $P^\ast$ of a simple polytope $P$ in $\R^n$ is simplicial, so the facet $F_u$ of $P^\ast$ corresponding to a vertex $u$ of $P$ is a simplex.
Thus, the Wachspress coordinates~\eqref{eq:barycentric} reduce in this case to
\begin{equation}
\label{eq:barycentricSimple}
	\forall u \in V(P) : \quad \beta_u(t) := \frac{\mathrm{vol}(F_u) \cdot \prod \limits_{F \in \mathcal{F}(P) : \, u \notin F} \ell_{v_F}(t) }{\mathrm{adj}_{P^\ast}(t)}.
\end{equation}
The linear form $\ell_{v_F}$ simply is the defining equation of the hyperplane spanned by the facet~$F$.
As the polytope $P$ is simple, the numerator of each rational function in~\eqref{eq:barycentricSimple} has degree $d-n$, where $d$ denotes the number of facets of $P$.
Hence, these numerators define a rational map
\begin{equation}
\label{eq:wachspressMap}
\begin{split}
	\omega_P: \PP^n &\,\dashrightarrow \PP^{N-1},\\
	t &\longmapsto \left( \prod \limits_{F \in \mathcal{F}(P) : \, u \notin F} \ell_F(t) \right)_{u \in V(P)},
\end{split}
\end{equation}
where $N$ denotes the number of vertices of $P$ and $\ell_F$ is a homogeneous linear equation defining the projective closure of the hyperplane spanned by the facet $F$.
We call $\omega_P$ the \emph{Wachspress map} of the simple polytope $P$, and call the coordinates of $\omega_P$ enumerated by the vertices $u\in V(P)$ the \emph{Wachspress coordinates}.
The Wachspress map is not defined everywhere:
if we assume the hyperplane arrangement $\HP$ to be simple, it turns out that $\omega_P$ is undefined exactly along the residual arrangement of $P$ (for a proof see Section~\ref{sec:barycentric}).

\begin{theorem}
\label{thm:baseLocus}
For a full-dimensional polytope $P$ in $\PP^n$ with  
a simple hyperplane arrangement $\HP$,
the base locus of the Wachspress map $\omega_P$ is the residual arrangement $\RP$.
\end{theorem}

This assertion shows that the $N$ homogeneous forms in~\eqref{eq:wachspressMap} live in the space
$$\Omega_P := H^0(\PP^n, \mathcal{I}_{\RP}(d-n)).$$
Evaluating the homogeneous form corresponding to a vertex $u \in V(P)$ at a vertex $v \in V(P)$ yields zero if and only if $u \neq v$.
Hence, the $N$ homogeneous forms in~\eqref{eq:wachspressMap}  are linearly independent in $\Omega_P$.
In fact, they form a basis of $\Omega_P$, as we show in Section~\ref{sec:adjoint} (see Theorem~\ref{thm:wachspressCoords}).

\begin{theorem}
\label{thm:dimOmega}
For a full-dimensional polytope $P$ in $\PP^n$ with  
a simple hyperplane arrangement $\HP$,
the dimension of $\Omega_P$ equals the number of vertices of $P$.
\end{theorem}

Hence, the Wachspress map is in fact of the form $\omega_P: \PP^n \dashrightarrow \PP(\Omega_P^\ast) \cong \PP^{N-1}$.
Its image $W_P := \overline{\omega_P(\PP^n)}$ is the \emph{Wachspress variety} of the polytope $P$.
The projective span $\mathbb{V}_P$ of the image of the unique adjoint $A_P$ under the Wachspress map is a projective subspace of $\PP(\Omega_P^\ast)$. 
We denote the linear projection from this subspace by $\rho_P$.
In Section~\ref{sec:barycentric}, we show that this projection restricted to the Wachspress variety is the inverse of the Wachspress map.

\begin{theorem}
\label{thm:inverseOfWachspress}
Let $P$ be a full-dimensional polytope in $\PP^n$ with $N$ vertices and a simple hyperplane arrangement $\HP$.
The dimension of $\mathbb{V}_P = \mathrm{span}\lbrace \omega_P(A_P)  \rbrace \subset \PP(\Omega_P^\ast)$ is $N-n-2$. 
The projection $\rho_P: \PP(\Omega_P^\ast) \dashrightarrow \PP^n$ from $\mathbb{V}_P$ restricted to the Wachspress variety $W_P$ is the inverse of the Wachspress map $\omega_P$.
\end{theorem}

As the Wachspress map $\omega_P$ is rational, we consider the blow-up $\pi_P: X_P \to \PP^n$ of $\PP^n$ along the base locus of $\omega_P$ (i.e., along the residual arrangement $\RP$ according to Theorem~\ref{thm:baseLocus}).
Thus $X_P$ is the closure of the graph of the Wachspress map $\omega_P$.
In this way, the projection of $X_P \subset \PP^n \times \PP(\Omega_P^\ast)$ onto the second factor yields a lifting $\tilde{\omega}_P: X_P \to W_P$ of the Wachspress map.
We summarize our maps in the commutative diagram in Figure~\ref{fig:diagram}
(we introduce the blowup $\pi_P^s: X_P^s \to \PP^n$ after Proposition~\ref{prop:wachspressVariety}):
\begin{figure}
\centering
\begin{normalsize}
\begin{tikzcd}
&
\PP^n \times \PP(\Omega_P^\ast)
\arrow[dr] \\
X_P^s \arrow[r]
\arrow[dr, "\pi_P^s" below]
&
X_P \arrow[d, "\pi_P" left]
\arrow[u, hook]
\arrow[dr, "\tilde{\omega}_P" {above, near start}] &[16mm]
\PP(\Omega_P^\ast) \arrow[dl, dashed, "\;\rho_P" {above, near start}]
\\
&
\PP^n \arrow[r, dashed, "\omega_P" below] &[16mm]
W_P  \arrow[u, hook]
\end{tikzcd}
\end{normalsize}
\caption{The Wachspress map $\omega_P$ with its inverse and associated blowups.}
\label{fig:diagram}
\end{figure}
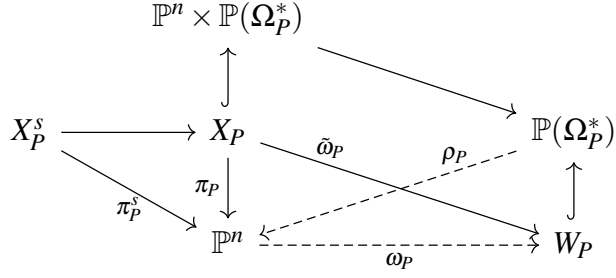

\begin{example}[see \cite{IS14}]\label{ex:surface} 
When $P$ is a $d$-gon in $\PP^2$, the Wachspress coordinates span a  $d$-dimensional vector space $\Omega_P$ and the residual arrangement $\RP$ consists of $\binom{d}{2}-d$ points,
so $\pi_P:X_P\to \PP^2$ is the blowup of these points and the morphism $\tilde{\omega}_P$ maps $X_P$ into $ \PP^{d-1}.$
The image $W_P$ is a surface of degree $\binom{d-2}{2}+1$.  The adjoint curve $A_P$ is, for a general $P$, a smooth curve of degree $d-3$ whose image under $\omega_P$ has degree $\binom{d-3}{2}$.
\hfill$\diamondsuit$
\end{example}

We extend these results by Irving and Schenck to polytopes in $3$-space.

\begin{example} \label{ex:prismCube}
If $P$ is a tetrahedron, the residual arrangment $\RP$ is the empty set, the adjoint surface is empty and the Wachspress map is the identity map on $\PP^3$. 

If $P$ is a triangular prism, the adjoint surface is the unique plane spanned by the line and the point of $\RP$ (see Example  \ref{ex:firstAdjoints}).
The Wachspress map is given by $(\ell_{\triangle_1}:\ell_{\triangle_2}) \otimes (\ell_{\square_1}:\ell_{\square_2}:\ell_{\square_3})$,
where the $\ell_{\triangle_i}$ and $\ell_{\square_j}$ denote the defining equations of the planes spanned by triangular and quadrangular facets, respectively.
The image of the Wachspress map is the Segre embedding of $\PP^1\times \PP^2$ in $\PP^5$.  The restriction of the Wachspress map to the adjoint plane is the projection from the point, so the image of the adjoint plane is a line.

If $P$ is a cube, perturbed so that the plane arrangement $\HP$ is simple, the adjoint surface is the unique quadric surface that contains the three lines of $\RP$ (see Example \ref {ex:firstAdjoints}).
The Wachspress map is given by $(\ell_1:\ell_6) \otimes (\ell_2 : \ell_5) \otimes (\ell_3 : \ell_4)$,
where $\ell_i$ and $\ell_{7-i}$ denote the defining equations of two planes spanned by opposite facets of the perturbed cube $P$.
The image of the Wachspress map is the Segre embedding of $\PP^1\times \PP^1\times \PP^1$ in $\PP^7$.  The restriction of the Wachspress map to the adjoint surface contracts the ruling of lines that meet all three lines  in $\RP$, so the image of the surface is a twisted cubic curve.
\hfill$\diamondsuit$
\end{example}

In fact, for all other three-dimensional polytopes defined by simple plane arrangements,
the image of the adjoint surface $A_P$ under the Wachspress map $\omega_P$ is a surface.
We prove the following extension of the results by Irving and Schenck in Section~\ref{sec:barycentric}.

\begin{proposition}
\label{prop:wachspressVariety}
Let $P$ be a full-dimensional polytope in $\PP^3$  with $d$ facets and a simple plane arrangement $\HP$.
Let $a$ be the number of isolated points in the residual arrangement~$\RP$,
let $b$ be the number of double points (i.e. points where exactly two lines in $\RP$ meet),
and let $c$ be the number of triple points (where three lines in $\RP$ meet).
   The rational variety $W_P\subset \PP^{2d-5}$ has degree 
 $$
{\rm deg}\; W_P\!=\!2b+4c-a-\frac{1}{2}(d-3)(d^2\!-11d+26)\!=\!b+2c+1-\frac{1}{6}(d-3)(d-4)(d-11) $$  
   and sectional genus
  $g(W_P)=b+2c +1+\frac{1}{2}(d-3)(d-6).$
 
  The image $\bar A_P := \overline{\omega_P(A_P)}  \subset W_P$ of the adjoint surface  $A_P$
is a curve if and only if $P$ is a triangular prism or a cube (see Example~\ref{ex:prismCube}).
If $P$ is not a triangular prism, cube, or tetrahedron, the image $\bar A_P$ is a surface of degree
    $${\rm deg}\; \bar A_P=2b+4c-a-\frac{1}{2}(d-3)(d-4)(d-6)=b+2c+1-\frac{1}{6}(d-3)(d^2-12d+38)$$
and sectional genus
  $g(\bar A_P)=b+2c +1-\frac{1}{2}(d-3)(d-4).$
  
  The image $\bar D\subset W_P$ of any surface in $X_P$ linearly equivalent to the strict transform of $\HP$ has degree
  $${\rm deg}\; \bar D= \sum \limits_{F \in \mathcal{F}(P)}  \binom{|V(F)|-2}{2}+d=4b+9c-3a-\frac{1}{2}(d-3)(3d^2-30d+64),$$
while the sectional genus is   
$$g(\bar D)=\!\! \sum \limits_{F \in \mathcal{F}(P)} \binom{|V(F)|-3}{2}+2d-5=1+4b+9c-3a-\frac{1}{2}(d-3)(3d^2-30d+68 ).$$
  \end{proposition}

As the variety $X_P$ is typically not smooth (see Lemma~\ref{exceptional}),
we blow up a little further.
The map $\pi_P^s$ (see Figure~\ref{fig:diagram}) also blows up $\PP^n$ along $\RP$, but decomposed into a sequence of blowups each along a smooth locus.

 To be precise, we consider a decomposition $\RP=\bigcup_{c=2}^n \RPc$, where $ \RPc$ is the union of the irreducible components of $\RP$ of codimension $c$ that are not contained in any component of codimension $c-1$ in $\RP$. Similarly, we decompose the singular locus ${\rm Sing}\RP= \bigcup_{c=3}^n \RPcs$, where $\RPcs$ is the union of codimension $c$ 
 linear spaces in the singular locus of $\RP$ that lie on $c$ hyperplanes in $\HP$.  
  We first blow up all isolated points and $0$-dimensional singular points in~$\RP$, i.e ${\mathcal R}_{P,n}\cup {\mathcal R}_{P,s,n}$.    Next, we blow up the strict transform of the lines ${\mathcal R}_{P,n-1}\cup {\mathcal R}_{P,s,n-1}$, and so on.   The procedure ends with the blowup of the strict transform of the codimension three locus ${\mathcal R}_{P,3}\cup {\mathcal R}_{P,s,3}$ and finally of the strict transform of the codimension two locus~${\mathcal R}_{P,2}$.  The map $\pi_P^s: X_P^s \to \PP^n$  is the composition of these blowups, which ensures that $X_P^s$ is smooth.

\paragraph{\textbf{Polytopal hypersurfaces.}}
The blow-up $\pi_P^s$ allows us to relate adjoints of polytopes to the classical notion of adjoints of hypersurfaces and divisors.
For a hypersurface  $Y\subset\PP^n$ of degree $d$, an \emph{adjoint} is a hypersurface of degree $d-n-1$ that contains every component of codimension $c$ of the singular locus of $Y$ where $Y$ has multiplicity $c$.
More generally, for a smooth variety $X$ and a smooth subvariety $Y \subset X$ of codimension one, an \emph{adjoint to $Y$ in $X$} is a codimension-one subvariety of $X$ whose class in the Chow ring of $X$ is $K_X + [Y]$ (i.e., the canonical class of $X$ plus the class of $Y$).
This adjoint appears in the \emph{adjunction formula}: the canonical class $K_Y$ of $Y$ is $K_X + [Y]$ restricted to $Y$.  
Therefore adjoints play an important role in the classification of varieties (for a nice survey, see~\cite{BS}).

If the strict transform of $\HP$ in $X_P^s$ is linearly equivalent to a smooth divisor~$\tilde D$,
then we can compare adjoints of $\tilde D$ in $X_P^s$ to the adjoint $A_P$ of the polytope $P$.
More formally, we study the linear system $\Gamma_P$ of divisors $D$ in $\PP^n$ which have degree $d$ and vanish with multiplicity $c$ along $\RPc$.
We call the hypersurfaces in $\Gamma_P$ the \emph{polytopal hypersurfaces} associated to $P$.
Note that the hyperplane arrangement $\HP$ is a polytopal hypersurface.
%We call the strict transforms $\tilde D$ in $X_P^s$ of the divisors $D$ in the linear system~$\Gamma_P$ the \emph{polytopal hypersurfaces} associated to $P$.
In Section~\ref{sec:deformations}, we show that each smooth strict transform $\tilde D$ in $X_P^s$ of a polytopal hypersurface $D \in \Gamma_P$ associated to a polytope $P$ with a simple hyperplane arrangement $\HP$ has a unique adjoint in $X_P^s$. 
The hypersurface $A_P$ of the polytope $P$ is an adjoint for any $D\in \Gamma_P$ that is smooth outside $\RP$, and its strict transform is the main component of the unique adjoint of $\tilde D$ in $X_P^s$ (for a proof see Section~\ref{sec:deformations}).

\begin{proposition}
\label{prop:relateAdjoints}
Let $P$ be a full-dimensional polytope in $\PP^n$ with a simple hyperplane arrangement $\HP$.
Moreover, let $\tilde{A}_P$ be the strict transform of the adjoint $A_P$ under the blow-up $\pi_P^s: X_P^s \to \PP^n$, and let $[\tilde{A}_P]$ be its  class in the Chow ring of~$X_P^s$.
%Then $$K_{X_P} + [\tHP] = [\tilde{A}_P]+[E_P],$$
Every strict transform $\tilde D$ in $X_P^s$ of a polytopal hypersurface $D \in \Gamma_P$ satisfies $$K_{X_P^s} + [\tilde{D}] = [\tilde{A}_P]+[E_P],$$
where  $E_P$ is an effective divisor whose image $\pi_P^s(E_P)$ is contained in the singular locus of~$\RP$.
If $\tilde D$ is smooth, it has a unique adjoint in $X_P^s$, and thus a unique canonical divisor: namely $\tilde{D} \cap (\tilde{A}_P+E_P)$.
\end{proposition}

As a smooth strict transform $\tilde D$ of a polytopal hypersurface $D$ has a unique canonical divisor, it is interesting to determine its birational type.
In Section~\ref{sec:deformations},  we identify all combinatorial types of simple polytopes in dimensions two and three which admit smooth strict transforms of their polytopal hypersurfaces and we describe their birational types.

\begin{proposition}
\label{prop:TypesInDim2}
Let $d \in \mathbb{Z}, d \geq 3$, and
let $P$ be a general polygon in $\PP^2$ with $d$ edges.
There is a polygonal curve $D \in \Gamma_P$ with a smooth strict transform $\tilde D$ in $X_P^s$ if and only if $d \leq 6$.
These curves are elliptic.
\end{proposition}

\begin{theorem}
\label{thm:TypesInDim3}
Let $\mathcal{C}$ be a combinatorial type of simple three-dimensional polytopes, and
let $P$ be a general polytope in $\PP^3$ of type $\mathcal{C}$. 
There is a polytopal surface $D \in \Gamma_P$ with a smooth strict transform $\tilde D$ in $X_P^s$ if and only if
$\mathcal{C}$ is one of the nine combinatorial types in Table~\ref{tab:polytopes}.

In that case, the general $D \in \Gamma_P$ 
is either an elliptic surface (if $P$ has a hexagonal facet)
or a K3-surface (if $P$ has at most pentagonal facets).

\end{theorem}

We describe all combinatorial types of simple three-dimensional polytopes with smooth strict transforms of their polytopal surfaces in detail in Subsections~\ref{ssec:prisms} and \ref{ssec:Nprisms}.
Some of their key properties are summarized in Table~\ref{tab:polytopes}.

Since Wachspress asked for a geometric construction of the unique adjoint associated to a tesseract (i.e., a four-dimensional hypercube) in 2011~\cite{Wachs11},
we discuss this higher dimensional example in detail.

\begin{example}
\label{ex:tess}
Let $P$ be a hypercube in $\PP^n$, which is perturbed so that its hyperplane arrangement is simple.
The polytope $P$ has $2n$ facets and $2^n$ vertices.
The residual arrangement $\RP$ consists of $n$ subspaces of codimension two,
as every pair of hyperplanes spanned by opposite facets of $P$ intersects in such a subspace.
There is an $(n-2)$-dimensional family of lines in $\PP^n$ passing through all $n$ subspaces in $\RP$.
The union of these lines forms the adjoint hypersurface $A_P$ of degree $n-1$.

Similar to the cube case (cf. Example~\ref{ex:prismCube}), we see that the Wachspress variety $W_P$ is the Segre embedding of $\left( \PP^1 \right)^n$ in $\PP^{2^n-1}$.
The restriction of the Wachspress map to the adjoint hypersurface contracts each line in the ruling described above.
Explicit calculations for $n \leq 5$ suggest that
the image of $A_P$ under the Wachspress map is a rational variety of dimension $n-2$ projectively equivalent to the image of a general $(n-2)$-dimensional subspace of $\PP^n$.
In fact, we computed for small $n$ that
the intersection of $W_P$ with two general hyperplanes which contain the image of $A_P$ has two components:
both are rational and isomorphic to the image of $\PP^{n-2}$ mapped by hypersurfaces of degree $n$ that contain $n$ linear spaces of dimension $n-4$ in general position
(these $n$ linear spaces are the intersection of the $\PP^{n-2}$ with $\RP$).
The intersection of the two components is an anticanonical divisor on each of them.  

The general hyperplane sections of $W_P$ are images of hypersurfaces of degree $n$ that contain $\RP$, while the polytopal hypersurfaces in $\Gamma_P$ have degree $2n$ and are singular along~$\RP$.
So  $P$ has an irreducible polytopal hypersurface and its images in $W_P$ is the intersection with a quadric hypersurface. 
In particular, that image is a Calabi-Yau variety.
\hfill$\diamondsuit$
\end{example}

It is natural to ask if all  our results can be extended to polytopes in dimension four and higher.
Moreover, since many of our results are restricted to polytopes which are defined by simple hyperplane arrangements,
it would be interesting to study more general classes of polytopes.

\begin{problem}
Can the results in Proposition~\ref{prop:wachspressVariety}, Theorem~\ref{thm:TypesInDim3} and Table~\ref{tab:polytopes} 
be extended to polytopes in $\PP^n$ whose hyperplane arrangements are simple, with $n>3$?

In particular, is there a simple description of such polytopes which have smooth strict transforms of their polytopal hypersurfaces?
What are their birational types?
\end{problem}

\begin{problem}
Can our results be extended to polytopes whose hyperplane arrangements are not simple? 

In particular, 
find a general definition of the adjoint hypersurface of a polytope
which does not depend on the simplicity of its hyperplane arrangement
and does not involve limits as in Corollary~\ref{cor:limit}.
\end{problem}

\begin{remark}
\label{rem:cohenMac}
Let $\mathcal{C}$ be one of the combinatorial types in Table~\ref{tab:polytopes}
and let $P$ be a general polytope in $\PP^3$ of type $\mathcal{C}$.
We verified computationally that the rational variety $W_P\subset \PP^{2d-5}$ is arithmetically Cohen Macaulay.
This further extends results by Irving and Schenck~\cite{IS14}, who showed in the case of polygons in the plane that the Wachspress surface $W_P$ is arithmetically Cohen Macaulay.
 \hfill$\diamondsuit$
\end{remark}

\begin{problem}
Does Remark~\ref{rem:cohenMac} hold for \emph{every} combinatorial type of simple three-
dimensional polytopes?
What about simple polytopes in $n$ dimensions, and non-simple polytopes?
\end{problem}

\begin{landscape}
\begin{table}
\caption{Combinatorial types of simple polytopes $P$ in $\PP^3$ with smooth strict transforms $\tilde D$ of polytopal surfaces $D \in \Gamma_P$. \hspace*{-30mm} \newline
\small A $\circ$ incident to $2$ lines in the drawings of $\RP$ means that the lines do \emph{not} intersect in $\PP^3$. $(a,b,c)$ are defined in Prop.~\ref{prop:wachspressVariety}.
}
\label{tab:polytopes}
\begin{footnotesize}
\begin{tabular}{cccccccc}
\Xhline{2\arrayrulewidth}
\rule{0pt}{3ex}
comb. &
facet & 
$\RP$ &
$(a,b,c)$ & 
$W_P$ & 
$\overline{w_P(A_P)}$ & 
$\dim \Gamma_P$ & 
$\overline{w_P(D)}$   
\\
type &
sizes &
&
&
(deg., sec. genus) &
(deg., sec. genus) &&
(deg., sec. genus)
\\
\hline 
\multirow{4}{*}{\includegraphics[width = 0.1\textwidth]{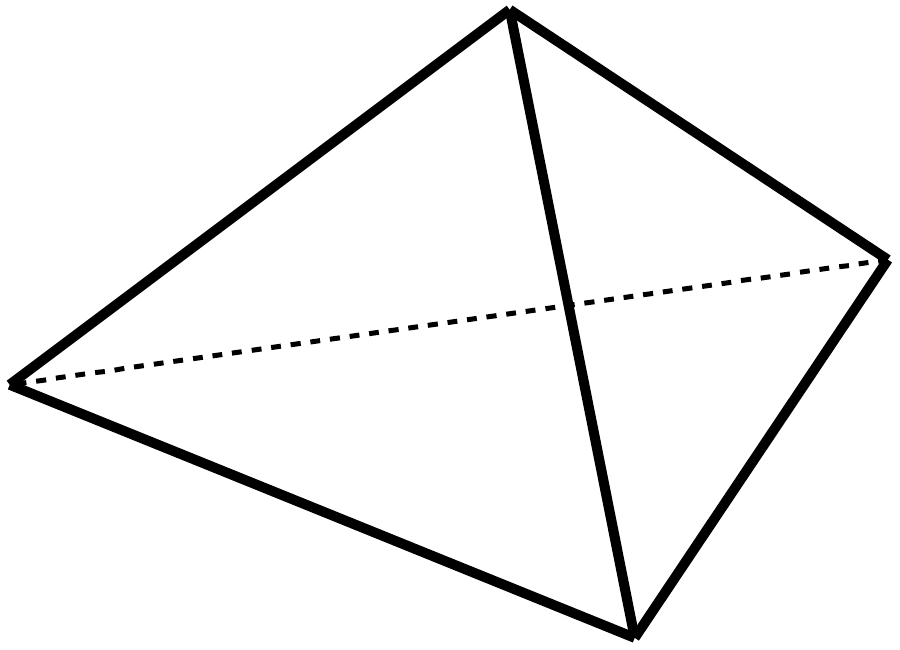}} &
\multirow{4}{*}{$3\,3\,3\,3$}
&&
\multirow{4}{*}{$(0,0,0)$}
&&&
\multirow{4}{*}{$34$}
\\ &&&&
$\PP^3$ &
$0$ &&
minimal K3 
\\ &&&&
$(1,0)$ &&&
(smooth quartic in $\PP^3$)
\\\vspace*{-3mm}
&\\
\multirow{4}{*}{\includegraphics[width = 0.1\textwidth]{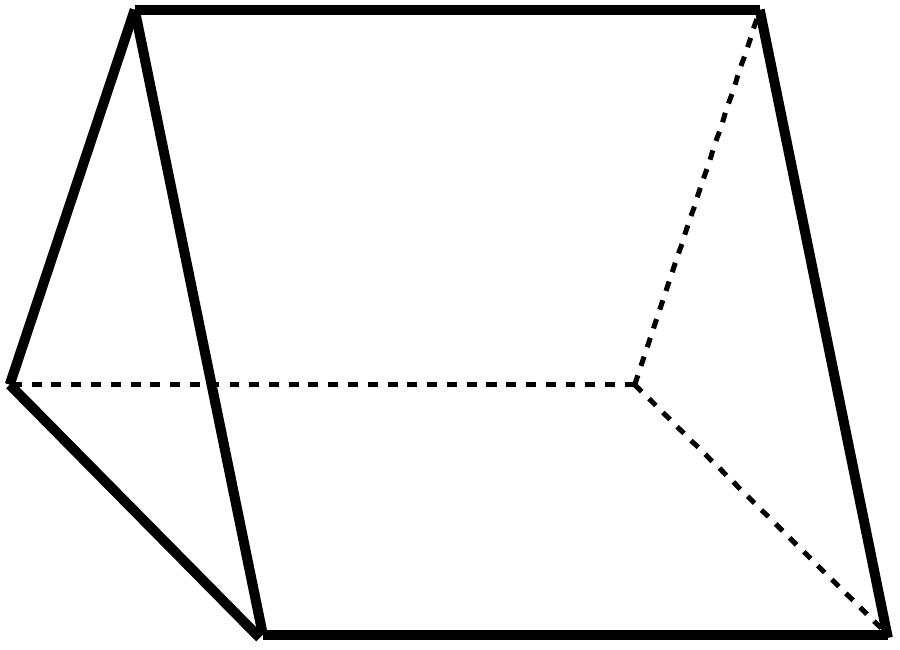}} &
\multirow{4}{*}{$4\,4\,4\,3\,3$} &
\multirow{4}{*}{\includegraphics[width = 0.1\textwidth]{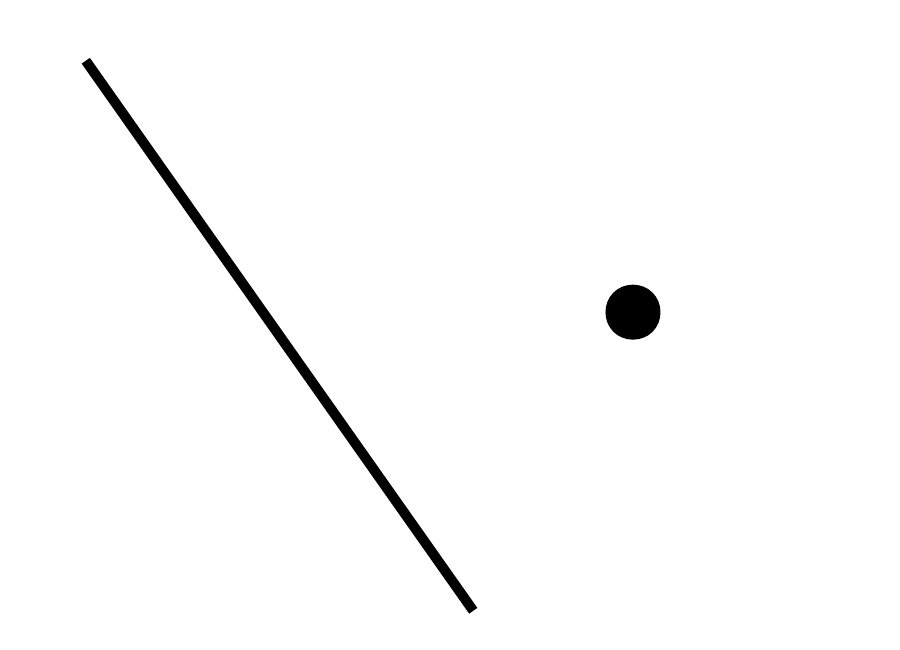}} &
\multirow{4}{*}{$(1,0,0)$}
&&&
\multirow{4}{*}{$23$}
\\ &&&&
$\PP^1 \times \PP^2 \subset \PP^5$ &
line &&
minimal K3
\\ &&&&
$(3,0)$ &&&
$(8,5)$
\\\vspace*{-3mm}
&\\
\multirow{4}{*}{\includegraphics[width = 0.1\textwidth]{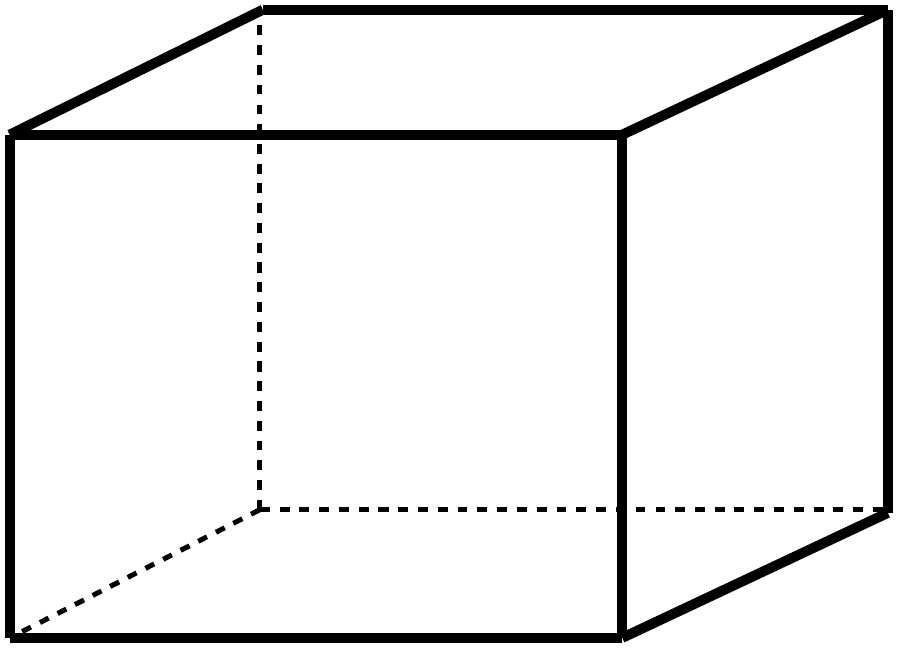}} &
\multirow{4}{*}{$4\,4\,4\,4\,4\,4$} &
\multirow{4}{*}{\includegraphics[width = 0.1\textwidth]{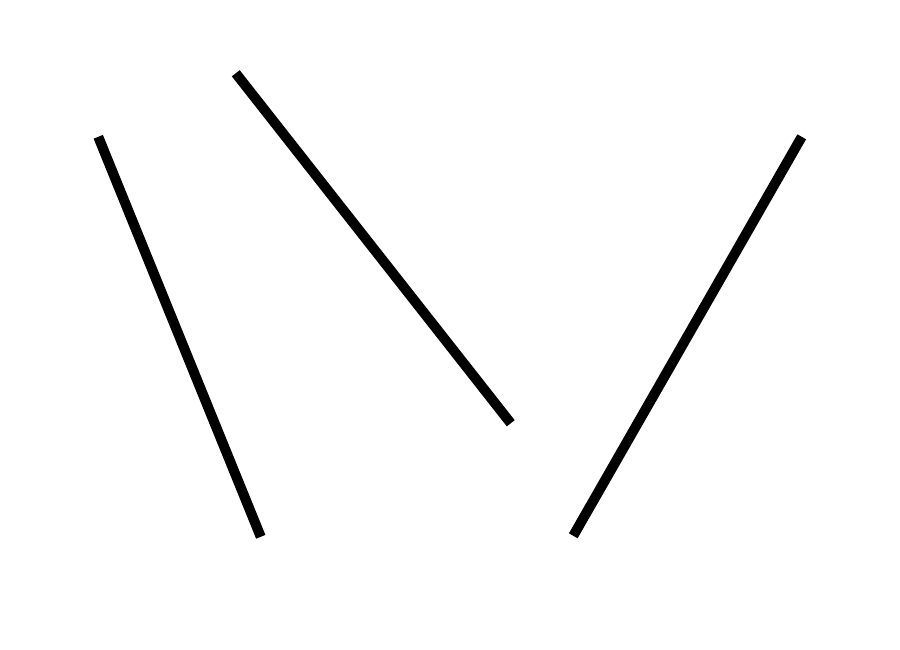}} &
\multirow{4}{*}{$(0,0,0)$}
&&&
\multirow{4}{*}{$26$} 
\\ &&&&
$\PP^1 \times \PP^1 \times \PP^1 \subset \PP^7$ &
twisted cubic curve &&
minimal K3
\\ &&&&
$(6,1)$ &&&
$(12,7)$
\\\vspace*{-3mm}
&\\
\multirow{4}{*}{\includegraphics[width = 0.1\textwidth]{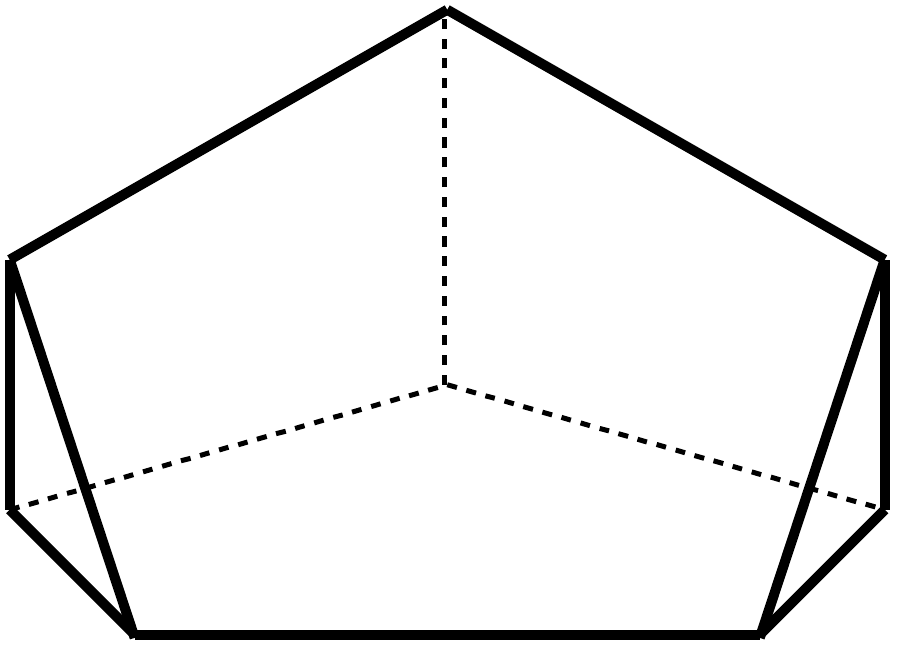}} &
\multirow{4}{*}{$5\,5\,4\,4\,3\,3$} &
\multirow{4}{*}{\includegraphics[width = 0.1\textwidth]{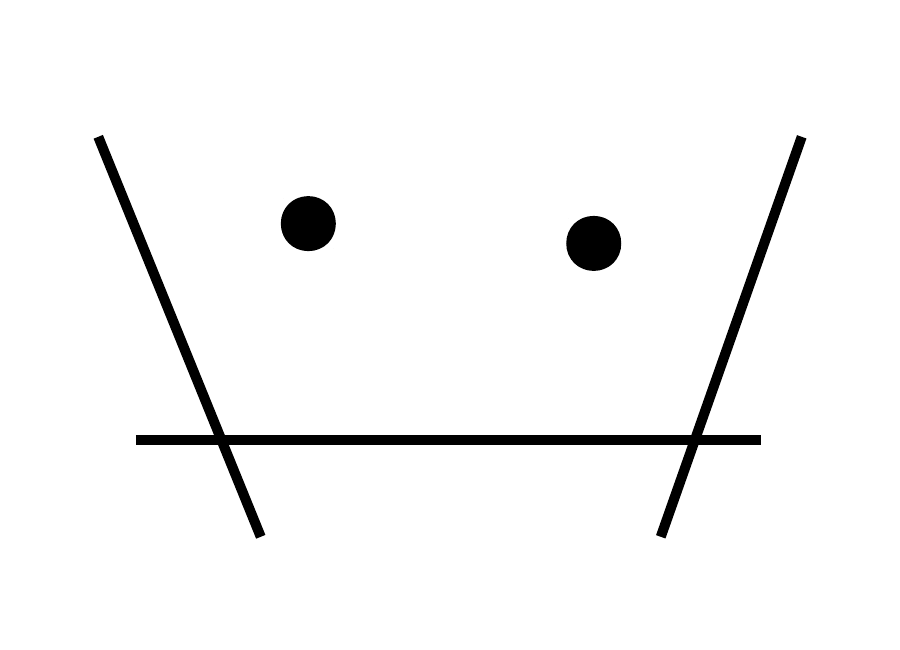}} &
\multirow{4}{*}{$(2,2,0)$}
&&&
\multirow{4}{*}{$17$} 
\\ &&&&
$W_P \subset \PP^7$&
quadric surface &&
{non-minimal K3}
\\ &&&&
$(8,3)$ & 
$(2,0)$ &&
$(14,9)$
\\\vspace*{-3mm}
&\\
\multirow{4}{*}{\includegraphics[width = 0.1\textwidth]{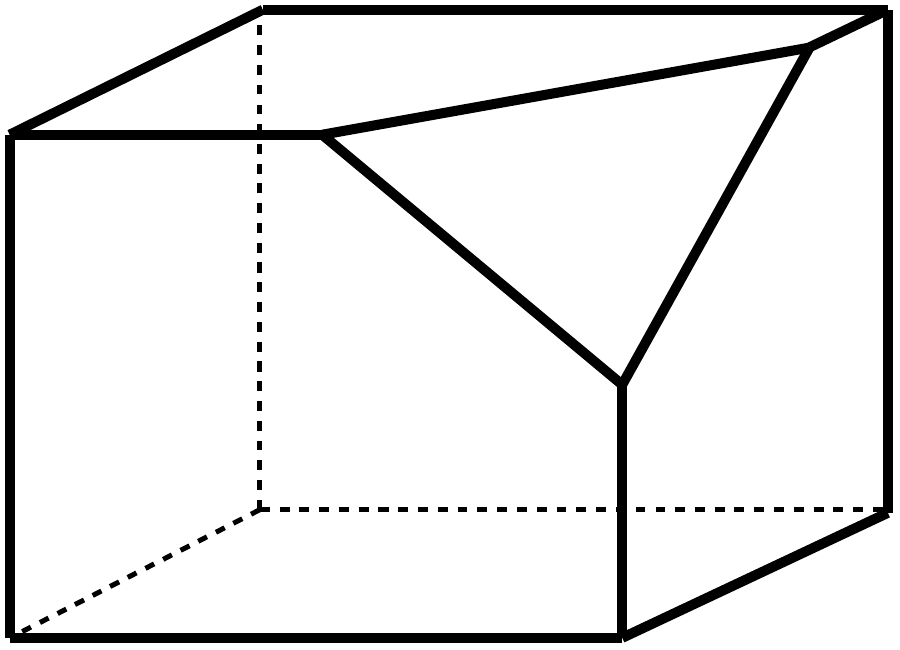}} &
\multirow{4}{*}{$5\,5\,5\,4\,4\,4\,3$} &
\multirow{4}{*}{\includegraphics[width = 0.1\textwidth]{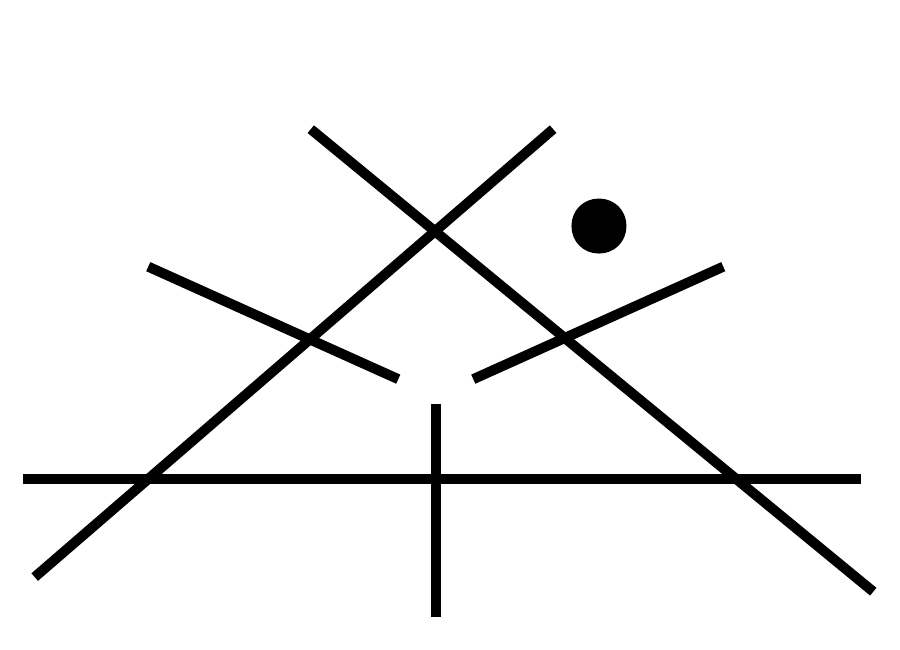}} &
\multirow{4}{*}{$(1,6,0)$}
&&&
\multirow{4}{*}{$7$} 
\\ &&&&
$W_P \subset \PP^9$ & 
del Pezzo surface in $\PP^5$ &&
{non-minimal K3}
\\ &&&&
$(15,9)$ &
$(5,1)$ &&
$(19,12)$
\\\vspace*{-3mm}
&\\
\multirow{4}{*}{\includegraphics[width = 0.1\textwidth]{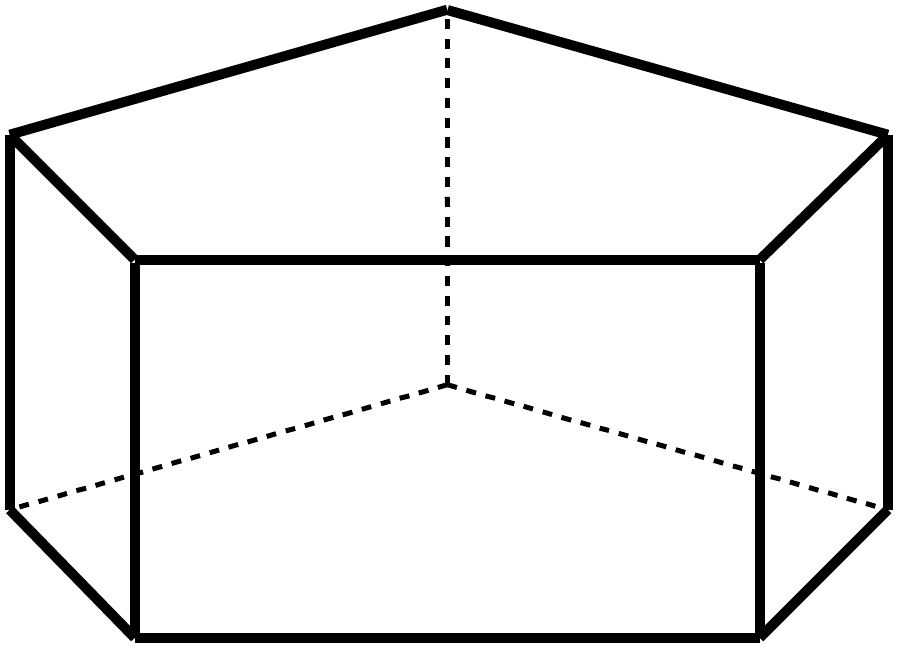}} &
\multirow{4}{*}{$5\,5\,4\,4\,4\,4\,4$} &
\multirow{4}{*}{\includegraphics[width = 0.1\textwidth]{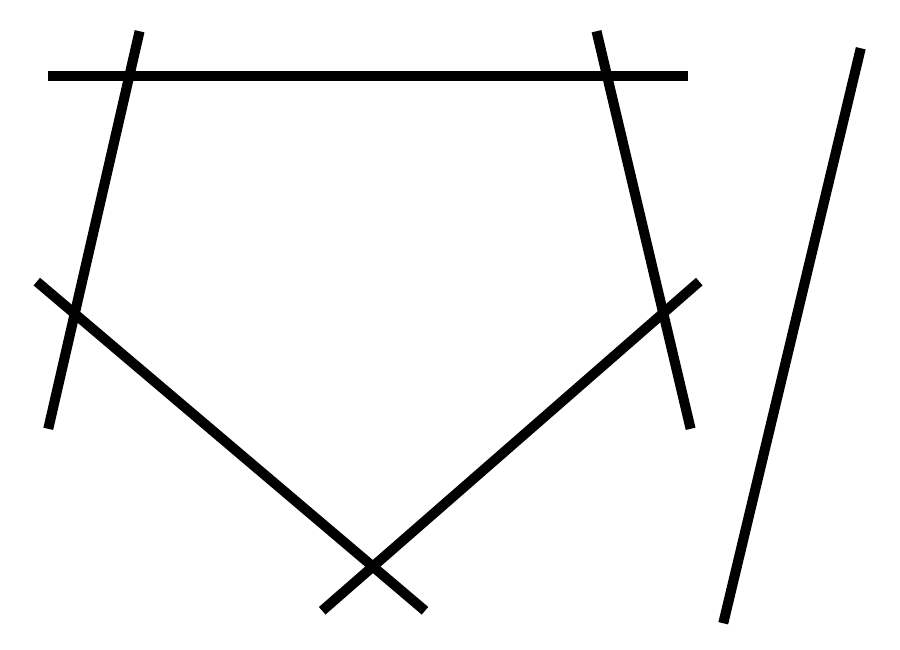}} &
\multirow{4}{*}{$(0,5,0)$}
&&&
\multirow{4}{*}{$12$} 
\\ &&&&
Fano $3$-fold in $\PP^9$ &
rational scroll in $\PP^5$ &&
{non-minimal K3}
\\ &&&&
$(14,8)$ &
$(4,0)$ &&
$(18,11)$
\\\vspace*{-3mm}
&\\
\multirow{4}{*}{\includegraphics[width = 0.1\textwidth]{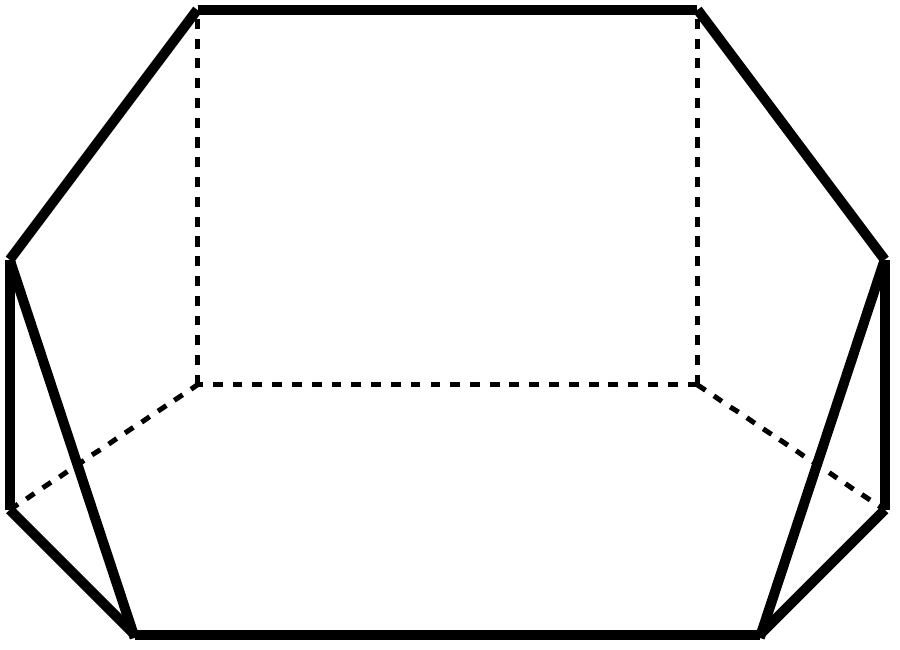}} &
\multirow{4}{*}{$6\,6\,4\,4\,4\,3\,3$} &
\multirow{4}{*}{\includegraphics[width = 0.1\textwidth]{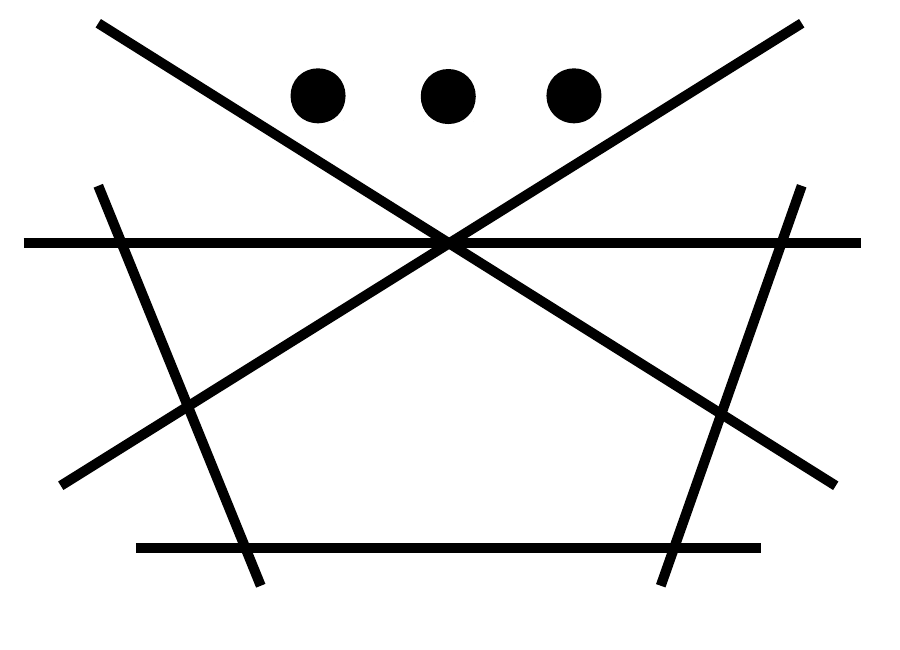}} &
\multirow{4}{*}{$(3,6,1)$}
&&&
\multirow{4}{*}{$4$} 
\\ &&&&
$W_P \subset \PP^9$ &
rational elliptic surface in $\PP^5$ &&
{minimal elliptic}
\\ &&&&
$(17,11)$ & 
$(7,3)$ &&
$(22,15)$
\\\vspace*{-3mm}
&\\
\multirow{4}{*}{\includegraphics[width = 0.1\textwidth]{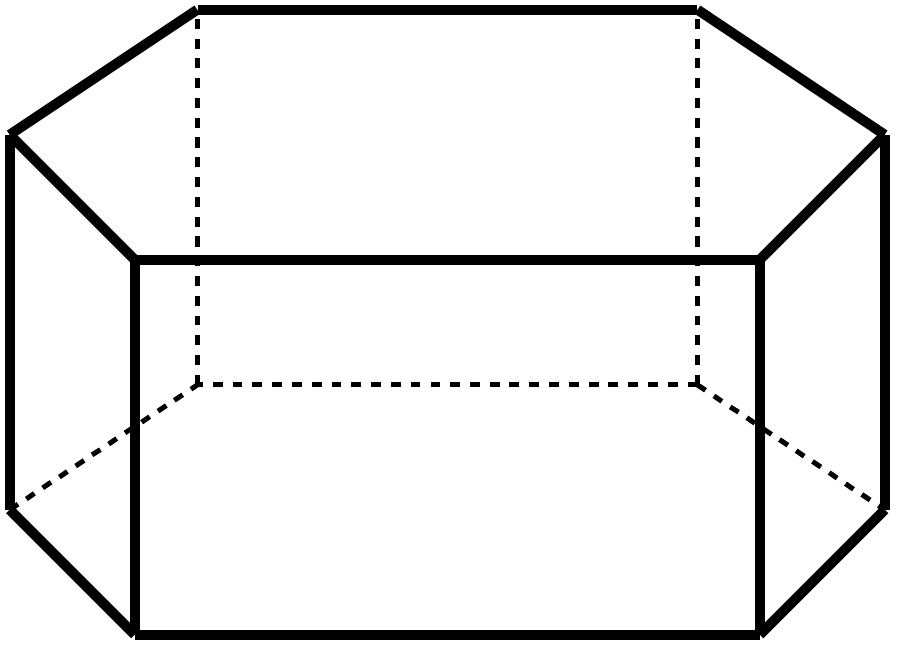}} &
\multirow{4}{*}{$6\,6\,4\,4\,4\,4\,4\,4$} &
\multirow{4}{*}{\includegraphics[width = 0.14\textwidth]{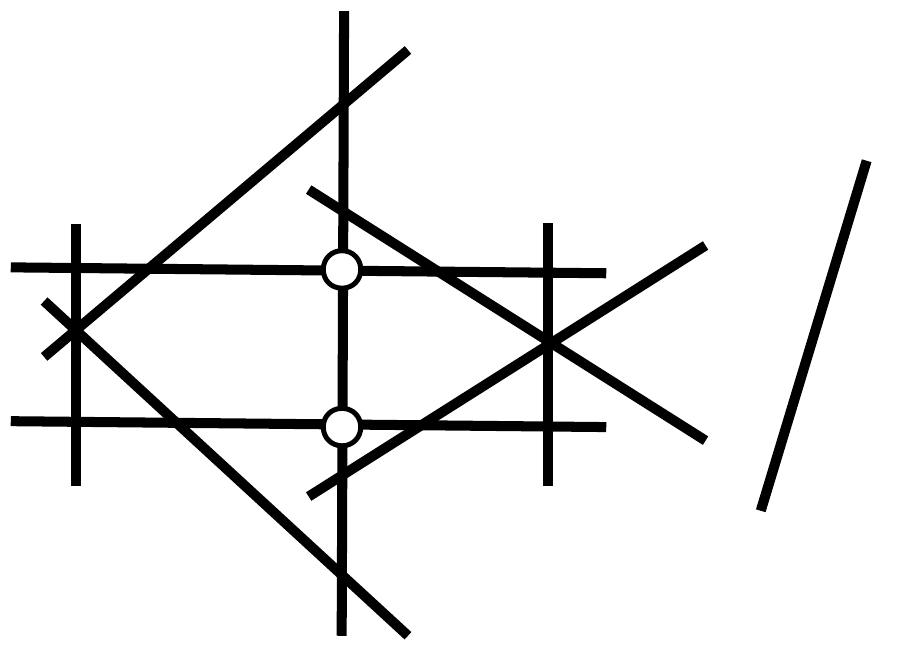}} &
\multirow{4}{*}{$(0,12,2)$}
&&&
\multirow{4}{*}{$3$} 
\\ &&&&
$W_P \subset \PP^{11}$ &
elliptic K3-surface in $\PP^7$ &&
{minimal elliptic}
\\ &&&&
$(27,22)$ &
$(12,7)$ &&
$(26,17)$
\\\vspace*{-3mm}
&\\
\multirow{4}{*}{\includegraphics[width = 0.1\textwidth]{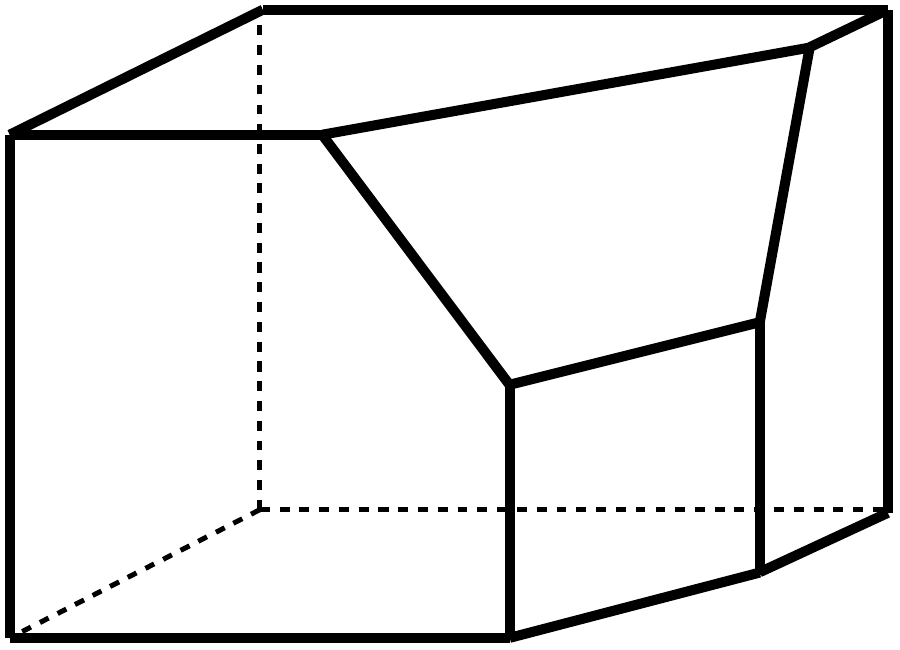}} &
\multirow{4}{*}{$5\,5\,5\,5\,4\,4\,4\,4$} &
\multirow{4}{*}{\includegraphics[width = 0.14\textwidth]{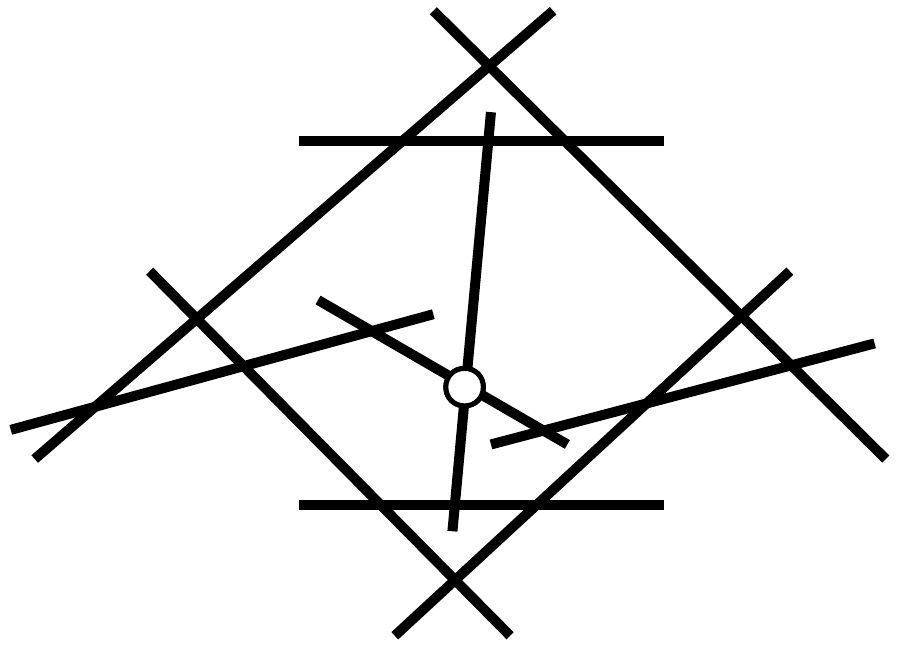}} &
\multirow{4}{*}{$(0,16,0)$}
&&&
\multirow{4}{*}{$1$} &
\\ &&&&
$W_P \subset \PP^{11}$ &
K3-surface in $\PP^7$ &&
{non-minimal K3}
\\ &&&&
$(27,22)$ &
$(12,7)$ &&
$(24,15)$
\\\vspace*{-3mm}
&\\
&\\
\Xhline{2\arrayrulewidth}
\end{tabular}
\end{footnotesize}
\end{table}
\end{landscape}

 \section{The adjoint}
 \label{sec:adjoint}
Our main result in this section is to show that every full-dimensional polytope $P$ defined by a simple hyperplane arrangement has a unique \emph{adjoint} $A_P$.

\begin{theorem}
\label{thm:adjoint}
Let $d, n \in \Z_{>0}$ with $d \geq n+1$.
For every full-dimensional polytope $P$ in~$\PP^n$ with $d$ facets and a simple hyperplane arrangement $\HP$,
there is a unique hypersurface $A_P$ in~$\PP^n$ of degree $d-n-1$ vanishing on $\RP$,
i.e. $h^0 (\PP^n, \mathcal{I}_{\RP}(d-n-1)) = 1$.
Moreover, this hypersurface does not pass through any vertex of $P$,
and $h^1 (\PP^n, \mathcal{I}_{\RP}(d-n-1)) = 0$.
\end{theorem}

We prove this assertion by induction on $n$ and $d$. 
To start the induction, we need to consider the cases $d = n+1$ (i.e. $P$ is a simplex) and $n=1$.
As there are no polytopes with more than two facets in $\PP^1$, the case $n=1$ follows from the case where $P$ is a simplex.

\begin{lemma}
\label{lem:tetrahedronAdjoint}
Theorem~\ref{thm:adjoint} holds if $P$ is a simplex, i.e. if $d = n+1$.
\end{lemma}

\begin{proof}
If $P$ is a simplex, its residual arrangement $\RP$ is empty.  
%We denote by $\iota_{\RP}$ the embedding of this empty set $\RP$ into $\PP^n$.
%The ideal sheaf $\mathcal{I}_{\RP}$ is defined as the kernel of the pullback $(\iota_{\RP})^\ast$.
%In this case, the pushforward of $\mathcal{O}_{\RP}$ to $\PP^n$ is the $0$-sheaf,
So $\mathcal{I}_{\RP} = \mathcal{O}_{\PP^n}$.
This shows that $h^0 (\PP^n, \mathcal{I}_{\RP}) = 1$ and $h^1 (\PP^n, \mathcal{I}_{\RP}) = 0$.
Moreover, the up to scalar, unique nonzero global section of $\mathcal{I}_{\RP}$ is a constant function on $\PP^n$, which does not vanish on any vertex of $P$.
\end{proof}

For the induction step, we assume $d > n+1$. We pick any hyperplane $H$ in the arrangement $\HP$ and consider the polytope $Q \subset \PP^n$ obtained by removing the facet corresponding to $H$ from $P$.
To relate the residual arrangements of $P$ and $Q$, we first state several combinatorial observations: 
\begin{remark}\label{rmk:CO}
\begin{enumerate}
\item \label{enum:facet} The facet $F := P \cap H$ corresponding to $H$ is an $(n-1)$-dimensional polytope.
The faces of $P$ which are not faces of $Q$ are exactly the faces of $F$.
\item \label{enum:newFaces} We denote by $k$ the number of facets of $F$. 
These facets are the intersections of $H$ with $k$ hyperplanes $H_1, H_2, \ldots, H_k$ in $\HP$. 
The $d-k-1$ intersections of $H$ with the other hyperplanes in $\HP$ are contained in the residual arrangement $\RP$ (and not in $\mathcal{R}_Q$).
\item \label{enum:cutOff} The hyperplane $H$ \emph{cuts off} faces of $Q$, which are no longer faces of $P$.
This could be a single vertex, or up to several faces of codimension at least two.
We denote the union of all subspaces corresponding to cut-off-faces by $\mathcal{C}$.
All subspaces in $\mathcal{C}$ are intersections of the $k$ hyperplanes $H_1, H_2, \ldots, H_k$.
Note that the set of cut-off faces has the structure of a polyhedral complex whose $1$-skeleton is connected, unless it is a single vertex.
\item \label{enum:codim2} The subspaces of codimension two in $\RP$ are exactly the subspaces of codimension two in $\mathcal{R}_Q$
together with the $d-k-1$ subspaces described in item~\ref{enum:newFaces} of this remark
and the subspaces of codimension two in the cut-off part $\mathcal{C}$.
In particular, every codimension-two subspace in $\mathcal{R}_Q$ is contained in $\RP$.
\item \label{enum:codimGeneral} More generally, for $2 < c \leq n$, there can be subspaces of codimension $c$ which are contained in either one of the two residual arrangements of $P$ and $Q$, and not in the other.
Such subspaces which are contained in $\RP$ but not $\mathcal{R}_Q$ are either contained in the cut-off part $\mathcal{C}$ or are the intersection of a common face of $P$ and $Q$ with the hyperplane $H$.
Codimension-$c$ subspaces in $\mathcal{R}_Q$ which do not appear as subspaces of codimension $c$ in $\RP$ are contained in a higher-dimensional subspace in the cut-off part~$\mathcal{C}$.
 \hfill$\diamondsuit$
\end{enumerate}
\end{remark}
\begin{corollary}
\label{cor:residualRestricted}
The arrangement $(\RP)|_H$ obtained from $\RP$ by intersecting every irreducible component of $\RP$ with the hyperplane $H$ 
consists of $d-k-1$ subspaces of codimension one in $H$
as well as the residual arrangement $\mathcal{R}_F$ of the facet $F$ inside of $H$.
\end{corollary}

\begin{proof}
The $d-k-1$ hyperplanes in $H$ are the ones described in Remark \ref{rmk:CO}.\ref{enum:newFaces}.
In addition, $(\RP)|_H$ consists of the intersection of $H$ with subspaces in $\RP$ which are intersections of $H, H_1, H_2, \ldots, H_k$.
These subspaces do not contain faces of $P$, so they do not contain faces of $F$ and lie in the residual arrangement $\mathcal{R}_F$.

For the other direction, we consider a maximal subspace $V$ in $\mathcal{R}_F$.
If $V$ would contain a face of $P$, this face would also be a face of $Q$ due to Remark \ref{rmk:CO}.\ref{enum:facet}.
As $V$ is contained in $H$, we see that the hyperplane $H$ contains a face of $Q$, which contradicts our assumption that the hyperplane arrangement $\HP$ is simple.
Hence, $V$ is contained in a subspace of $\RP$.
This subspace is either $V$ itself, or $V$ is the intersection of this subspace with $H$. 
\end{proof}

With this, we consider the map which restricts a polynomial in the ideal of the residual arrangement $\RP$ to the hyperplane $H$:
\begin{align}
\label{eq:restriction}
\mathcal{I}_{\RP} (d-n-1)  \overset{|_H}{\longrightarrow} \mathcal{I}_{(\RP) |_H} (d-n-1).
\end{align}
Note that our goal is to show that $\mathcal{I}_{\RP} (d-n-1)$ has a unique global section up to scalar, whose zero locus is the adjoint $A_P$.
The map~\eqref{eq:restriction} is surjective. 
By determining its kernel, we can extend it to a short exact sequence:

\begin{proposition}
\label{prop:sequence}
The map~\eqref{eq:restriction} can be extended to the following short exact sequence:
$$
0 \longrightarrow \mathcal{I}_{\mathcal{R}_Q \cup\,\mathcal{C}}(d-n-2) \overset{\cdot H}{\longrightarrow} \mathcal{I}_{\RP} (d-n-1)  \overset{|_H}{\longrightarrow} \mathcal{I}_{(\RP) |_H} (d-n-1) \longrightarrow 0,
$$
where the left map is the multiplication with the defining equation of $H$.
\end{proposition}

\begin{proof}
The kernel of the map~\eqref{eq:restriction} consists of functions which vanish on $\RP$ and on $H$.
Hence, we need to determine the parts in $\RP$ which are not contained in $H$.
Due to Remark~\ref{rmk:CO}.\ref{enum:cutOff}, the subspaces in the cut-off part $\mathcal{C}$ are contained in $\RP$ but not in $H$.
Remark~\ref{rmk:CO}.\ref{enum:codim2} and~\ref{rmk:CO}.\ref{enum:codimGeneral} show that subspaces in $\mathcal{R}_Q$ are either contained in $\RP$ or in subspaces in $\mathcal{C}$.
As we assume the hyperplane arrangement $\HP$ to be simple, the subspaces in $\mathcal{R}_Q$ are not contained in $H$.
Thus, $\mathcal{R}_Q \cup \mathcal{C}$ consists of subspaces in $\RP$ which are not contained in~$H$.
Finally, Remark~\ref{rmk:CO}.\ref{enum:codim2} and~\ref{rmk:CO}.\ref{enum:codimGeneral} imply that all subspaces in $\RP$ which do not lie in~$H$ are contained in~$\mathcal{R}_Q \cup \mathcal{C}$.
\end{proof}

Now we conclude the proof of Theorem~\ref{thm:adjoint}, by computing the cohomology of this short exact sequence.

\begin{proposition}
\label{prop:cohomology}
We have $h^0(\PP^n, \mathcal{I}_{\mathcal{R}_Q \cup\,\mathcal{C}}(d-n-2)) = 0$ and $h^1(\PP^n, \mathcal{I}_{\mathcal{R}_Q \cup\,\mathcal{C}}(d-n-2)) = 0$.
\end{proposition}

\begin{proof}
By applying our induction hypothesis on $Q$, we see that $h^0(\PP^n, \mathcal{I}_{\mathcal{R}_Q }(d-n-2)) = 1$.
So $Q$ has a unique adjoint $A_Q$. Again by the induction hypothesis, this adjoint does not pass through any vertices of $Q$.
Since the cut-off part $\mathcal{C}$ contains at least one vertex, we deduce that $h^0(\PP^n, \mathcal{I}_{\mathcal{R}_Q \cup\,\mathcal{C}}(d-n-2)) = 0$.

For the second part, we consider the embedding $\iota: \mathcal{R}_Q \cup \mathcal{C} \hookrightarrow \PP^n$
and the following short exact sequence
together with the dimensions of the first cohomology groups:
\begin{align*}
\begin{array}{ccccccccc}
0 & \longrightarrow & \mathcal{I}_{\mathcal{R}_Q \cup\,\mathcal{C}}(d-n-2) & \longrightarrow & \mathcal{O}_{\PP^n}(d-n-2) & \overset{\iota^\ast}{\longrightarrow} & \iota_\ast \mathcal{O}_{\mathcal{R}_Q \cup\,\mathcal{C}}(d-n-2) & \longrightarrow & 0.\\
\hspace*{1mm} \\
h^0(\PP^n, \cdot) \hspace*{-6mm} && 0 && \binom{d-2}{n} && \alpha \\ 
h^1(\PP^n, \cdot) \hspace*{-6mm} && \alpha - \binom{d-2}{n} && 0
\end{array}
\end{align*}
In this diagram, we denote by $\alpha$ the dimension of the space of global sections of $\iota_\ast \mathcal{O}_{\mathcal{R}_Q \cup\,\mathcal{C}}(d-n-2)$.
To conclude the proof of Proposition~\ref{prop:cohomology}, we need to show that $\alpha = \binom{d-2}{n}$.

From the analogous short exact sequence for $\mathcal{I}_{\mathcal{R}_Q}(d-n-2)$
and our induction hypothesis that $h^0(\PP^n, \mathcal{I}_{\mathcal{R}_Q}(d-n-2)) = 1$ and $h^1(\PP^n, \mathcal{I}_{\mathcal{R}_Q}(d-n-2)) = 0$, 
we deduce that the space of global sections of the pushforward of $\mathcal{O}_{\mathcal{R}_Q}(d-n-2)$ to $\PP^n$ has dimension $\binom{d-2}{n}-1$.
Now we pick any vertex $v$ in the cut-off part $\mathcal{C}$. 
As this vertex is not contained in the residual arrangement $\mathcal{R}_Q$, 
we see that the space of global sections of the pushforward of $\mathcal{O}_{\mathcal{R}_Q \cup \lbrace v \rbrace}(d-n-2)$ to $\PP^n$ has dimension~$\binom{d-2}{n}$.

If the cut-off part $\mathcal{C}$ consists only of this vertex $v$, we are done.
Otherwise, we pick a cut-off edge $e$ such that $v$ is one of its two vertices.
The line $e$ is the intersection of $n-1$ hyperplanes in $\mathcal{H}_Q$.
There are two more hyperplanes in $\mathcal{H}_Q$ which define the two vertices of the edge $e$.
The other $d-n-2$ hyperplanes in $\mathcal{H}_Q$ intersect the line $e$ in $d-n-2$ points which are contained in subspaces in the residual arrangement $\mathcal{R}_Q$.
Hence, functions on $\mathcal{R}_Q \cup  e $ of degree $d-n-2$ are uniquely determined by functions of this degree on $\mathcal{R}_Q \cup \lbrace v \rbrace$,
i.e. the space of global sections of the pushforward of $\mathcal{O}_{\mathcal{R}_Q \cup  e }(d-n-2)$ to $\PP^n$ also has dimension~$\binom{d-2}{n}$.

By Remark \ref{rmk:CO}.\ref{enum:cutOff}, the cut-off edges are connected. 
By repeating the same argument as above, functions of degree $d-n-2$  on $\mathcal{R}_Q \cup \mathcal{E}$,
where $\mathcal{E} \subset \mathcal{C}$ denotes the union of all lines corresponding to cut-off edges, 
are uniquely determined by functions of this degree on $\mathcal{R}_Q \cup \lbrace v \rbrace$.
Thus, the space of global sections of the pushforward of $\mathcal{O}_{\mathcal{R}_Q \cup\,\mathcal{E}}(d-n-2)$ to $\PP^n$ has dimension $\binom{d-2}{n}$ as well.

If $\mathcal{C} = \mathcal{E}$, we are done.
Otherwise, we prove by induction on $\delta \in \lbrace 1, 2, \ldots, n-2 \rbrace$ that functions of degree $d-n-2$  on $\mathcal{R}_Q \cup\,\mathcal{C}_\delta$, 
where $\mathcal{C}_\delta$ denotes the union of all linear subspaces corresponding to $\delta$-dimensional cut-off faces,
are uniquely determined by functions of this degree on $\mathcal{R}_Q \cup \lbrace v \rbrace$.
As $\mathcal{C}_1 =  \mathcal{E}$, we have proven the induction beginning above.
For the induction step, we assume $\delta > 1$ and consider an arbitrary cut-off face $f$ of dimension $\delta$.
This face $f$ is defined by $n-\delta$ hyperplanes in $\mathcal{H}_Q$.
The other $d+\delta-n-1$ hyperplanes in $\mathcal{H}_Q$ intersect $f$ in $(\delta-1)$-dimensional subspaces.
Some of these subspaces are facets of $f$, 
and our induction hypothesis yields that functions of degree $d-n-2$ on $\mathcal{R}_Q \cup \lbrace v \rbrace$ extend uniquely to these facets.
The other hyperplanes in $f$ are already contained in subspaces in the residual arrangement $\mathcal{R}_Q$.
A general line $\ell$ in the face $f$ intersects the $d+\delta-n-1$ hyperplanes in $f$ in $d+\delta-n-1$ points.
Since $d+\delta-n-1 > d-n-2$, functions of degree $d-n-2$ on $\mathcal{R}_Q \cup \lbrace v \rbrace$ extend uniquely to the whole line $\ell$.
As $\ell$ was chosen generally, functions of degree $d-n-2$  on $\mathcal{R}_Q \cup  f $ are uniquely determined by functions of this degree on $\mathcal{R}_Q \cup \lbrace v \rbrace$.
This shows that the space of global sections of the pushforward of $\mathcal{O}_{\mathcal{R}_Q \cup\,\mathcal{C}_\delta}(d-n-2)$ to $\PP^n$ has dimension $\binom{d-2}{n}$.
Thus, $\alpha = \binom{d-2}{n}$ and Proposition~\ref{prop:cohomology} is proven.
\end{proof}

\bigskip
\begin{proof}[{Proof of Theorems~\ref{thm:adjointIntro} and~\ref{thm:adjoint}}]
Theorem~\ref{thm:adjoint} implies Theorem~\ref{thm:adjointIntro}.
We prove Theorem~\ref{thm:adjoint} by induction on $n$ and $d$. 
The induction beginning, i.e. $n=1$ or $d=n+1$, follows from Lemma~\ref{lem:tetrahedronAdjoint}.
For the induction step, we assume $n > 1$ and $d > n+1$, and we use the same notation as introduced in Remark \ref{rmk:CO}. 

We first prove that $P$ has a unique adjoint. 
By applying our induction hypothesis to the facet $F$, we see that $h^0(H, \mathcal{I}_{\mathcal{R}_F}(k-n)) = 1$,
meaning that $F$ has a unique adjoint $A_F$ as a polytope in $H$.
From Corollary~\ref{cor:residualRestricted} we deduce $h^0(H, \mathcal{I}_{(\RP) |_H} (d-n-1)) = 1$:
any nonzero section of $\mathcal{I}_{(\RP) |_H} (d-n-1)$ vanishes on $d-k-1$ hyperplanes in $H$ as well as the adjoint $A_F$ of~$F$, so these sections are all proportional.
Similarly, the induction hypothesis and Corollary~\ref{cor:residualRestricted} show $h^1(H, \mathcal{I}_{(\RP) |_H} (d-n-1)) = 0$.
Hence, Propositions~\ref{prop:sequence} and~\ref{prop:cohomology} show that $h^0(\PP^n, \mathcal{I}_{\RP} (d-n-1)) = 1$, i.e. the polytope $P$ has a unique adjoint $A_P$, and that $h^1(\PP^n, \mathcal{I}_{\RP} (d-n-1)) = 0$.

Finally, we assume for contradiction that this adjoint $A_P$ passes through a vertex $v$ of~$P$.
We then pick any hyperplane $H$ in $\HP$ containing $v$ and repeat the above argument (including the short exact sequence) with this specific hyperplane $H$.
Due to Corollary~\ref{cor:residualRestricted}, the intersection $A_P \cap H$ of the adjoint and the hyperplane consists of (several) linear subspaces of codimension one in $H$ and of the adjoint of the facet $P \cap H$.
By applying our induction hypothesis to this facet, we see that the adjoint of the facet does not pass through any vertex of $P$.
Hence, one of the linear subspaces of codimension one has to contain $v$.
This contradicts our assumption that the hyperplane arrangement $\HP$ is simple, 
since $v$ is a vertex of the facet $P \cap H$ and, by Remark~\ref{rmk:CO}.\ref{enum:newFaces}, $v$ is  also contained in one of the hyperplanes in $\HP$ not defining this facet.
\end{proof}

\bigskip
Now we have generalized Wachspress' notion of adjoints of polygons to
polytopes defined by simple hyperplane arrangements.
Next we show that Warren's adjoint polynomial of a polytope $P$ vanishes along the residual arrangement of the dual polytope $P^\ast$.
This shows that our notion of the adjoint hypersurface of $P^\ast$ coincides with Warren's adjoint of $P$ if the polytope $P^\ast$ is defined by a simple hyperplane arrangement.

\begin{proof}[{Proof of Proposition~\ref{prop:WarrenCompare}}]
We proceed analogous to Warren's proof of the similar statement that $\mathrm{adj}_P$ vanishes on the codimension two part of $\mathcal{R}_{P^\ast}$~\cite[Thm.~5]{War96}.
We use the following equation derived by Warren~\cite[Thm.~3]{War96}:
for any linear function $L(t)$, 
\begin{align}
\label{eq:warrenLinear}
L(t) \mathrm{adj}_P(t) = \sum \limits_{F \in \mathcal{F}(P)} L(v_F) \mathrm{adj}_F(t) \prod \limits_{v \in V(P) \setminus V(F)} \ell_{v}(t),
\end{align}
where 
the set of vertices of $P$ is $V(P)$, the set of facets is $\mathcal{F}(P)$, the vertex of the dual polytope $P^\ast$ corresponding to a facet $F \in \mathcal{F}(P)$ is denoted by $v_F$, and $\ell_v(t) = 1 - v_1t_1 - v_2t_2 - \ldots - v_nt_n$.

Now we use induction on $n$ to prove Proposition~\ref{prop:WarrenCompare}.
If $n = 1$, the polytopes $P$ and $P^\ast$ are line segments, so the residual arrangement  $\mathcal{R}_{P^\ast}$ is empty. 
If $n > 1$, we consider an arbitrary irreducible component $C$ of $\mathcal{R}_{P^\ast}$ of codimension $c$.
This component $C$ is the intersection of $c$ hyperplanes $H^1, H^2, \ldots, H^c$ in $\mathcal{H}_{P^\ast}$.
These hyperplanes correspond to $c$ vertices $v^1, v^2, \ldots, v^c$ of the polytope $P$.
That the intersection of the hyperplanes $H^1, H^2, \ldots, H^c$ does not contain a face of $P^\ast$
is equivalent to the fact that the vertices $v^1, v^2, \ldots, v^c$ do not span a face of $P$.

We show that the adjoint polynomial $\mathrm{adj}_P$ vanishes on $C$, by proving that, for each linear function $L(t)$, each summand in the right hand side of~\eqref{eq:warrenLinear} vanishes on $C$.
Thus, we fix an arbitrary facet $F$ of $P$. 
If at least one of the vertices $v^1, v^2, \ldots, v^c$ is not contained in~$F$, say $v^i \in V(P) \setminus V(F)$, 
then the summand in the  right hand side of~\eqref{eq:warrenLinear} corresponding to $F$ contains the term $\ell_{v^i}(t)$.
As this term is the defining equation of the hyperplane $H^i$, it vanishes on $C$.
Hence, we only have to consider the case that all vertices $v^1, v^2, \ldots, v^c$ are vertices of the facet $F$.
This means that the hyperplanes $H^1, H^2, \ldots, H^c$ contain the vertex $v_F$, i.e. $v_F \in C$.
Note that $H^i / \R v_F$ is the hyperplane defining the facet of $F^\ast$ which corresponds to the vertex $v^i$ of $F$.
As the vertices $v^1, v^2, \ldots, v^c$ do not span a face of $F$,
the intersection $C / \R v_F$ of these hyperplanes does not contain a face of the dual polytope~$F^\ast$.
So $C / \R v_F$ is contained in the residual arrangement $\mathcal{R}_{F^\ast}$. 
By the induction hypothesis, Warren's adjoint polynomial of the polytope $F$ (considered as a full-dimensional polytope in the hyperplane of $\R^n$ spanned by $F$) vanishes on $C / \R v_F$.
As a function, we can concatenate this polynomial with the projection from the vertex $v_F$.
This yields the adjoint $\mathrm{adj}_F(t)$ in~\eqref{eq:warrenLinear}.
Hence, $\mathrm{adj}_F(t)$ vanishes on $C$.

All in all, we have shown that the adjoint polynomial $\mathrm{adj}_P$ vanishes on the residual arrangement $\mathcal{R}_{P^\ast}$.
When the hyperplane arrangement $\mathcal{H}_{P^\ast}$ is simple and $d$ denotes the number of hyperplanes in $\mathcal{H}_{P^\ast}$,
we see from Theorem~\ref{thm:adjointIntro} that $\mathrm{adj}_P$ is, up to scalar, the unique polynomial of degree $d-n-1$ vanishing on $\mathcal{R}_{P^\ast}$,
so its zero locus is $A_{P^\ast}$.
\end{proof}

\bigskip
From this we deduce that a polytope with a non-simple hyperplane arrangement also has a unique adjoint: 
it is the unique limit of the adjoint hypersurfaces of small perturbations of the polytope.

\begin{proof}
[{Proof of Corollary~\ref{cor:limit}}]
Passing to an affine chart, we may consider $P$ as a convex polytope in $\R^n$.
We may further assume that $P_t$ and $Q_t$ are families of convex polytopes as well.
By Proposition~\ref{prop:WarrenCompare}, their unique adjoint hypersurfaces $A_{P_t}$ and $A_{Q_t}$
are the zero loci of the adjoint polynomials $\mathrm{adj}_{P_t^\ast}$ and $\mathrm{adj}_{Q_t^\ast}$, respectively.

Since the family $P_t^\ast$ is obtained by continuously moving the vertices of the polytopes in the family, we can triangulate the polytopes $P_t^\ast$ consistently, i.e., such that the triangulations move continuously with the family.
Indeed we can choose triangulations that do not introduce any new vertices such that the vertices of each simplex are vertices of the respective triangulated polytope. 
The resulting family $\tau(P_t^\ast)$ of triangulated simplicial polytopes converges for $t \to 0$ to a triangulation $\tau(P^\ast)$ of $P^\ast$.
From~\eqref{eq:warrenAdjoint}, we see that the adjoint polynomials $\mathrm{adj}_{\tau(P_t^\ast)} = \mathrm{adj}_{P_t^\ast}$ converge to $\mathrm{adj}_{\tau(P^\ast)} = \mathrm{adj}_{P^\ast}$.
Similarly, by picking some consistent triangulation of the family $Q_t^\ast$, we deduce that
$\lim \limits_{t \to 0} \mathrm{adj}_{Q_t^\ast} = \mathrm{adj}_{P^\ast}$.
Hence, both families of adjoint hypersurfaces $A_{P_t}$ and $A_{Q_t}$ converge to the zero locus of $\mathrm{adj}_{P^\ast}$.
\end{proof}

\bigskip
Next we show that Warren's adjoint is the central factor in Segre classes of monomial schemes.

\begin{proof}[{Proof of Proposition~\ref{prop:segreClass}}]
According to~\cite[Thm.~1.1]{aluffi2}, the Segre class of  $S_\mathcal{A}$ is
\begin{align}
\label{eq:integral}
\int_{N_\mathcal{A}} \frac{n!X_1 \cdots X_n}{(1+w_1X_1 + \ldots + w_nX_n)^{n+1}} dw.
\end{align}
By~\cite[Lem.~2.5]{aluffi2}, the integral~\eqref{eq:integral} over a simplex $\sigma$ with finite vertices is
$\frac{n! \mathrm{vol}(\sigma) X_1 \cdots X_n}{\prod_{v \in V(\sigma)} \ell_v(-X)}$
where $X = (X_1, \ldots, X_n)$.
So if we pick any triangulation $\tau(N_\mathcal{A})$ of the Newton region $N_\mathcal{A}$ that does not introduce any new vertices, we see that the Segre class of $S_\mathcal{A}$ is
\begin{align}
\label{eq:aluffiProof}
\begin{split}
\int_{N_\mathcal{A}} & \frac{n!X_1 \cdots X_n}{(1+w_1X_1 + \ldots + w_nX_n)^{n+1}} dw
= \!\!\!\!\!\!\!\!\! \sum_{\sigma \in \tau(N_\mathcal{A})} \int_{\sigma} \frac{n!X_1 \cdots X_n}{(1+w_1X_1 + \ldots + w_nX_n)^{n+1}} dw \\
&= \sum_{\sigma \in \tau(N_\mathcal{A})} \frac{n! \mathrm{vol}(\sigma) X_1 \cdots X_n}{\prod \limits_{v \in V(\sigma)} \ell_v(-X)} \\
&= \frac{n! X_1 \cdots X_n}{\prod \limits_{v \in V(N_\mathcal{A})} \ell_v(-X)} \sum_{\sigma \in \tau(N_\mathcal{A})} \!\! \mathrm{vol}(\sigma) \!\! \prod_{v \in V(N_\mathcal{A}) \setminus V(\sigma)} \!\! \ell_v(-X) \\
&= \frac{n! X_1 \cdots X_n}{\prod \limits_{v \in V(N_\mathcal{A})} \ell_v(-X)} \mathrm{adj}_{N_\mathcal{A}}(-X),
\end{split}
\end{align}
which proves the assertion.
\end{proof}

\begin{remark}
\label{rem:infiniteNewtonRegionExtended}
If the Newton region $N_\mathcal{A}$ has vertices at infinity in the direction of some standard basis vectors, we can do a similar computation as in~\eqref{eq:aluffiProof}.
For a simplex $\sigma$ with such infinite vertices, 
we consider the finite simplex $\sigma^\circ$ obtained by projecting $\sigma$ along its infinite directions. 
We denote by $\widehat{\mathrm{vol}}(\sigma) := \dim(\sigma^\circ)! \cdot \mathrm{vol}(\sigma^\circ)$ the normalized volume of $\sigma^\circ$.
By~\cite[Lem.~2.5]{aluffi2},
the integral~\eqref{eq:integral} over a simplex $\sigma$ with finite vertices and infinite vertices in the direction of some standard basis vectors is
$\frac{\widehat{\mathrm{vol}}(\sigma) X_1 \cdots X_n}{\prod_{v \in V(\sigma)} \ell_v(-X)}$,
where $\ell_{v_i}(t) = -t_i$ for the vertex $v_i$ at infinity in the direction of the $i$-th standard basis vector. Repeating the computation in~\eqref{eq:aluffiProof}, we derive that the Segre class of $S_\mathcal{A}$ is
\begin{align*}
\frac{n! X_1 \cdots X_n}{\prod \limits_{v \in V(N_\mathcal{A})} \ell_v(-X)} &\mathrm{adj}_{N_\mathcal{A}}(-X),\\
\quad\text{ where }
&\mathrm{adj}_{N_\mathcal{A}}(-X) = 
\sum_{\sigma \in \tau(N_\mathcal{A})} \!\! \frac{\widehat{\mathrm{vol}}(\sigma)}{n!} \!\! \prod_{v \in V(N_\mathcal{A}) \setminus V(\sigma)} \!\! \ell_v(-X).
\end{align*}
Since the adjoint is well-behaved under limits, 
it makes sense to define the adjoint of the unbounded Newton region $N_\mathcal{A}$ in that way.
 \hfill$\diamondsuit$
\end{remark}

\bigskip
Using arguments similar to those in the proof of Theorem~\ref{thm:adjoint}, we show that the dimension of the space $\Omega_P = H^0(\PP^n, \mathcal{I}_{\RP}(d-n))$ equals the number of vertices of the polytope $P$.

\begin{theorem}
\label{thm:wachspressCoords}
Let $d, n \in \Z_{>0}$ with $d \geq n+1$.
For every full-dimensional polytope $P$ in~$\PP^n$ with $d$ facets, $N$ vertices and a simple hyperplane arrangement $\HP$,
we have that $h^0(\PP^n, \mathcal{I}_{\RP}(d-n)) = N$.
\end{theorem}

\begin{proof}[{Proof of Theorems~\ref{thm:dimOmega} and~\ref{thm:wachspressCoords}}]
Theorem~\ref{thm:wachspressCoords} implies Theorem~\ref{thm:dimOmega}.
We prove Theorem~\ref{thm:wachspressCoords} again by induction on $n$ and $d$.
To start the induction, we consider the case that $P$ is a simplex, i.e. $d = n+1$.
As in the proof of Lemma~\ref{lem:tetrahedronAdjoint}, we see that $\mathcal{I}_{\RP} = \mathcal{O}_{\PP^n}$.
This shows $h^0(\PP^n, \mathcal{I}_{\RP}(d-n)) = n+1$, which equals the number of vertices of~$P$.
Since there are no polytopes with more than two facets in $\PP^1$, we have also proven the assertion for the case $n=1$. 

For the induction step, we assume $n > 1$ and $d > n+1$.
As before, we pick a hyperplane $H$ in $\HP$ and consider the polytope $Q$ obtained by removing the facet $F$ corresponding to $H$ from $P$.
In the following we will use the same notation as introduced in Remark~\ref{rmk:CO}.
Viewing $F$ as a polytope in $H$ with $k$ facets, the induction hypopthesis yields that $h^0(H, \mathcal{I}_{\mathcal{R}_F}(k-n+1))$ is the number $N_F$ of vertices of $F$.
By Corollary~\ref{cor:residualRestricted}, we have $h^0(H, \mathcal{I}_{(\RP)|_H}(d-n)) = N_F$.
In the following, we show that
\begin{align}
\label{eq:cohomologyOfResidualQplusCutoff0}
h^0(\PP^n, \mathcal{I}_{\mathcal{R}_Q \cup\,\mathcal{C}}(d-n-1)) &= N_Q - N_C
\quad\text{ and } \\
\label{eq:cohomologyOfResidualQplusCutoff1}
 h^1(\PP^n, \mathcal{I}_{\mathcal{R}_Q \cup\,\mathcal{C}}(d-n-1)) &= 0,
\end{align}
where $N_Q$ denotes the number of vertices of $Q$ and $N_C$ denotes the number of vertices in the cut-off part $\mathcal{C}$.
Note that the number $N$ of vertices of $P$ equals $N_Q-N_C+N_F$.
Hence, tensoring the short exact sequence in Proposition~\ref{prop:sequence} by $\mathcal{O}(1)$,
we see that~\eqref{eq:cohomologyOfResidualQplusCutoff0} and~\eqref{eq:cohomologyOfResidualQplusCutoff1} 
imply Theorem~\ref{thm:wachspressCoords}.

%--------------------\textbf{OLD}--------------------------

%Let us first prove~\eqref{eq:cohomologyOfResidualQplusCutoff0}.
%The induction hypothesis yields $h^0(\PP^n, \mathcal{I}_{\mathcal{R}_Q}(d-n-1)) = N_Q$.
%On the other hand, for any simple hyperplane arrangement of $d-1$ hyperplanes in $\PP^n$, there are no hypersurfaces of degree $d-n-1$ passing through all the $\binom{d-1}{n}$ points of intersection of $n$ hyperplanes:  there are $d-n$ such points on any line of intersection of $n-1$ of the hyperplanes, so by Bezout all these lines would be contained in the hypersurface, similarly all codimension $n-2$ spaces, etc, a contradiction.
%Therefore the $N_Q$ vertices of $Q$ are mapped to linearly independent points by the $N_Q$ coordinates of $\omega_Q$.

%There is a Zariski open set of global sections of $\mathcal{I}_{\mathcal{R}_Q}(d-n-1)$ that do not vanish on any vertex of $Q$.
%The points of the arrangement $\mathcal{H}_Q$ are the union of the points of $\mathcal{R}_Q$ and the vertices of~$Q$.  The vertices in $\mathcal{C}$ are a subset of the vertices of $Q$, so 
% the vertices in $\mathcal{C}$ impose independent conditions on forms of degree $d-n-1$ that %which are supposed to 
% vanish on $\mathcal{R}_Q$.
 % and the vertices $V(\mathcal{C})$ in $\mathcal{C}$.
%So $h^0(\PP^n, \mathcal{I}_{\mathcal{R}_Q \cup V(\mathcal{C})}(d-n-1)) = N_Q - N_C$.

%--------------------\textbf{NEW}--------------------------

We first prove~\eqref{eq:cohomologyOfResidualQplusCutoff0}.
By the induction hypothesis, $h^0(\PP^n, \mathcal{I}_{\mathcal{R}_Q}(d-n-1)) \!= N_Q$.
Evaluating the Wachspress coordinate of $\omega_Q$ corresponding to a vertex $u \in V(Q)$ at a vertex $v \in V(Q)$ yields zero if and only if $u \neq v$.
Thus the $N_Q$ vertices of $Q$ are mapped to linearly independent points by the $N_Q$ coordinates of $\omega_Q$.
Hence the vertices in any subset of $V(Q)$ impose independent conditions on forms of degree $d-n-1$ that vanish on $\mathcal{R}_Q$.
In particular, since $V(\mathcal{C}) \subset V(Q)$, we have $h^0(\PP^n, \mathcal{I}_{\mathcal{R}_Q \cup V(\mathcal{C})}(d-n-1)) = N_Q - N_C$.

%--------------------\textbf{END}--------------------------

Now we show by induction on $\delta \in \lbrace 0, 1, \ldots, n-2 \rbrace$ that polynomials of degree $d-n-1$ which vanish on $\mathcal{R}_Q \cup V(\mathcal{C})$ also vanish on $\mathcal{C}_\delta$,
where $\mathcal{C}_\delta$ denotes the union of all linear subspaces corresponding to $\delta$-dimensional cut-off faces.
As $\mathcal{C}_0 = V(\mathcal{C})$, the induction beginning is trivial. 
For the induction step, we proceed similarly to the proof of Proposition~\ref{prop:cohomology}.
We consider an arbitrary cut-off face $f$ of dimension $\delta > 0$.
This face $f$ is defined by $n-\delta$ hyperplanes in $\mathcal{H}_Q$. 
The other $d+\delta-n-1$ hyperplanes in $\mathcal{H}_Q$ intersect $f$ in $(\delta-1)$-dimensional subspaces. 
Some of these are facets of $f$, on which global sections of $\mathcal{I}_{\mathcal{R}_Q \cup V(\mathcal{C})}(d-n-1)$ vanish due to the induction hypothesis,
and the others are already contained in subspaces in $\mathcal{R}_Q$.
Hence, polynomials of degree $d-n-1$ which vanish on $\mathcal{R}_Q \cup V(\mathcal{C})$ also vanish on these $d+\delta-n-1$ hyperplanes in $f$.
Since $d+\delta-n-1 > d-n-1$, such polynomials also vanish on $f$.
This proofs the induction step and thus \eqref{eq:cohomologyOfResidualQplusCutoff0}.

Finally, we show \eqref{eq:cohomologyOfResidualQplusCutoff1}.
As in the proof of Proposition~\ref{prop:cohomology}, we consider the following short exact sequence
together with the dimensions of the first cohomology groups:
\begin{align*}
\begin{array}{ccccccccc}
0 & \longrightarrow & \mathcal{I}_{\mathcal{R}_Q \cup\,\mathcal{C}}(d-n-1) & \longrightarrow & \mathcal{O}_{\PP^n}(d-n-1) & \overset{\iota^\ast}{\longrightarrow} & \iota_\ast \mathcal{O}_{\mathcal{R}_Q \cup\,\mathcal{C}}(d-n-1) & \longrightarrow & 0.\\
\hspace*{1mm} \\
h^0(\PP^n, \cdot) \hspace*{-6mm} && N_Q  - N_C && \binom{d-1}{n} && \beta \\ 
h^1(\PP^n, \cdot) \hspace*{-6mm} && \hspace*{-2mm} \beta - \binom{d-1}{n} + N_Q - N_C \hspace*{-5mm} && 0
\end{array}
\end{align*}
Here $\beta$ denotes the dimension of the space of global sections of $\iota_\ast \mathcal{O}_{\mathcal{R}_Q \cup\,\mathcal{C}}(d-n-1)$.
To conclude the proof of~\eqref{eq:cohomologyOfResidualQplusCutoff1} and thus of Theorem~\ref{thm:wachspressCoords}, 
we show that $\beta = \binom{d-1}{n} - N_Q + N_C$.

We first observe that $\binom{d-1}{n} - N_Q + N_C$ is exactly the number of vertices in $\mathcal{R}_Q \cup \mathcal{C}$.
Clearly, there are $\binom{d-1}{n} - N_Q + N_C$ functions of degree $d-n-1$ defined on these vertices.
Let us now consider an arbitrary face $f$ of dimension $\delta > 0$ which is either contained in $\mathcal{R}_Q$ or in~$\mathcal{C}$.
This face $f$ is defined by $n-\delta$ hyperplanes in $\mathcal{H}_Q$.
Thus, there are $\binom{d+\delta-n-1}{\delta}$ vertices in $f$ which are defined by the other $d+\delta-n-1$ hyperplanes in $\mathcal{H}_Q$.
All of these vertices are vertices in $\mathcal{R}_Q \cup \mathcal{C}$.
A function defined on these $\binom{d+\delta-n-1}{\delta}$ vertices extends uniquely to a function of degree $d-n-1$ on $f$,
and each function of this degree on $f$ is uniquely determined by its values on these vertices.
As the face $f$ was chosen arbitrarily, we have shown that $\beta = \binom{d-1}{n} - N_Q + N_C$.
Hence, we have proven \eqref{eq:cohomologyOfResidualQplusCutoff1} and Theorem~\ref{thm:wachspressCoords}.
\end{proof}

 \section{The Wachspress coordinate map}
 \label{sec:barycentric}
 The Wachspress coordinates, see \eqref{eq:wachspressMap}, are enumerated by the vertices $V(P)$ of $P$. 
 For a simple polytope $P$ in $\R^n$ with $d$ facets and a vertex $u\in V(P)$, the coordinate $$\omega_{P,u} (t) := \prod_{F \in \mathcal{F}(P): \, u \notin F} \ell_F(t)$$ vanishes on the $d-n$ hyperplanes in $\HP$ that do not contain $u$:  
 $$Z(\omega_{P,u})=\bigcup\limits_{H\in \HP: \, u\notin H}H.$$
 
 \begin{remark} 
\label{rem:wachspressRestricted}
 If $H$ is a hyperplane spanned by a facet $F$ of $P$,
then only the Wachspress coordinates of vertices in $F$ do not vanish on $H$.
By Corollary \ref{cor:residualRestricted}, if the polytope $F$ has $k$ facets, the coordinates of vertices in $F$ have a common factor vanishing on $d-k-1$ subspaces of codimension one in $H$.
The residual factors are precisely the Wachspress coordinates of~$\omega_F$.
 \hfill$\diamondsuit$
 \end{remark}
 
We now show that the common zero locus of the Wachspress coordinates, their {\em base locus}, is the residual arrangement $\RP$.

\begin{proof}[{Proof of Theorem~\ref{thm:baseLocus}}]
We first show this assertion set-theoretically.
To begin with, we consider an irreducible component $L$ of the residual arrangement $\RP$. 
It is the intersection of hyperplanes $H_1, H_2, \ldots, H_c$ in the arrangement $\HP$. 
For every vertex $u$ of $P$, at least one of these hyperplanes does not contain $u$.
Thus, the Wachspress coordinate $\omega_{P,u}$ vanishes along $L$. This shows that the residual arrangement $\RP$ is contained in the base locus of the Wachspress coordinates.

For the other direction, 
we show by induction on $n$ that the base locus of the Wachspress coordinates is contained in the residual arrangement $\RP$.
For the induction beginning, i.e. $n=1$, the polytope $P$ is a line segment and the common zero locus of its two Wachspress coordinates is empty.
For the induction step, we assume $n > 1$ and consider a point $p$ in the base locus of the Wachspress coordinates.
 As each coordinate vanishes on a union of hyperplanes in $\HP$, the base locus is contained in $\HP$.
So the point $p$ lies in a hyperplane $H$ in $\HP$.
Since also the Wachspress coordinates of vertices of $P$ in $H$ vanish on $p$, the point $p$ must in fact lie on at least one other hyperplane in $\HP$.
If there is another hyperplane $H'$ in $\HP$ containing $p$ such that the intersection $H \cap H'$ is in the residual intersection $\RP$, we are done.
Otherwise each hyperplane containing $p$ is either $H$ or defines a facet of the facet $F = P \cap H$ of $P$ in $H$.
In that case, the point $p$ is also contained in the base locus of the Wachspress coordinates of $F$ viewed as a polytope in $H$.
By our induction hypothesis, $p$ is contained in the residual arrangement $\mathcal{R}_F$ and thus also in $\RP$. 
Hence, we have shown that the base locus of the Wachspress coordinates of $P$ equals $\RP$ set-theoretically.

To show that the equality also holds for the scheme, i.e. that the base locus has no embedded components, we let now $p$ be any point in $\RP$ and consider the union 
$$Z_p=\bigcup\limits_{p\in L\subset \RP} L$$ 
of linear spaces in $\RP$ that pass through the point $p$.
We first prove two claims which show how the linear equations of the hyperplanes in $\HP$ yield defining equations of $Z_p$.

\begin{claimnumbered}\label {genspan} If $L_1$ and $L_2$ are linear spaces through $p$ that are intersections of hyperplanes in $\HP$,
then the linear span of $L_1\cup L_2$ is either the whole ambient space or it is also an intersection of hyperplanes in $\HP$.
In particular, the span of the union of any collection of components in $Z_p$ is the intersection of a (possible empty)  set of hyperplanes in $\HP$.
\end{claimnumbered}
\begin{proof}
\renewcommand\qedsymbol{$\diamondsuit$}
If $L_i$ has codimension $c_i$, then there are $c_i$ hyperplanes in $\HP$ whose intersection is~$L_i$.
If $L_1\cap L_2$ has codimension $c$, then  the set of $c$ hyperplanes in $\HP$ whose intersection is $L_1 \cap L_2$ is the union of the two sets of hyperplanes that contain $L_1$ or $L_2$, respectively.
These two sets must therefore have $c_1+c_2-c$ common hyperplanes whose intersection must coincide with the span of $L_1\cup L_2$.   The last statement of the claim follows by induction.
\end{proof}

\begin{claimnumbered}\label{ZpGen}
If $Z_p$ is irreducible, it is the scheme theoretic intersection of all hyperplanes containing $Z_p$.
Otherwise $Z_p$ is the scheme theoretic intersection of these hyperplanes together with the set of reducible quadrics
with one component containing one of the (maximal) linear spaces $L$ that is a component in $\RP$ passing through $p$ while the other component contains the span $L'$ of the rest $Z_p\setminus L$. 
\end{claimnumbered}
\begin{proof}
\renewcommand\qedsymbol{$\diamondsuit$}
The first part of this assertion is clear.
So we assume that $Z_p$ is reducible.
Note that $L'$ is not the whole ambient space due to the simplicity of the hyperplane arrangement~$\HP$.
% To see this, note first that $P$ is simple so if $L$ is a linear component in $\RP$ through $p$, the linear span of $Z_p\setminus L$ is the intersection of hyperplanes that does not contain $L$.   
Clearly $Z_p$ is contained in the intersection of the described hyperplanes and reducible quadrics.
To prove equality, let $L''$ be a linear space that is a component in the intersection of these hyperplanes and the collection of reducible quadrics, one set of quadrics for each component $L$ in $Z_p$.  Then $L''$ is either one of the components $L$ in $Z_p$, or it is contained in the linear span $L'$ of $Z_p\setminus L$ for every component $L$.  But since $\HP$ is simple, the intersection of all such linear spans $L'$  is the intersection of all the linear components $L$ in $Z_p$, which proves the claim.
\end{proof}

For each $p\in \RP$, we consider the hyperplanes $H^s_1,...,H^s_r$ in $\HP$ that contain $Z_p$, and the hyperplanes $H_1,...,H_t$ in $\HP$ that do not contain $Z_p$ but still pass through~$p$.
Since the hyperplane arrangement $\HP$ is simple, we have $r+t\leq n$.
Furthermore, we consider the following set of Wachspress coordinates:
%$$\Omega(p,1)= 
%\{\omega_{P,u}|u\in V(P),Z(\omega_{P,u})=H_i \cup Y_u , L\cup %L'\subset H_i, p\notin Y_u\}$$ 
$$\Omega(p,s)= 
\{\omega_{P,u}|u\in V(P), \,
\exists i: H^s_i \subset
Z(\omega_{P,u}), \, p \notin \overline{Z(\omega_{P,u}) \setminus H^s_i} \}.$$
If $Z_p$ is reducible, we also consider a second set of Wachspress coordinates
for every linear component $L\subset \RP$ through $p$:
%$$\Omega(p,L)= \{\omega_{P,u}|u\in V(P),Z(\omega_{P,u})=H_i \cup H_j\cup Y_u , L\subset H_i, L'\subset H_j, p\notin Y_u\}.$$
$$\Omega(p,L)= \left\lbrace \omega_{P,u} \;\middle\vert\;
\begin{array}{l}
u\in V(P), \,
\exists i, j: H_i \cup H_j \subset
Z(\omega_{P,u}), \, L\subset H_i, \\
L'\subset H_j, 
\, p \notin \overline{Z(\omega_{P,u}) \setminus (H_i \cup H_j)}
\end{array}
\right\rbrace.$$
To prove the theorem, it suffices to show the following two claims.
\begin{claimnumbered}
The intersection  of the hyperplanes $H_i^s$ in the zero sets of the coordinates in $\Omega(p,s)$ equals the span of $Z_p$.
\end{claimnumbered}
\begin{proof}
\renewcommand\qedsymbol{$\diamondsuit$}
If $Z_p$ spans the ambient space, $r=0$ and there is nothing to prove.
Otherwise, for $1\leq i\leq r$,  we set 
$$L_{i} := \bigcap\limits_{1\leq j\leq r: \,j\not=i}H^s_j \cap \bigcap \limits_{1\leq k\leq t}H_k .$$ 
Then $L_{i}$ is a linear space through $p$ that is not contained in $\RP$.   In fact, if it is, then it is contained in a component of $Z_p$ that at the same time is not contained in $H^s_i$, which is impossible.
Therefore $L_{i}$ contains a face of $P$.
Since the hyperplane arrangement $\HP$ is simple, the linear space $L_i$ is spanned by vertices of $P$.  
As $H^s_i$ defines a hyperplane in $L_i$, these vertices cannot all lie in $H^s_i$.  Therefore at least one vertex, say $u$, lies in $L_{i}\setminus H^s_i$.  
 The coordinate $\omega_{P,u}$ therefore  vanishes on $H^s_i$ but not on any of the other hyperplanes in $\HP$ that pass through~$p$, so $\omega_{P,u}\in \Omega(p,s)$.  Notice that we have such a coordinate for each hyperplane $H^s_i$ that contains $Z_p$.  
\end{proof}
\begin{claimnumbered}
If $Z_p$ is reducible, the intersection inside the span of $Z_p$ of the reducible quadrics $H_i\cup H_j$  in the zero sets of the coordinates in $\Omega(p,L)$ equals  $L\cup L'$. 
\end{claimnumbered}
\begin{proof}
\renewcommand\qedsymbol{$\diamondsuit$}
 By Claim \ref{genspan}, both $L$ and $L'$ are intersections of some of the hyperplanes $H_1$, $\ldots$, $H_t$ inside the span of $Z_p$.  
 So there are $1\leq i,j\leq t$ such that $L\subset H_i$ and  $L'\subset H_j$. We set
 $$L_{i,j}:= \bigcap\limits_{1\leq l\leq r}H^s_l \cap \bigcap\limits_{1\leq k\leq t: \, k\notin \lbrace i,j \rbrace}H_k.$$ 
Then $L_{i,j}$ is a linear space through $p$ that is not contained in $\RP$.   In fact, if it is, then it is contained in a component $L''$ of $Z_p$ that at the same time is contained in neither $H_i$ nor $H_j$, so $L''$ is distinct from $L$ and is not contained in $L'$ against the assumption that $Z_p\setminus L\subset L'$. 
Therefore $L_{i,j}$ is spanned by vertices of $P$.  
None of these vertices is contained in $H_i$, since any such vertex would be contained in $L$, in particular in $\RP$, which does not contain any vertices of $P$.
 On the other hand, these vertices span $L_{i,j}$, so they cannot all lie in $H_j$.  Therefore at least one vertex, say $u$, lies in $L_{i,j}\setminus (H_i\cup H_j)$.  
 The coordinate $\omega_{P,u}$ therefore  vanishes on $H_i$ and on $H_j$ but not on any of the other hyperplanes in $\HP$ that pass through~$p$, so $\omega_{P,u}\in \Omega(p,L)$.
 Notice that we have such a coordinate for each hyperplane $H_i$ that contains $L$ but not $Z_p$, and each hyperplane $H_j$ that contains~$L'$ but not $Z_p$.  
\end{proof}

Varying this construction over all decompositions of $Z_p$ into $L$ and $L'$,
we have recovered the scheme theoretic intersection of hyperplanes and reducible quadrics yielding $Z_p$ as described in Claim~\ref{ZpGen}.
Moreover, we have shown that each of these hyperplanes or quadrics is contained in the zero locus of a Wachspress coordinate such that the other components in the zero locus do not pass through $p$.
This concludes the proof of the theorem.
\end{proof}

Theorems~\ref{thm:baseLocus} and~\ref{thm:dimOmega} imply that the Wachspress map $$\omega_P: \PP^n \dashrightarrow \PP(\Omega_P^\ast)$$
is given by all homogeneous forms of degree $d-n$ which vanish along the residual arrangement~$\RP$.
In the remainder of this section, we study the image of the Wachspress map, i.e. the Wachspress variety $W_P = \overline{\omega_P(\PP^n)}$, 
as well as the image of the adjoint $A_P$ under $\omega_P$.
First we examine the projective space $\mathbb{V}_P \subset \PP(\Omega_P^\ast)$ spanned by the image of the adjoint.

\begin{lemma}
\label{lem:dimImageAdjoint}
Let $P$ be a full-dimensional polytope in $\PP^n$ with $N$ vertices and a simple hyperplane arrangement $\HP$.
The dimension of $\mathbb{V}_P = \mathrm{span}\lbrace \omega_P(A_P)  \rbrace \subset \PP(\Omega_P^\ast)$ is $N-n-2$.
\end{lemma}

\begin{proof}
Due to Proposition~\ref{prop:WarrenCompare}, Warren's adjoint polynomial $\mathrm{adj}_{P^\ast}$ vanishes along $\RP$ and has degree $d-n-1$.
Here we consider its homogeneous version. 
Since the forms $t_0 \mathrm{adj}_{P^\ast}(t)$, $t_1 \mathrm{adj}_{P^\ast}(t)$, $\ldots, t_n \mathrm{adj}_{P^\ast}(t)$ are linearly independent in $\Omega_P$, 
we can extend them to a basis of the space $\Omega_P$ by adding $N-n-1$ forms, say $\varphi_1(t)$, $\varphi_2(t)$, $\ldots$, $\varphi_{N-n-1}(t)$.
With these coordinates we can pick a different realization of the Wachspress map:
\begin{align}
\label{eq:wachspressMapWithAdjoint}
w_P(t) = \left( t_0 \mathrm{adj}_{P^\ast}(t): t_1 \mathrm{adj}_{P^\ast}(t) : \cdots : t_n \mathrm{adj}_{P^\ast}(t) :  \varphi_1(t) :  \varphi_2(t) :  \cdots :  \varphi_{N-n-1}(t) \right).
\end{align}
As the first $n+1$ forms vanish on the adjoint $A_P$, we see immediately that the dimension of~$\mathbb{V}_P$ is at most $N-n-2$.
To prove that the dimension of $\mathbb{V}_P$ cannot be smaller, 
it is enough to show that every hyperplane in $\PP(\Omega_P^\ast)$ containing $\mathbb{V}_P$ is defined by a linear form in the first $n+1$ coordinates in~\eqref{eq:wachspressMapWithAdjoint}. 
Indeed, the preimage under $\omega_P$ of a hyperplane in $\PP(\Omega_P^\ast)$ containing $\mathbb{V}_P$ is a hypersurface in $\PP^n$ of degree $d-n$ which contains the adjoint $A_P$, 
so it is the zero locus of a linear combination of the first $n+1$ entries in~\eqref{eq:wachspressMapWithAdjoint}.
\end{proof}
 
This lemma shows that the projection $\rho_P$ from $\mathbb{V}_P$ maps $\PP(\Omega_P^\ast) \cong \PP^{N-1}$ onto $\PP^n$.
In fact, the restriction of this projection to the Wachspress variety $W_P$ is the inverse of the Wachspress map $\omega_P$.

\begin{proof}[{Proof of Theorem~\ref{thm:inverseOfWachspress}}]
From~\eqref{eq:wachspressMapWithAdjoint} we see that $\rho_P \circ \omega_P$ maps a point $p \in \PP^n$ to  $\mathrm{adj}_{P^\ast}(p) \cdot p$.
Hence, outside of the adjoint $A_P$, this map is the identity. 

A point $(x:y) \in W_P$ is of the form $x = \mathrm{adj}_{P^\ast}(p) \cdot p$ and $y = \varphi(p)$ for some point $p \in \PP^n$.
Again by~\eqref{eq:wachspressMapWithAdjoint}, the map $\omega_P \circ \rho_P$ maps $(x:y)$ to $(\mathrm{adj}_{P^\ast}(x) \cdot x : \varphi(x)) = \mathrm{adj}_{P^\ast}(t)^{d-n} (x:y)$.
Thus, outside of  $\mathbb{V}_P$, the map $\omega_P \circ \rho_P$ is the identity.
\end{proof}

In the remainder of this section we focus on polytopes $P$ in $\PP^3$ with simple plane arrangements $\HP$.
To show Proposition~\ref{prop:wachspressVariety}, 
we first exhibit the combinatorial structure of the residual arrangement $\RP$ and
afterwards we investigate the blowup $\pi_P:X_P\to \PP^3$ along~$\RP$.

\begin{lemma}\label{lem:residualArrP3}
Let $P \subset \PP^3$ be a full-dimensional polytope with $d$ facets and a simple plane arrangement $\HP$.
The residual arrangement $\RP$ consists of $\binom{d-3}{2}$ lines.
Moreover, the $a$ isolated points in $\RP$ as well as the $b$ double and the $c$ triple intersections of the lines in $\RP$ satisfy
$$b+2c-a = \frac{(d-2)(d-4)(d-6)}{3}.$$
\end{lemma}

\begin{proof}
Since $P$ is simple, twice the number of its edges is three times the number of its vertices.
Together with Euler's formula we get that $P$ has $2(d-2)$ vertices and $3(d-2)$ edges.
This shows that $\RP$ has $\binom{d}{2} - 3(d-2) = \binom{d-3}{2}$ lines.

Let $S \in \binom{\HP}{3}$ be a set consisting of three planes in $\HP$.
The pairwise intersections of these planes are three lines. 
We denote by $\alpha(S) \in \lbrace 0,1,2,3 \rbrace$ the number of these lines that are part of the residual arrangement $\RP$.
If $\alpha(S) = 0$ (i.e., all three lines are edges of~$P$), then the intersection of the three planes in $S$ is either a vertex of $P$ or an isolated point of $\RP$. 
As $P$ has $2(d-2)$ vertices, we have shown that
$$ b+2c-a-2(d-2)  = \sum_{S \in \binom{\HP}{3}} \left(  \alpha(S) -1 \right).$$
Since a fixed line can be derived from exactly $d-2$ sets $S \in \binom{\mathcal{\HP}}{3}$ as the intersection of two of the planes contained in $S$, the sum $\sum_S \alpha(S)$ counts every line in the residual arrangement $(d-2)$ times. 
This yields
\begin{align*}
b+2c-a - 2(d-2)= \sum_{S \in \binom{\HP}{3}} \left(  \alpha(S) -1 \right)
%= \sum_{S \in \binom{\HP}{3}}  \alpha(S) - \binom{d}{3}
 = (d-2) \cdot \binom{d-3}{2} - \binom{d}{3}.
\end{align*}
\end{proof}

  \begin{lemma}\label{exceptional} 
Let $P \subset \PP^3$ be a full-dimensional polytope with a simple plane arrangement~$\HP$.
We consider the blowup $\pi_P:X_P\to \PP^3$ along the residual arrangement $\RP$.
  \begin{enumerate}
  \item The exceptional divisor in $X_P$ has one component isomorphic to $\PP^2$ over each isolated point and over each triple point, and one component isomorphic to a rational ruled surface over each line.
  \item  The exceptional divisor as a Cartier divisor on $X_P$ has multiplicity one along each component lying over isolated points and lines, and multiplicity two along each plane over a triple point.
  \item The components of the exceptional divisor intersect pairwise in a line if they lie over lines through a double point, or if one component lies over a line through a triple point and the other lies over the triple point.  Two components lying over lines through a triple point meet in a node in the component over the triple point.
  \item  The singularities of the threefold $X_P$ are isolated quadratic singularities, one over each double point and three over each triple point.  
  \item The strict transform in the blowup of a general quadric surface through the three lines of a triple point is smooth and intersects the exceptional component over the triple point in a line that does not pass through any of the nodes.
  \end{enumerate}
  \end{lemma}
  \begin{proof}  The statements are all local over $\RP$. The blowup of a  smooth variety along a smooth subvariety  is smooth. The exceptional divisor is a projective bundle over the subvariety, see \cite[Thm 8.24]{HAR}.  It therefore suffices to check the statement concerning double and triple points on $\RP$.
  
  A double point on a curve is locally given by equations $x=yz=0$ in $\A^3$.  Then the blowup along the curve is defined by $sx-tzy=0$ in  $\A^3\times \PP^1$, where $s,t$ are the homogeneous coordinates on $\PP^1$.  The exceptional set over $(0,0,0)$ is then isomorphic to $\PP^1$, while the blowup is smooth except at the point $s=x=z=y=0$ where it has a quadratic singularity.  The exceptional divisors over the lines $x=y=0$ and $x=z=0$, are planes that meet along the line $x=y=z=0$ in the blowup.
  
  A triple point on a curve is locally given by equations $xy=yz=xz=0$ in $\A^3$.  Then the blowup along the curve is defined by 
  $${\rm rank}\begin{pmatrix} xy&yz&xz\\ s&t&u\end{pmatrix}\leq 1\quad {\rm i.e.}\quad z(xt-yu)=y(xt-zs)=x(yu-zs)=0$$ 
  in  $\A^3\times \PP^2$, where $s,t,u$ are the homogeneous coordinates on $\PP^2$.  The exceptional set over $(0,0,0)$ is then isomorphic to $\PP^2$.  In the affine part where $s=1$, the equations of the blowup are $z(xt-yu)=y(xt-z)=x(yu-z)=0$.  These equations define the threefold $z-xt=xt-yu=0$ which is smooth except at the origin $x=y=z=t=u=0$, where it has a quadratic singularity.  For $t=1$ and $u=1$ we get two more points, all three quadratic singularities that lie in the exceptional plane $x=y=z=0$.   
  The exceptional divisor lying over the line $y=z=0$ has two components, the plane $x=y=z=0$ and the plane $y=z=t=0$. The two planes meet along the line $x=y=z=t=0$.
   The latter plane lies over the whole line, the other only over the 
  triple point.  Similarly, over each of the two other lines  there are two components, one is the common plane over the triple point, the other lies over the whole line. 
  The three components distinct from $x=y=z=0$ intersect pairwise in the nodes over the triple point:  when $s=1$ the two planes $y=z=t=0$ and $x=z=u=0$ intersect in the node $x=y=z=t=u=0$.  
  
  The multiplicity of the exceptional components in the exceptional divisor as a Cartier divisor is given by the multiplicity of the generators of the ideal along the locus in $\A^3$ that is blown up.  The generators are smooth along the lines except at the triple point where the multiplicity is two.
  
  The quadric cone  $xy+xz+yz=0$ passes through the three lines of a triple point.  Since the blowup is defined by the locus where the matrix $\begin{pmatrix} xy&yz&xz\\ s&t&u\end{pmatrix}$ has rank one, the strict transform of the surface contains the line $s+t+u=0$ in the exceptional plane over the triple point, which does not pass through any of the nodes.
  \end{proof}
    
In the following, we use the notation from Proposition~\ref{prop:wachspressVariety}.
  Let $E_{(3)}$ be the union of exceptional divisors of the blowup $\pi_P: X_P\to \PP^3$ that lie over the $a$ isolated points in $\RP$.   
  Let $E_{(2)}$ be the union of exceptional divisors that dominate the lines in $\RP$ %with multiplicity one along each component that dominates a line, and multiplicity two along the components over triple points.  We write $F=F_l+2F_t$, where $F_l$ is the union of components that dominate lines, while 
  and let $E_{(3),t}$ be the union of components lying over triple points.  Thus $E_{(3)}+E_{(2)}+2E_{(3),t}$ is the exceptional divisor of the blowup as a Cartier divisor.
%The summand $G$ consists of $a$ disjoint components that are also disjoint from $F$.   The components in $F_l$ intersect pairwise in a $\PP^1$ if the corresponding lines in $\RP$ meet in a double point, and in a singular point in $X_P$ if the lines meet at a triple point.  A component in $F_l$ that corresponds to a line that passes through a triple point intersects the corresponding component of $F_t$ in a line through two nodes of $X_P$.   

We use the divisor classes on $X_P$ to compute degrees and sectional genera of $W_P$, of the image of the adjoint surface $A_P$ and of surfaces in $W_P$ linearly equivalent to the image of the plane arrangement $\HP$.  It would be natural to compute these degrees and genera as intersection numbers on $X_P$.  But $X_P$ is singular, so these intersection numbers are not easily defined.
On the other hand, by Theorem \ref{thm:baseLocus}, the singular locus of the common zeros of the Wachspress coordinates consists of the double points and the triple points of $\RP$. 
So by  Bertini's theorem the general linear combination of Wachspress coordinates vanishes on a surface $S$ that has quadratic singularities at the triple points and is smooth elsewhere. 
   
 Quadratic singularities on $S$ are resolved by a single blowup with a $(-2)$-curve as exceptional divisor, a line in the exceptional plane over a triple point of $\RP$ according to Lemma~\ref{exceptional}.  In particular the strict transform $\tilde S$ of $S$ is smooth on $X_P$.  So by restricting to $\tilde S$ we may compute the intersection numbers we need for our degree and genus computations.

We let $h$ denote the pullback to $X_P$ of the class of a plane in $\PP^3$. 
  Since $X_P$ has only isolated quadratic singularities, the canonical divisor on $X_P$ is Cartier and given by
$K_{X_P}=-4h+2E_{(3)}+E_{(2)}+\alpha E_{(3),t}$ 
for some integer $\alpha$.
The class of the smooth surfaces $\tilde S$ on $X_P$ is the pullback $h_{W_P}$ of hyperplane sections on $W_P$ to  $X_P$.  It is given by
\begin{align} \label{eq:classTildeS}
[\tilde S]=h_{W_P}=(d-3)h-E_{(3)}-E_{(2)}-2E_{(3),t},
\end{align}
where $d$ is the number of facets of the polytope $P$, since $\tilde S$ has multiplicity one at the isolated points and along the lines of $\RP$ and multiplicity two at the triple points.
By adjunction the canonical divisor on $\tilde S$ is the restriction to the surface of $h_{W_P}+K_{X_P}=(d-7)h+E_{(3)}+(\alpha - 2)E_{(3),t}$.  The restriction of $E_{(3),t}$ to the surface is the union of exceptional $(-2)$-curves lying over the quadratic singularities in $S$.  The intersection number between the canonical divisor on $\tilde S$ and these curves is zero, so we conclude that $\alpha=2$, and therefore
 $$K_{X_P}=-4h+2E_{(3)}+E_{(2)}+2E_{(3),t}.$$

We recall that $\Gamma_P$ is the linear system of polytopal surfaces of $P$, i.e. of divisors in $\PP^3$ which have degree $d$ and vanish with multiplicity $c$ along the codimension-$c$ part $\RPc$ of the residual arrangement.
The strict transform $\tilde D$ of a  general divisor $D \in \Gamma_P$ is linearly equivalent to the strict transform of $\HP$.
Since the polytopal surface $D$ has multiplicity three along the isolated points, two along the lines in $\RP$ and three at the triple points, the strict transform $\tilde D$ on $X_P$ belongs to the class
\begin{align}\label{eq:classTildeD}
[\tilde D]=dh-3E_{(3)}-2E_{(2)}-3E_{(3),t}.
\end{align}
The adjoint $A_P$ is a surface of degree $d-4$ with multiplicity one along the isolated points, one along the lines, and two at the triple points, so its strict transform $\tilde A_P$ on $X_P$ belongs to the class
\begin{align}\label{eq:classTildeAdjoint}
[\tilde A_P]=(d-4)h-E_{(3)}-E_{(2)}-2E_{(3),t}=[\tilde D]+K_{X_P}-E_{(3),t}.
\end{align}

By adjunction the restriction of $[\tilde D]+K_{X_P}=(d-4)h-E_{(3)}-E_{(2)}-E_{(3),t}$ to any surface in the class  $[\tilde D]$ is the canonical divisor on this surface.
The image in $\PP^3$ of any surface in the class $(d-4)h-E_{(3)}-E_{(2)}-E_{(3),t}$ would necessarily be a surface of degree $d-4$ that contains $\RP$, explaining the name ``adjoint'' for $A_P$. 

We are now ready to prove Proposition~\ref{prop:wachspressVariety}.
For this, we first investigate when the image $\bar A_P := \overline{\omega_P(A_P)}  \subset W_P$ of the adjoint surface $A_P$ is a curve.
\begin{lemma}\label{lem:adjointImageIsCurve}
Let $P \subset \PP^3$ be a full-dimensional  polytope with a simple plane arrangement~$\HP$.
The image $\bar A_P$ of the adjoint surface $A_P$ under the Wachspress map $\omega_P$ is a curve
if and only if $P$ is a triangular prism or a cube.
\end{lemma}

\begin{proof}
We describe in Remark~\ref{rem:wachspressRestricted} and Example~\ref{ex:surface} how the Wachspress map $\omega_P$ restricts to the planes in $\HP$.
Each plane in  $\HP$ is mapped birationally to a surface in $W_P$ and, since the vertices in $P$ are mapped to linearly independent points,  distinct planes are  mapped to distinct surfaces.
The adjoint $A_P$ intersects the strict transform $\tHP$ of $\HP$ in a curve in each component of $\HP$.
So the adjoint is mapped to a curve only if the intersection with each component of $\HP$ is contracted in $W_P$.
The intersection of the adjoint with the strict transform of a plane spanned by a facet of $P$ is trivial if the facet is a triangle.
This intersection is a line contracted in $W_P$ if the facet is quadrangular (due to Example~\ref{ex:surface}).
If the facet has at least five vertices, the intersection with the adjoint is not contracted (again by Example~\ref{ex:surface}).
So the image of the adjoint is a curve only if every facet of $P$ has at most four vertices.

Since there are three facets through each of the $2d-4$ vertices in $P$ (where $d := |\mathcal{F}(P)|$), we have $\sum_{F \in \mathcal{F}(P)} |V(F)| = 6d-12$.
If every facet of $P$ has at most four vertices, this implies $6d-12 \leq 4d$, which is equivalent to $d \leq 6$. 
There are exactly four combinatorial types of simple polytopes with at most six facets (see the first four rows of Table~\ref{tab:polytopes}).
One of them has pentagonal facets, so in that case $\bar A_P$ is a surface.
The adjoint of a tetrahedron is empty.
By Example~\ref{ex:prismCube}, 
the image $\bar A_P$ of the adjoint $A_P$ of a triangular prism or a cube is a curve.
\end{proof}

 \begin{proof}[{Proof of Proposition~\ref{prop:wachspressVariety}}]
  We restrict our computations to a smooth surface $\tilde S$ in the class of $h_{W_P}$ on $X_P$.
  We denote by $h$ also its restriction to $\tilde S$.   Furthermore, we denote by $E,C$ and $E_t$ the restriction to $\tilde S$ of the exceptional divisors $E_{(3)},E_{(2)}$ and $E_{(3),t}$ respectively.
  Clearly, $E$ and $E_t$ are exceptional curves on $\tilde S$, while $C$ is the strict transform of the union of lines in $\RP$.
  %We will compute the degrees and sectional genera of $W_P$ and $\bar A_P$ as intersection numbers on $\tilde S$. 
% Recall that $\tilde S$ itself is by definition the pullback to $X_P$ of a hyperplane section of $X_P$.  
% The class  $h_{W_P}$ of $\tilde S$ on $W_P$ is 
% $$h_{W_P}=(d-3)h-G-F_l-2F_t,$$
Due to~\eqref{eq:classTildeS},
the restriction of $h_{W_P}$ to $\tilde S$ is
 \begin{align}\label{eq:hTildeS}
h_{\tilde S}=(d-3)h-E-C-2E_t.
\end{align}

The following intersection numbers on $\tilde S$ follow from Lemma~\ref{exceptional}:
 \begin{align}\label{eq:intno}
C\cdot E=0, \; E_t \cdot E = 0, \; E^2=-a, \; C\cdot E_t=3c \; {\rm and} \; E_t^2=-2c,
\end{align}
where the latter equality holds because triple points of $\RP$ are quadratic singularities on $S$.
Moreover, we have
\begin{align}\label{eq:hintno}
h\cdot E=h\cdot E_t=0\;, h^2=(d-3)\;{\rm and}\; h\cdot C=\binom{d-3}{2},
\end{align}
 since there are $\binom{d-3}{2}$ lines in $\RP$ (due to Lemma~\ref{lem:residualArrP3}).  
 Notice that $$(C+ E_t)^2= C^2+4c   = (C+ 2E_t)^2.$$
To compute the self-intersection of $C+E_t$, we use the adjunction formula on $\tilde S$. First we compute the arithmetic genus of $C+E_t$ as a union of lines with transverse intersections, using repeatedly the formula for the sectional genus of a reduced and reducible curve: 
\begin{align} \label{eq:genusformula}
g(C_1\cup C_2)=g(C_1)+g(C_2)+C_1\cdot C_2-1.
\end{align}
There are $\binom{d-3}{2}+c$ lines all together in $C+E_t$,   
and  $b+3c$ intersection points between these lines.
So the arithmetic genus is 
$$g(C+E_t)=-\binom{d-3}{2}-c+1+b+3c=b+2c+1-\binom{d-3}{2}.$$
The canonical divisor on $\tilde S$ is $K_{\tilde S}=(h_{W_P}+K_{X_P})|_{\tilde S}=(d-7)h+E$.
By the adjunction formula, we have%says %by adjunction on $\tilde S$ the canonical divisor on $C+E_t$ has degree
\begin{align*}
2g(C+E_t)-2&=(C+E_t)\cdot (C+E_t+K_{\tilde S})\\
&=(C+E_t)^2+(d-7)h\cdot (C+E_t)=(C+E_t)^2+(d-7)\binom{d-3}{2},
\end{align*}
which yields
\begin{align*}
(C+2E_t)^2=(C+E_t)^2&=2g(C+E_t)-2-(d-7)\binom{d-3}{2}\\
&=2(b+2c-\binom{d-3}{2})-(d-7)\binom{d-3}{2}\\
&=2b+4c-(d-5)\binom{d-3}{2}.
\end{align*}
%Now  that the image of $C$ on $S$ consists of $
%Each of the disjoint $(-1)$-curves in $E'$ intersects $C$ in two points.  Therefore 
 %$$h_{\tilde S}\cdot E'=(-2E'-C)\cdot E'=2b-2b=0,$$ so all of these curves are contracted on $\bar S\subset W_P$.
 %Similarly, each of the disjoint curves in $E"$ intersects $C$ in three points, so
 %$h_{\tilde S}\cdot E"=(-2E"-C)\cdot E"=4c-3c=1$.  Hence, each $(-2)$-curve in $E"$ is mapped to a line in $\bar S\subset W_P$.
 %The canonical divisor on $X_P$ is 
 %$K_{X_P}=-4h+2G+2G'+F'+2G"+F$
%so, by adjunction, the canonical divisor on  $\tilde S$ belongs to the class 
% $$K_{\tilde S}=K_{X_P}|_{\tilde S}+h_{\tilde S}=(d-7)h+E+E'.$$
% For the strict transform $L$ of a line on $S$ that passes through no isolated point in $\RP$, but  $\alpha$ double points and $\beta$  triple points, $L\cdot h=1$, $L\cdot E=0$, $L\cdot %E'=\alpha$ and $L\cdot E"=\beta$  so  the adjunction formula says 
% $$L^2+L\cdot  K_{\tilde S}=L^2+(d-7)+\alpha =-2, \quad {\rm i.e.}\quad L^2=(5-d)-\alpha.$$
% The number of incidence (double points, lines) in $\RP$ is $2b$, so since there are $\binom{d-3}{2}$ lines in $\RP$ we get 
% $$C^2=(5-d)\binom{d-3}{2}-2b.$$
 %With $h^2=d-3$ and $h\cdot (C+E_t)=\binom{d-3}{2}$ on $\tilde S$ we
 We may now compute the degree of $W_P$  as the degree of the image in $W_P$ of $\tilde S$:
 \begin{align*}
 {\rm deg}\;W_P=h_{\tilde S}^2&=((d-3)h-E-C-2E_t)^2\\&
 =(d-3)^3-2(d-3)\binom{d-3}{2}-a+2b+4c-(d-5)\binom{d-3}{2}\\
 &=2b+4c-a-\frac{1}{2}(d-3)(d^2-11d+26).
 \end{align*}
Due to~\eqref{eq:classTildeAdjoint}
 the degree of $\bar A_P$ is computed as
  \begin{align*}{\rm deg}\;\bar A_P&=h_{\tilde S}\cdot  [\tilde A_P]|_{\tilde S}=((d-3)h-E-C-2E_t)\cdot ((d-4)h-E-C-2E_t)\\
 &=2b+4c-a-\frac{1}{2}(d-3)(d-4)(d-6).
  \end{align*}
This degree is zero if $\bar A_P$ is not a surface (see Table~\ref{tab:polytopes} and Lemma~\ref{lem:adjointImageIsCurve}).
The second expressions of the degree formulas for $W_P$ and $\bar A_P$ found in Proposition~\ref{prop:wachspressVariety}
are obtained by plugging the relation between $a$, $b$ and $c$ in Lemma~\ref{lem:residualArrP3} into the above formulas.

The sectional genus of $W_P$ coincides with the sectional genus of $\tilde S$ which by the adjunction formula is
 \begin{align*}
 g(W_P)=g(\tilde S)&=1+\frac{1}{2}(h_{\tilde S}\cdot(h_{\tilde S}+K_{\tilde S}))\\
 &=1+\frac{1}{2}((d-3)h-E-C-2E_t)((2d-10)h-C-2E_t)\\
&=1+b+2c+\frac{1}{2}(d-3)(d-6).
 \end{align*}
 Similarly the genus of $\bar A_P$, when it is a surface, is
  \begin{align*}
  g(\bar A_P)&=1+\frac{1}{2}([\tilde A_P]|_{\tilde S}\cdot ([\tilde A_P]|_{\tilde S}+K_{\tilde S})\\
 &= 1 + \frac12   ((d-4)h-E-C-2E_t)\cdot ((2d-11)h-C-2E_t)\\
  &=1+b+2c-\frac{1}{2}(d-3)(d-4).
 \end{align*}

For the degree and sectional genus of any surface $\bar D$ we use the fact that these surfaces are linearly equivalent to $\bar{\mathcal{H}}_P := \overline{\omega_P(\HP)}$ on $W_P$.
Therefore the degree of $\bar D$ equals the degree of $\bar{\mathcal{H}}_P$,  which is $\sum_{F \in \mathcal{F}(P)} (\binom{|V(F)|-2}{2}+1)$ by Example \ref{ex:surface}.
Since $\bar D$ is the image of $\tilde D$ in $W_P$,  this degree may also be computed in $\tilde S$ in terms of the numbers $(a,b,c)$ using \eqref{eq:classTildeD},\eqref{eq:hTildeS}, \eqref{eq:intno} and \eqref{eq:hintno}: 
 $${\rm deg}\;\bar D\!=\!h_{\tilde S}\cdot [\tilde D]\!=\!h_{\tilde S}\cdot(dh-3E-2C-3E_t)=4b+9c-3a-\frac{1}{2}(d-3)(3d^2\!-30d+64).$$

The sectional genus may be computed on $\bar{\mathcal{H}}_P$:  The sectional genus of a component of an $e$-gon facet is $\binom{e-3}{2}$.
Each line through an edge of $P$ is mapped by $\omega_P$ to a line and the simple polytope $P$ has $3(d-2)$ edges.
Therefore, using the formula (\ref{eq:genusformula}),  %$g(C_1\cup C_2)=g(C_1)+g(C_2)+C_1\cdot C_2-1$,  
the sectional genus of $\bar{\mathcal{H}}_P$ is
  $$g(\bar{\mathcal{H}}_P)=\!\sum_{F \in \mathcal{F}(P)}\binom{|V(F)|-3}{2}+ 3(d-2)-(d-1)=\!\sum_{F \in \mathcal{F}(P)}\binom{|V(F)|-3}{2}+2d-5.$$
  Again, in terms of the numbers $(a,b,c)$ this genus is
\begin{align*}
g(\bar D)&=1+\frac{1}{2}(dh-3E-2C-3E_t)((2d-7)h-2E-2C-3E_t)\\
&=1+4b+9c-3a-\frac{1}{2}(d-3)(3d^2-30d+68 ).
\qedhere
  \end{align*}
 %The canonical curve on $\tilde P$ consists of lines from each $4$-gon facet, and conics from each $5$-gon facet of $P$.  By $\omega_P$ the lines are contracted, while the conics are mapped to lines in $P$.  In fact this may be seen by the restriction of $\omega_P$ to the adjoint surface $A_P$, so by equivalence, the canonical curve on $\tilde D$ is mapped similarly to $\bar D$.
  %The canonical curve on $\bar D$ is the intersection $\bar D\cap \bar A_P$.  When $P$ has no $6$-gon facets, the components of the canonical curve are all rational. On $\bar A_P$ they are equivalent to $\bar A_P\cap \bar P$.  We list the components in a table below.
\end{proof}

We conclude this section with a description how the extended Wachspress map $\tilde \omega_P: X_P \to W_P$
restricts to the components of the exceptional divisor described in Lemma~\ref{exceptional}.

 \begin{remark}  The restriction of the Wachspress map $\tilde \omega_P$ to the exceptional divisor on $X_P$ maps the components over isolated points and triple points in $\RP$ to planes, and the components over lines to scrolls in $W_P$.
  \hfill$\diamondsuit$
 \end{remark}

\section{Polytopal hypersurfaces}
\label{sec:deformations}

In this section we consider the smooth blowup $\pi_P^s: X_P^s \to \PP^n$.
We first prove Proposition~\ref{prop:relateAdjoints}, before we aim to classify which polytopes admit smooth strict transforms of their polytopal hypersurfaces.
  
\begin{proof}[{Proof of Proposition~\ref{prop:relateAdjoints}}]
  We denote by $E_c$ the class of the strict transform in $X_P^s$ of the exceptional divisor over $\RPc$ and by $E_{c,s}$ the class of the strict transform in $X_P^s$ of the exceptional divisor over $\RPcs$.
 Let $h$ be the class of the strict transform in $X_P^s$ of a general hyperplane in $\PP^n$.  The canonical divisor on $X_P^s$ is then
 $$K_{X_P^s}=-(n+1)h+\sum_c(c-1)E_{c}+\sum_c (c-1) E_{c,s},$$
while the strict transform of $\HP$ is
  $$[\tHP]=dh-\sum_ccE_{c}-\sum_c cE_{c,s},$$
  since $\HP$ has multiplicity $c$ along any component of both  $\RPc$ and $\RPcs$.  However, 
  any polytopal hypersurface $D \in \Gamma_P$ also has multiplicity $c$ along $\RPc$, but it may have lower multiplicity along some of the components of $\RPcs$. 
  Therefore the strict transform belongs to the class
  $$[\tilde D]=dh-\sum_ccE_{c}-\sum_c cE_{c,s}+ \sum_i a_{D,i}F_{i},$$
  where each $F_i$ is a component of $E_{c,s}$ for some $c$ and $a_{D,i}\geq 0$.  The adjunction formula on $X_P^s$ says that the  class of the canonical divisor on the strict transform $\tilde D$ is the restriction to $\tilde D$ of 
 $$[\tilde D]+K_{X_P^s}=(d-n-1)h-\sum_cE_{c}-\sum_c E_{c,s}+ \sum_i a_{D,i}F_{i}.$$
   The  adjoint hypersurface $A_P$ has multiplicity one along each component of $\RPc$ and multiplicity at least one along each component of $\RPcs$, so the strict transform  belongs to the class
   $$[\tilde A_P]=(d-n-1)h-\sum_cE_{c}-\sum_c E_{c,s}- \sum_i b_{A,i}F_{i},$$
   where $b_{A,i}\geq 0$.  Thus
\begin{align}\label{eq:adjointPolytopalHS}
   [\tilde D]+K_{X_P^s}=[\tilde A_P]+\sum_i (a_{D,i}+ b_{A,i})F_i
\end{align}
   and each $F_i$ is mapped by $\pi_P^s$ to the singular locus of $\RP$.
As the adjoint $A_P$ of the polytope $P$ is unique, by Theorem~\ref{thm:adjointIntro},
the strict transform $\tilde{A}_P$ is the only hypersurface in $X_P^s$ whose class equals $[\tilde{A}_P]$.
Together with~\eqref{eq:adjointPolytopalHS}, this implies that the strict transform $\tilde D$ has a unique canonical divisor, the restriction of  $\tilde{A}_P+\sum (a_{D,i}+ b_{A,i})F_i$ to $\tilde D$.
%adjoint in $X_P$, namely $\tilde{A_P}$.
%Since $X_P$ is the blow up of a projective space, the polytopal hypersurface $\tilde D$ has a unique canonical divisor, which is its unique adjoint $\tilde{A_P}$ in $X_P$ restricted to $\tilde D$.
\end{proof}

If the strict transform $\tilde D$ of a divisor $D \in \Gamma_P$ is smooth, 
then the polytopal hypersurface $D$ has to be irreducible.
We show now that many polytopes $P$ do not admit an irreducible polytopal hypersurface $D \in \Gamma_P$.
  
  \begin{lemma}\label{irred}
Let $P\subset \PP^n$ be a full-dimensional polytope with $d$ facets and a simple hyperplane arrangement $\HP$.
Moreover, let $F$ be a facet of $P$ which has $e$ facets.
  If $n<e<d-n-1$ or $n = e < d-n-2$, the hyperplane $H_F$ spanned by $F$  is a fixed component in $\Gamma_P$.  
%$d>8$ facets, or $P$ has a triangular facet and $d>7$, then every surface in $\Gamma_P$ is reducible. In particular, when $d>e+3$, any 
\end{lemma}
\begin{proof}
 By Corollary \ref{cor:residualRestricted}, the intersection $\RP\cap H_F$ is the union of $d-e-1$ hyperplanes in $H_F$ and $\RF$, the residual arrangement of $F$.
Any hypersurface $D$ in $\Gamma_P$ must be singular along the $d-e-1$ hyperplanes in $H_F$ and have multiplicity  at least two along $\RF$. 
  Therefore, if $D$ does not contain $H_F$, then $D \cap H_F$ is a hypersurface in $H_F$ which contains the $d-e-1$ double hyperplanes. 
The residual hypersurface in $H_F$ has degree $d- 2(d-e-1)=2e-d+2$ and is singular along $\RF$.
If $e > n$, the facet $F$ is not a simplex and its residual arrangement $\mathcal{R}_F$ is not empty.
  Since the adjoint hypersurface $A_F$ is the unique hypersurface in $H_F$ of degree $e-n$ that contains $\RF$ (by Theorem~\ref{thm:adjointIntro}), any hypersurface that is singular along 
the non-empty residual arrangement $\RF$ must have degree at least $e-n+1$.
  Thus, if $e > n$, we have $2e-d+2\geq e-n+1$, i.e. $e\geq d-n-1$. 
If $e = n$, so if $F$ is a simplex, we can only conclude that $2e-d+2 \geq 0$, i.e. $e \geq d-n-2$.  
%  If $e>3$, the $e$-gon has a unique adjoint,  a unique plane curve of degree $e-3$ through the $\binom{e}{2}-e$ isolated points.  The residual curve must therefore have degree at least $e-2$.    Therefore $d\leq e+4$.  If $e=3$, the residual curve has degree $8-d\geq 0$.
  \end{proof}

This lemma suffices to find all polytopes in $\PP^2$ and in $\PP^3$ which have polytopal hypersurfaces with smooth strict transforms under $\pi_P^s$.

\begin{proof}[{Proof of Proposition~\ref{prop:TypesInDim2}}]
If $P$ is a $d$-gon, every facet of $P$ has exactly two facets.
Hence, if $d > 6$, then Lemma~\ref{irred} implies that the only divisor in $\Gamma_P$ is $\HP$.
Since $\HP$ is reducible, its strict transform $\tHP$ cannot be smooth.
 
For all $d\leq 6$, the linear system $\Gamma_P$ has positive dimension and  base points only at  the points in $\RP$, so by Bertini's theorem a general polygonal curve $D \in \Gamma_P$ is irreducible  with simple nodes exactly at the points in $\RP$.  The strict transform $\tilde D$ is a smooth elliptic curve.
\end{proof}

\begin{corollary} \label{cor:irredP3}
Let $P$ be a full-dimensional polytope in $\PP^3$ with a simple plane arrangement~$\HP$
and an irreducible polytopal surface in $\Gamma_P$.
The combinatorial type of $P$ is one of the nine combinatorial types in Table~\ref{tab:polytopes}.
\end{corollary}
\begin{proof} 
Let $d$ denote the number of facets of $P$.
Since there are three facets through each of the $2d-4$ vertices in $P$, we have $\sum_{F \in \mathcal{F}(P)} |V(F)| = 6d-12$. 
By Lemma~\ref{irred}, since $\Gamma_P$ contains an irreducible element, each facet of $P$ has at least  $d-5$ edges. 
Therefore $d(d-5)\leq \sum_{F \in \mathcal{F}(P)} |V(F)| = 6d-12$, which implies $d\leq 9$.  
%So if $P$ has degree $\geq 9$  it must have $e$-gon facets with $e< 6$.   The planes of these facets are, by Lemma \ref{irred}, components of any surface in $\Gamma_P$.  
If $d=9$, Lemma~\ref{irred} further implies that $P$ cannot have triangular or quadrangular facets.
Thus each of the nine facets has at least five edges, 
so $9 \cdot 5 \leq \sum_{F \in \mathcal{F}(P)} |V(F)| = 6\cdot 9-12$, a contradiction.
Hence, we have shown that $d \leq 8$.

All combinatorial types of simple three-dimensional polytopes with at most six facets are depicted in Table~\ref{tab:polytopes}.
Hence, we have to investigate the cases that $P$ has seven or eight facets.

Let us first assume that  $d=8$.
There are exactly $14$ combinatorial types of simple three-
\linebreak[4]
dimensional polytopes with eight facets.
Two of these are depicted in the last two rows of Table~\ref{tab:polytopes}.
The other $12$ have a triangular facet. 
We will show now that these $12$ combinatorial types do not admit an irreducible divisor $D \in \Gamma_P$.
For this, we assume for contradiction that $P$ has a triangular facet. 
At least one of the triangular facets of $P$ has a neighboring facet $F$ with $e := |V(F)| \leq 6$.
Due to the triangular facet, the residual arrangement $\RP$ has an isolated point $p$, which is contained in the plane $H_F$ spanned by the facet $F$.
By Corollary~\ref{cor:residualRestricted}, the intersection $\RP \cap H_F$ consists of $7-e$ lines and the residual arrangement $\mathcal{R}_F$ of the facet $F$ which is $\binom{e}{2}-e$ points, among them $p$.
An irreducible polytopal surface $D \in \Gamma_P$ intersects $H_F$ in a curve of degree eight, which is double along the $7-e$ lines.
The residual curve $C$ in $H_F$ has degree $8-2(7-e) = 2(e-3)$ and multiplicity at least two at the $\binom{e}{2}-e$ points of $\mathcal{R}_F$.
At the isolated point $p$, the curve $C$ must even have multiplicity three.
We observe that $e \neq 3$ (since $\mathcal{R}_F \neq \emptyset$).
If $e \in \lbrace 4,5 \rbrace$, there is a unique curve of degree $2(e-3)$ with multiplicity two at the $\binom{e}{2}-e$ points (namely the adjoint $A_F$ doubled), but this curve does not have multiplicity three at $p$.
If $e=6$, there is a pencil of curves of degree $2(e-3)$ with multiplicity two at the nine points of $\mathcal{R}_F$, but none with the additional multiplicity three at $p$.
This contradicts the construction of the curve $C$.

Let us now assume that $d=7$.
There are exactly five combinatorial types of simple three-dimensional polytopes with seven facets.
Three of them are shown in Table~\ref{tab:polytopes}.
We will show in the following that the remaining two combinatorial types do not admit an irreducible divisor $D \in \Gamma_P$.
Both combinatorial types are depicted in Figure~\ref{fig:combType7}.
In both cases there is a hexagonal facet $F$.
The plane $H_F$ of this facet contains three isolated points $p_1, p_2, p_3$ of $\RP$ that are not collinear.  
In fact the plane $H_F$ intersects $\RP$ in exactly six points $q_1, \ldots, q_6$ on lines in $\RP$ in addition to the isolated points.
Through the nine points there is an irreducible plane cubic curve $C$.  
Any irreducible polytopal surface  $D \in \Gamma_P$ would intersect the plane $H_F$ in a curve $C_D$ of degree seven with double points at $q_1, \ldots, q_6$ and triple points at $p_1, p_2, p_3$.
We consider the intersection on $C$ with $C_D$  and with the union $C_L$ of the six lines spanned by the edges of the hexagon $F$.
The latter is a curve with a node at each of the nine points, so on $C$ the intersection with $C_L$ forms a divisor of degree $18$ that is a complete intersection.
Unless $C$ is a component of $C_D$, their interection forms a divisor of degree $21$ that is also a complete intersection.
The three isolated points $p_1,p_2,p_3$ form a divisor on $C$ that is precisely the difference between these two divisors.
It is therefore also a complete intersection, i.e. the intersection of $C$ with a line, which contradicts the observation that $p_1, p_2, p_3$ are not collinear.
If $C$ is a component of $C_D$, then we may use a similar argument for $C_D\setminus C$, using the fact that $q_1, \ldots, q_6$ do not lie on a conic.  
These contradictions imply that the plane $H_F$ must be a component of any polytopal surface.
\end{proof}

\begin{figure}
\centering
\includegraphics[width=0.25\textwidth]{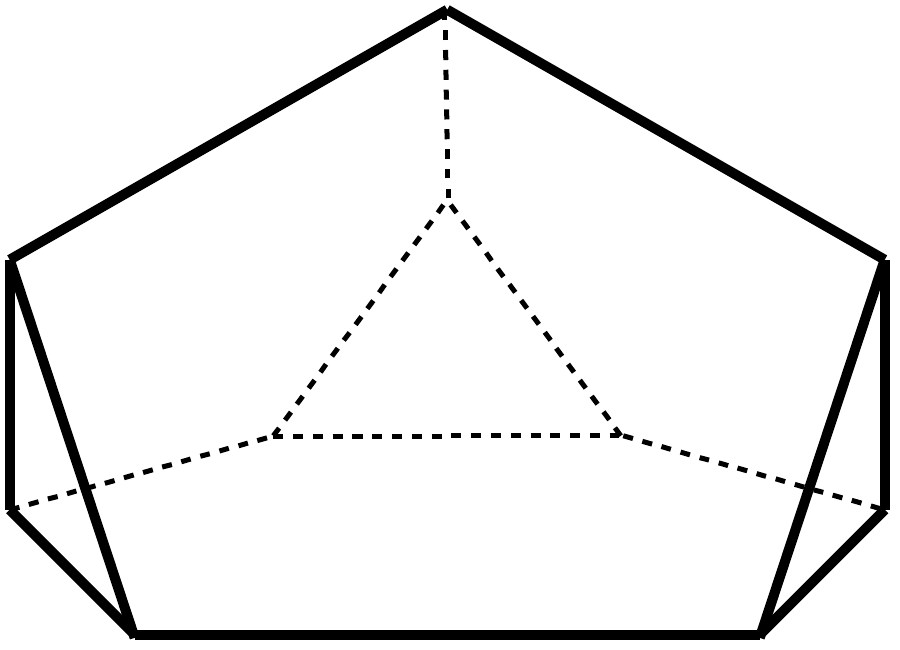}
\hspace*{35mm}
\includegraphics[width=0.25\textwidth]{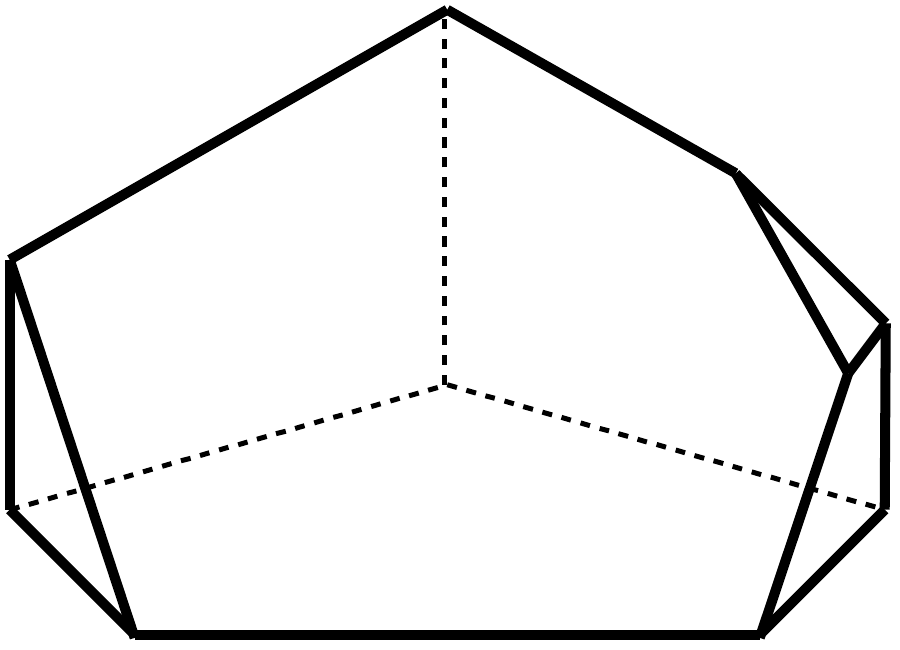}
\caption{Combinatorial types of simple polytopes in $3$-space with $7$ facets not shown in Table~\ref{tab:polytopes}.
Their vectors of facet sizes are $(6,5,5,5,3,3,3)$ and $(6,5,5,4,4,3,3)$.}
\label{fig:combType7}
\end{figure}

\begin{remark} \label{rem:smoothPolytSurfInTable}
We verified
with symbolic computations in \texttt{Macaulay2}~\cite{m2}
that,
for each combinatorial type $\mathcal{C}$ in Table~\ref{tab:polytopes},
the general polytopal surface $D \in \Gamma_P$ of a general polytope $P$ of type $\mathcal{C}$ is irreducible
and its strict transform $\tilde D$ in $X_P^s$ is smooth.
 \hfill$\diamondsuit$
\end{remark}

Corollary~\ref{cor:irredP3} and Remark~\ref{rem:smoothPolytSurfInTable} show that
the combinatorial types of three-dimensional simple polytopes which admit smooth strict transforms of polytopal surfaces
are exactly the combinatorial types listed in Table~\ref{tab:polytopes}.
Our next task is to determine the birational type of the general polytopal surfaces associated to these combinatoral types.
This is immediate for tetrahedra. 
For a general polytope $P$ of one of the other combinatorial types, we determine the canonical curves on its polytopal surfaces, starting with the canonical curve on the strict transform  $\tHP$ of  $\HP$.

\begin{example}When $P$ is a tetrahedron, its residual arrangement $\RP$ is empty.
Any quartic surface in $\PP^3$ belongs to $\Gamma_P$, so $\dim \Gamma_P = 34$.
\hfill$\diamondsuit$
\end{example}

\begin{lemma} \label{canonicalcurve} 
Let $P \subset \PP^3$ be a full-dimensional polytope with $d$ facets and a simple plane arrangement $\HP$.
%Assume that the vector of facets sizes is $(e_1,...,e_d)$.
On $\tHP$, the canonical curve $K_{\tHP}$ consists of exceptional curves and strict transforms of curves in the planes of~$\HP$. 
 The former are all contracted by the Wachspress map, while the latter are the strict transforms of the unique adjoint curve of the polygonal facet in each plane.
 These adjoint curves have degree $e-3$ in the plane of an $e$-gon, while their images by the Wachspress map are pairwise disjoint and have degree $\binom{e-3}{2}$ respectively.
 In particular, the degree of the image of $K_{\tHP}$ by the Wachspress map is 
$\sum_{F \in \mathcal{F}(P)}\binom{|V(F)|-3}{2}.$
%=4b+9c-3a-\frac{3}{2}(d-3)(d-4)(d-6).$$
%Its component in the plane of an $e$-gon is the unique adjoint curve to the $e$-gon.  
%The plane curve components of $K_{\tHP}$ are pairwise disjoint.
%Furthermore, the bicanonical system $|2K_{\tHP}|$  has a fixed curve of degree $2(e-3)$  in the plane of each $e$-gon facet in $P$ with $e \in \lbrace 4,5 \rbrace$ and a pencil of elliptic curves of degree $6$ in each plane of a hexagonal  facet.
\end{lemma}
\begin{proof}

%The canonical divisor on $\tHP$ is 
%$$K_{\tHP}\cong-4h+\sum_p 2F_p+\sum F_l+\sum 2F_t$$ and its restriction to 
The  blowup $\pi_P$ resolves the base locus of the Wachspress map $\omega_P$.
Thus the exceptional locus of the morphism $X_P^s\to X_P$ is contracted by the Wachspress morphism $\tilde \omega_P$.
So for the lemma it suffices to consider the strict transform of $H_P$ and its components in~$X_P$.

We consider the strict transform $\tilde H_F$ of a plane $H_F$ in $\HP$ spanned by an $e$-gon facet $F$.
Let $E_p$ be the restriction of the exceptional divisor over an isolated point $p\in \RP$ and
let $E_\ell$ be the restriction of the exceptional divisor that dominates a line $\ell\subset \RP$.
We denote by $h$ the pullback to $X_P$ of the class of a plane in $\PP^3$.
We observe that
$\tilde H_F$ is transverse to the strict transform of the other planes along the strict transform $\tilde{\mathcal{H}}_F$ of the union $\mathcal{H}_F$ of the lines spanned by the edges of the $e$-gon $F$.
The curve $\tilde{\mathcal{H}}_F$ belongs to the class $eh-\sum_{p\in H_F} 2E_p -\sum_{\ell\not\subset H_F}2E_\ell$ restricted to $\tilde H_F$.
The class of the canonical divisor on $\tilde H_F$ is $-3h+\sum_{p\in H_F} E_p+\sum_{\ell\not\subset H_F}E_\ell$.
The restriction  $K_{\tHP,F}$ of the canonical divisor on $\tHP$ to $\tilde H_F$ is equivalent to the sum of the canonical divisor on $\tilde H_F$ and the intersection $\tilde{\mathcal{H}}_F$ of $\tilde H_F$ with the residual $\tHP\setminus \tilde H_F$.  
%The degree of the canonical curve $K_{\tHP}$ is computed as an intersection number on $X_P$. It is 
%$$h_{W_P}\cdot K_{\tHP}=h_{W_P}\cdot \tilde D\cdot \tilde A_P=d(d-4)-h_{W_P}\cdot \sum_k 2G_k^2=d(d-4)- 2\binom{d-3}{2}=3(d-4).$$
%The canonical divisor on $\tilde H_e$ is the restriction of $-3h_{W_P}+\sum_i E_i+\sum_j F_j$.  The restriction  $K_{\tHP,e}$ of the canonical divisor on $\tHP$ to $\tilde H_e$ is equivalent to the sum of the canonical divisor on $\tilde H_e$ and the intersection $\tilde C_e$ of $\tilde H_e$ with the residual $\tHP\setminus \tilde H_e$.   
So
$$K_{\tHP,F}\cong(e-3)h
-\sum_{\substack{\text{isol. \!pt.} \\ p \in \RP, \\  p\in H_F}} E_p
-\sum_{\substack{\text{line} \\ \ell \subset  \RP, \\ \ell\notin H_F}}E_\ell,$$ 
which is the class of the strict transform $\tilde A_F$ of the unique adjoint curve $A_F$ of the $e$-gon $F$. 
Thus, the canonical curve on $\tHP$ is the union of the strict transforms of the adjoint curves of the facets of $P$.
Due to Remark~\ref{rem:wachspressRestricted} and Example~\ref{ex:surface}, the degree of the image of the canonical curve by the Wachspress map is $\sum_{F \in \mathcal{F}(P)}\binom{|V(F)|-3}{2}.$
%The degree may also be computed by the adjunction, 
%$h_{W_P}\cdot K_{\bar{\mathcal{H}}_P}=2 g(\bar{\mathcal{H}}_P)-2-{\rm deg}(\bar{\mathcal{H}}_P)$ together with the formulas from  Proposition \ref{prop:wachspressVariety}.

It is enough to show that the strict transforms of the adjoint curves are pairwise disjoint.
For this, we consider two facets $F$ and $F'$ of $P$ with $e$ and $e'$ vertices, respectively.
The planes $H_F$ and $H_{F'}$ spanned by these facets intersect in a line $L$.

We assume first that the line $L$ contains an edge of $P$.
In that case, the adjoint $A_F$ intersects $L$ precisely in the $e-3$ points of intersection on $L$ with the other lines in $\mathcal{H}_F$ which are not spanned by the two neighboring edges in $F$ of the edge spanning $L$.
The analogous statement holds for $A_{F'}$. 
Hence, the adjoint curves $A_F$ and $A_{F'}$ intersect exactly at the isolated points of $\RP$ which are contained in $L$.
These points are blown up on $X_P$.
Since  $H_F$ and $H_{F'}$ intersect transversally along $L$, the adjoints $A_F$ and $A_{F'}$ meet transversally along their points of intersection, so their strict transforms do not intersect on $X_P$. 
 
 If $L$ does not contain any edge of $P$, then $L$ is a line in the residual arrangement $\RP$ and is blown up on $X_P$.  
Since  $H_F$ and $H_{F'}$ meet transversally along $L$, their strict transforms do not intersect on  $X_P$. 
Therefore  the two curves $\tilde A_F$ and $\tilde A_{F'}$ are also disjoint.
%The bicanonical system restricts to each plane of $\tHP$  as the class of twice the unique adjoint curve.  When $e=3,4,5$ twice the adjoint curve is still unique in its linear system.  In the hexagonal case, $e=6$,  the adjoint curve is a plane cubic passing through nine residual points of intersection of the $6$-lines in the hexagon. Twice  this curve is of course singular in the $9$ points, but so is also the hexagon itself, so there is at least a pencil of sextic plane curves singular at the nine points.  Since the adjoint curve is unique, there is at most a pencil of such sextics, so the lemma follows.
\end{proof}

%Any smooth surface $\tilde D$ 

%Let $D$ be any surface in $\Gamma_P$. 
%The canonical curve $K_{\tilde D}$ on $\tilde D$ is linearly equivalent to $K_{\tHP}$ on the unique common adjoint $\tilde A_P$.    The arithmetic genus $g(K_{\tilde D})$ of $K_{\tilde D}$ is computed by adjunction on $X$:
%$$2g(K_{\tilde D})-2=\tilde A_P\cdot \tilde D\cdot(\tilde D+\tilde A_P+K_X)=2d(d-4)^2+4\binom{d-3}{2} G_i^2-(6d-16)\binom{d-3}{2}$$
%If $A_P$ is smooth, a line on $A_P$ has, by adjunction, selfintersection $G_i^2=-2+8-d=6-d$ so in that case $2g(K_{\tilde D})-2=-3(d-4)(d^2-9d+20)$.

%Any surface $D\in \Gamma_P$ has a canonical curve that decomposes into a plane curve of degree $(e-3)$ for each $e$-gon facet of $P$, since these are linearly equivalent curves on the adjoint $A_P$.  
\begin{corollary}\label{K3orelliptic} 
Let $P \subset \PP^3$ be a full-dimensional polytope with a simple plane arrangement~$\HP$
and a smooth strict transform $\tilde D$ in $X_P^s$ of a polytopal surface $D \in \Gamma_P$.
%Let $D\in \Gamma_P$ and assume that the strict transform $\tilde D\subset X_P$ is a smooth surface.  Then $P$ has no $e$-gon facets with $e>6$.
If $P$ has a hexagonal facet, then the general such $D$ is an elliptic surface.
If $P$ has only triangular, quadrangular and pentagonal facets, then general such $D$ is a $K3$-surface.
\end{corollary}
\begin{proof} 
We first show that $\tilde D$ is regular, i.e. $h^1({\cal O}_{\tilde D}(K_{\tilde D}))=0$.
Since %$h^1({\cal O}_{\tilde D}(K_{\tilde D})=0$. i.e. if $\tilde D$ is a smooth surface, it is regular:
 $h^1({\cal O}_{X_P^s}(K_{X_P^s}))=h^2({\cal O}_{X_P^s}(K_{X_P^s}))=0$, we deduce from the cohomology of the exact sequence
\begin{align*}
\begin{array}{ccccccccc}
0 & \longrightarrow & \mathcal{O}_{X_P^s} (K_{X_P^s}) & \longrightarrow & \mathcal{O}_{X_P^s}(\tilde A_P) & {\longrightarrow} & \mathcal{O}_{\tilde D}(K_{\tilde D}) & \longrightarrow & 0,\\
\hspace*{1mm} \\
\end{array}
\end{align*}
that $h^1({\cal O}_{\tilde D}(K_{\tilde D}))=h^1({\cal O}_{X_P^s}(\tilde A_P))$. 
But  $h^1({\cal O}_{X_P^s}(\tilde A_P))=h^1 (\PP^3, \mathcal{I}_{\RP}(d-4))=0$ by Theorem~\ref{thm:adjoint}, so $\tilde D$ is regular.

Next, the smooth surface $\tilde D$ has a unique canonical curve, so $\chi(\tilde D)=2$.   
If $P$ has no hexagonal facets, then, by Lemma \ref{canonicalcurve} and Theorem~\ref{thm:TypesInDim3}, the canonical curve $K_{\tHP}$  is the union of exceptional rational curves contracted by the Wachspress map and rational curves mapped to lines,
hence so is also $K_{\tilde D}$.
These curves must be exceptional curves of the first kind on $\tilde D$, i.e. a minimal model of $\tilde D$ is a $K3$-surface.
If $P$ has  a hexagonal facet, then, by Lemma \ref{canonicalcurve}, $K_{\tHP}$ contains in addition plane cubic curves, hence so does also the  canonical curve on $\tilde D$.
In particular, a minimal model of $\tilde D$ is a regular elliptic surface.
\end{proof} 

\begin{proof}[{Proof of Theorem~\ref{thm:TypesInDim3}}]
If the strict transform $\tilde D$ in $X_P^s$ of a polytopal surface $D \in \Gamma_P$ is smooth, 
then $D$ is, of course, irreducible. 
In that case, Corollary~\ref{cor:irredP3} implies that the combinatorial type of $P$ is listed in Table~\ref{tab:polytopes}.

By Remark~\ref{rem:smoothPolytSurfInTable} every combinatorial type in Table~\ref{tab:polytopes} admits a smooth strict transform $\tilde D$ in $X_P^s$ of a polytopal surface $D \in \Gamma_P$.
Their birational types are determined by Corollary~\ref{K3orelliptic}.
\end{proof}

In the remainder of this article, we provide more detailed descriptions of the Wachspress maps and polytopal surfaces associated to general polytopes of the combinatorial types listed in Table~\ref{tab:polytopes}. 
We use the notation from Proposition~\ref{prop:wachspressVariety} in what follows.
 
 \subsection{Prisms}
 \label{ssec:prisms}
 
  \begin{example}  
In Example \ref{ex:prismCube} we already described $\RP$, $A_P$ and $\omega_P$ for triangular prisms $(d=5)$ and perturbed cubes $(d=6)$.

When $P$ is a triangular prism, then $\dim \Gamma_P = 23$,  $W_P=\PP^1\times \PP^2\subset\PP^5$ in the Segre embedding and the image of the adjoint plane is a line $L$.  
%there are no such lines, each pair of quadrangular facets share a $1$-dimensional face. The planes of the three quadrangular facets intersect in an isolated point in $\RP$, while the planes of the triangular facets intersect in a line, as explained in Example \ref{firstex}.  
%   $d=5$ and  dim\;$\Gamma_P=23$.   The variety  
Let $D$ be a general polytopal surface for $P$. Then $\bar D$ is a $K3$-surface of degree $8$ and genus $5$ with Picard rank $2$: it is the intersection of the Segre threefold with a quadric hypersurface. The surface $D$ is the projection of $\bar D$ from three collinear points, namely the intersection of $\bar D$ with the line $L$. 

 When $P$ is a cube ($d=6$), then $\dim \Gamma_P = 26$ and $W_P = \PP^1\times \PP^1\times \PP^1 \subset \PP^7$ in the Segre embedding.
 For a general polytopal surface $D$, the image $\bar D\subset W_P$  is a $K3$-surface of degree $12$ and genus $7$: it is the intersection of $W_P$ with a quadric hypersurface. The image of the adjoint quadric surface is a twisted cubic curve that intersects $\bar D$ in $6$ points.   The surface $D$ is the projection of $\bar D$ from these $6$ points.
 \hfill$\diamondsuit$
 \end{example} 

From here on we assume that $d>6$.
Any two planes of quadrangular facets in the prism that do not share an edge, intersect in a line that belongs to the residual arrangment $\RP$.  
 In fact there are $(d-2)(d-5)/2$ such lines.  
The two planes of $(d-2)$-gon facets of the prism intersect in the remaining line in $\RP$, that we denote by $L_P$.
There are no isolated points in $\RP$ and the line $L_P$ is disjoint from the other lines in $\RP$.  
The double and triple points of $\RP$ therefore all lie on the intersections of three planes of quadrangular facets. 
If exactly two of the facets share an edge in $P$, there are two lines in $\RP$ through the point, 
and if no two of the three facets share an edge, there are three lines through the point that span~$\PP^3$.
The number of double points in $\RP$ is therefore $b=(d-2)(d-6)$, while the number of triple points is $c=\binom{d-2}{3}-(d-2)-(d-2)(d-6)=(d-2)(d-6)(d-7)/6$.
%When $d>5$, the singular points on $C$ appear as triple intersections in $\HP$ where all three planes contains quadrangular facets, and at most two of the three planes share a face of $P$.  If  two of the planes share a face in $P$, there are two lines in $C$ through the point, and if no two of the three planes share a face, there are three lines through the point that span $\PP^3$.
In $X_P$ there is a pencil $|B_P|$ of surfaces: the strict transforms of planes through the line $L_P$. 
The  planes of the $d-2$ quadrangular facets of $P$ cut each plane through $L_P$ in a $(d-2)$-gon. 
So the restriction of $\omega_P$ to any of the surfaces $B_P$ is the Wachspress map with respect to that polygon.
By Example~\ref{ex:surface}, the image $\bar B_P$ of such a surface has degree $\binom{d-4}{2}+1$ and spans a $\PP^{d-3}$, 
which means that $W_P$ lies in a rational normal $\PP^{d-3}$ scroll  in $\PP^{N-1}$, where $N = 2(d-2)$.
In particular the degree of this scroll is one more than its codimension, i.e. $N-d+3$, and its ideal is generated by $\binom{N-d+3}{2}$ quadrics.
%   The image of the adjoint surface $A_P$ has degree $\frac{1}{6}(d+1)(d-4)(d-6)$.
%\begin{example}
% When $P$ is a cube ($d=6$). The $3$-fold $W_P\subset \PP^7$ lies in three $\PP^3$-scrolls over $\PP^1$.  In fact $\RP$ is three skew lines, and $W_P$ is the $\PP^1\times \PP^1\times \PP^1$ in its Segre embedding in $\PP^7$.   Furthermore the general surface $\bar D$ in $\Gamma_P$ is the intersection of $W_P$ with a quadric hypersurface, so dim\;$\Gamma_P=26$,  and  $\bar A_P$ is a twisted cubic curve in a $\PP^3$ that intersects $\bar D$ in six points.  
 % $\bar D$ is a $K3$-surface of degree $12$ and genus $7$ with Picard rank $3$, it is the intersection of the threefold  $\PP^1\times\PP^1\times \PP^1$ with a quadric hypersurface.  The surface $D$ is the projection of $\bar D$ from six points that span a $\PP^3$.

% \end{example} 
 \begin{example}
When $P$ is a pentagonal prism ($d=7$; see row $6$ in Table~\ref{tab:polytopes}), then 
computations with \texttt{Macaulay2} show that 
dim\;$\Gamma_P=12$ and that the ideal of $W_P\subset\PP^9$ is arithmetically Gorenstein, generated by $15$ quadrics.  
$\bar A_P$ is a rational quartic scroll in a $\PP^5$.  Indeed, the $3$-fold $W_P\subset \PP^9$ lies in a $\PP^4$-scroll over $\PP^1$.  Each $\PP^4$ intersects $W_P$ in a surface $\bar B_P$, in this case a quartic Del Pezzo surface.  There are no isolated points and no triple points in $\RP$, so there is an equivalence  $h_{W_P}\cong -K_{X_P}$ on $X_P$ (where $h_{W_P}$ is the pullback of the class of a hyperplane section on $W_P$ to $X_P$), and the general surface $S$ equivalent to $h_{W_P}$ is a smooth quartic surface containing $\RP$, the union of a $5$-cycle of lines and a disjoint line $L_P$.  The  $3$-fold $W_P$ therefore has at most isolated singularities: it  is a $3$-fold of degree $14$ and any smooth curve section is a tetragonal canonical curve.  
The adjoint surface $A_P$ is a cubic surface through $\RP$ that is mapped to $\bar A_P$ in a $\PP^5\cap W_P$.  The plane sections $B_P\cap A_P$ that contain $L_P$ are mapped to lines in  $\bar A_P$, so $\bar A_P$ is a rational normal quartic scroll.  
Moreover, the image $\bar D$ of a general polytopal surface $D \in \Gamma_P$ 
is a non-minimal $K3$-surface with two $(-1)$-lines of degree $18$  and genus $11$. 
%with Picard rank $?+2$. 
The surface $D$ is the projection of $\bar D$ from $\mathrm{span}\lbrace \bar A_P\rbrace$, a $\PP^5$ that contains the two exceptional lines and intersects $\bar D$ in five more points.
 \hfill$\diamondsuit$
\end{example} 

\begin{example}
When $P$ is a hexagonal prism ($d=8$; see row $8$ in Table~\ref{tab:polytopes}), then 
computations with \texttt{Macaulay2} show that 
dim\;$\Gamma_P=3$ and that the ideal of $W_P\subset\PP^{11}$ is  generated by $22$ quadrics and $4$ cubics.
Residual to $W_P$ in the quadrics is a rational $\PP^2$-scroll of degree $5$ that intersects $W_P$ in  the image $\bar A_P$ of the adjoint surface,  a $K3$-surface of degree $12$ and genus $7$.
Indeed, the $3$-fold $W_P\subset \PP^{11}$ lies in a $\PP^5$-scroll over $\PP^1$.  Each $\PP^5$ of this scroll intersects $W_P$ in the image of a surface $B_P$, in this case a surface of degree $7$ and genus $3$. The adjoint surface $A_P$ is a quartic surface, and $\bar A_P$ in a $\PP^7\cap W_P$  is a $K3$-surface of degree $12$ and genus $7$.   The plane sections $B_P\cap A_P$ that contain $L_P$ are mapped to plane cubic curves in  $\bar A_P$.  So $\bar A_P$ lies in a rational normal $3$-fold scroll of degree $5$ in  a $\PP^7$ and this scroll is contained in every quadric that contains $W_P$. 
Moreover, the image $\bar D$ of a general polytopal surface $D \in \Gamma_P$ 
is a minimal elliptic surface of degree $26$ and genus $17$. % with Picard rank $2 ?$. 
The surface $D$ is the projection of $\bar D$ from $\mathrm{span}\lbrace \bar A_P\rbrace$, a $\PP^7$ that contains two plane cubic curves and intersects $\bar D$ in six more points.
 \hfill$\diamondsuit$
\end{example} 
 
\subsection{Non-prisms}
\label{ssec:Nprisms}
\begin{example}
When $P$ is a non-cube simple polytope with six facets (see row $4$ in Table~\ref{tab:polytopes}), then 
computations with \texttt{Macaulay2} show that 
dim\;$\Gamma_P=17$ and that the ideal of  $W_P\subset\PP^7$  is generated by $7$ quadrics.
The residual arrangement $\RP$ consists of three lines $L_1,L_2,L_3$ where $L_1$ and $L_3$ are skew lines that both intersect $L_2$, and two isolated points that span a line $L_0$ that does not intersect any of the other three lines $L_1, L_2, L_3$.   
%A general  Wachspress coordinate defines a smooth cubic surface in $\PP^3$.  These coordinates define a morphism $\omega_P$ on the blowup $X$ of $\PP^3$ along $\RP$.   On $X$ the union of the lines through each of the two points that meet the lines $L_i$ form altogether six planes that are contracted to lines in $W_P$.
 The adjoint surface $A_P$ is a smooth quadric surface that contains $\RP$. 
 It is mapped by $\omega_P$ to a smooth quadric surface in a $\PP^3\subset \PP^7$.    
The planes through $L_0$ are mapped to Del Pezzo surfaces of degree $4$ in~$W_P$.  Each one of them is a complete intersection of two quadrics in a $\PP^4$, so $W_P$ is contained in a five-dimensional rational normal $\PP^4$-scroll  over $\PP^1$ of degree $3$ in $\PP^7$. 
 %  The line $L_0$ is mapped to a line $\bar L_0$ in $W_P$, so the scroll $Y$ is the cone over $\PP^1\times \PP^2$ with vertex $\bar L_0$.
% A $\PP^6$ section of $W_P$ that contains one of these Del Pezzo surfaces, contains residually the image of a quadric surface through $L_1\cup L_2\cup L_3$, a rational normal scroll of degree $4$.
%In the scroll $Y$ the threefold $W_P$ is linked in two quadric sections to the union of two quadric threefolds.
Moreover, the image $\bar D$ of a general polytopal surface $D \in \Gamma_P$ 
 is a non-minimal $K3$-surface with two $(-1)$-lines of degree $14$ and genus $9$. The surface $D$ is the projection of $\bar D$ from $\mathrm{span}\lbrace \bar A_P\rbrace$, a $\PP^3$ that is spanned by the exceptional lines and intersects $\bar D$ in two more points.
  \hfill$\diamondsuit$
\end{example}

\begin{example}
When $P$ is a simple polytope with seven facets, including two hexagonal facets (see row $7$ in Table~\ref{tab:polytopes}), then  
computations with \texttt{Macaulay2} show that 
dim\;$\Gamma_P=4$ and that the ideal of $W_P\subset\PP^9$ is generated by $12$ quadrics and $3$ cubics.  
Residual to $W_P$ in the quadrics is a Segre threefold, $\PP^1\times \PP^2$, that intersects $W_P$ in  the surface $\bar A_P$,  a rational surface of degree $7$ and genus $3$.
The residual arrangement $\RP$ consists of six lines and three collinear isolated points. % $q_1,q_2,q_3$. 
%The lines are lines of intersection of five of the planes in $P$, and form a curve $C_0$ of arithmetic genus $3$.  The isolated points all lie on the line of intersection of the remaining two planes of $P$, the planes of hexagonal facets.  Three lines in $C_0$ pass trough the same point and span $\PP^3$, so any surface that contain $C_0$ is singular at this point.
%Still, $C_0$ is determinantal, defined by the $3\times 3$-minors of a $3\times 4$-matrix with linear entries.   
%Since $P$ is simple, the line $L$ of the three isolated points in $\RP$ does not intersect $C_0$, so the  adjoint surface does not contain that line.    A general  Wachspress coordinate defines a nodal quartic surface in $\PP^3$.  
Let $L$ be the line containing the three isolated points. 
The Wachspress map $\omega_P$ maps the  line $L$ to a line $\bar L$ in $W_P$.
 The planes through $L$ generate a pencil.  The residual arrangement $\RP$ intersects each of these planes in $9$ points.   These planes are mapped by $\omega_P$ to rational surfaces  of degree $7$ and genus $3$ that each span a $\PP^5$.   Any two of these surfaces span $\PP^9$.  Therefore $W_P$ is contained in a rational normal  $\PP^5$-scroll  over $\PP^1$,  which has degree $4$ and is a cone with vertex $\bar L$ over a $\PP^1\times \PP^3$.
The unique  plane cubic curve through the nine points is the intersection of each of these planes with the adjoint surface $A_P$.  
%The unique cubic plane cucurve These planes are, by $\omega_P$ mapped to surfaces of degree $7$ and genus $3$
These plane cubic curves are mapped by $\omega_P$ to plane cubic curves in $W_P$. 
% The adjoint surface $A_P$, is a nodal cubic surface that contains $\RP$.  It is mapped by $\omega_P$ to a surface of degree $7$ and genus $3$ in a $\PP^5\subset \PP^9$. 
% The $\PP^5$ span of the adjoint surface intersects $Y$ in a Segre threefold scroll $\PP^1\times \PP^2$;  the planes in the Segre threefold  intersect the adjoint surface in the plane cubic curves.
 %The quadrics in the ideal of $W_P$ therefore all contain  this Segre threefold. 
Moreover, the image $\bar D$ of a general polytopal surface $D \in \Gamma_P$ 
  is a minimal elliptic surface of degree $22$ and genus~$15$. % with Picard rank $2$. 
The surface $D$ is the projection of $\bar D$ from $\mathrm{span}\lbrace \bar A_P\rbrace$, a $\PP^5$ spanned by two plane cubic curves that intersects $\bar D$ in three more points.
 \hfill$\diamondsuit$
\end{example} 
\begin{example}
When $P$ is obtained by cutting off a vertex of a perturbed cube (see row $5$ in Table~\ref{tab:polytopes}), then 
computations with \texttt{Macaulay2} show that 
dim\;$\Gamma_P=7$ and  the ideal of $W_P\subset\PP^9$ is generated by $14$ quadrics.
%$\bar A_P$ is a smooth Del Pezzo surface of degree $5$ in a $\PP^5$.
%Indeed, $\RP$ consists of one isolated point and  six lines with no triple points.  
%The lines form a curve $C$ of arithmetic genus $1$, while the point does not lie on any trisecant line to $C$.  
%There are at most two lines through any point on $C$.   %A general  Wachspress coordinate defines a smooth quartic surface in $\PP^3$.  
 %These coordinates define a morphism $\omega_P$ on the blowup $X$ of $\PP^3$ along $\RP$.  
% There are three lines $ L_1,L_2,L_3$ in $C$ that are intersections of a plane of a quadrangle and the plane of a pentagon.  A plane through $L_i$ is intersected by  four lines of $C$ outside $L_i$.  This plane is mapped by $\omega_P$ to a Del Pezzo quintic surface in a $\PP^5$ in $\PP^9$.  Any two in the same pencil spans $\PP^9$, so $W_P$ is contained in three rational normal $\PP^5$-scrolls $Y_1,Y_2,Y_3$ over $\PP^1$, each of codimension $3$, so of degree $4$.  
% \begin{remark} A Del Pezzo surface of degree $5$ is the intersection of  two  Segre threefolds inside a quadric in $\PP^5$, so it is natural to suggest that $W_P=Y_1\cap Y_2\cap Y_3$, i.e. that each intersection $Y_i\cap Y_j$ is a scroll that intersects each $\PP^5$ of either $Y_i$ or $Y_j$ in a Segre threefold.  
%\end{remark}
The adjoint surface $A_P$ is a smooth cubic surface that contains $\RP$.  It is mapped by $\omega_P$ to a Del  Pezzo surface of degree $5$ in a $\PP^5\subset \PP^9$.  
Moreover, the image $\bar D$ of a general polytopal surface $D \in \Gamma_P$ 
is a non-minimal $K3$-surface with three $(-1)$-lines of degree $19$  and genus $12$. %with Picard rank $?+3$. 
The surface $D$ is the projection of $\bar D$ from $\mathrm{span}\lbrace \bar A_P\rbrace$, a $\PP^5$ that is spanned by the three exceptional lines and intersects $\bar D$ in three more points.
 \hfill$\diamondsuit$
\end{example} 

\begin{example}
When $P$ is a non-prism simple polytope with eight facets, none of which are triangles (see last row of Table~\ref{tab:polytopes}), then 
computations with \texttt{Macaulay2} show that 
dim\;$\Gamma_P=1$ and that the ideal of $W_P\subset\PP^{11}$ is generated by $22$ quadrics.
The residual arrangement $\RP$ consists of ten lines with no triple points that form a curve $C$ of arithmetic genus $7$.   
%A general  Wachspress coordinate defines a smooth quintic surface in $\PP^3$.  These coordinates define a morphism $\omega_P$ on the blowup $X$ of $\PP^3$ along $\RP$ onto $W_P\subset \PP^{11}$.  
There are two lines in $C$ that are intersections of two planes of  pentagons. Let $L$ be one of these lines.  Then $L$ intersects two other lines in $C$, so a plane through $L$ is intersected by seven lines of $C$ outside $L$.  This plane is mapped by $\omega_P$ to a surface of degree $9$ and genus $3$ in a $\PP^7$ in $\PP^{11}$.
The images of any two planes in the same pencil span the whole $\PP^{11}$, so $W_P$ is contained in two rational normal $\PP^7$-scrolls  over $\PP^1$, each of codimension $3$, so of degree $4$.  For each line in $C$ of intersection between a plane of a quadrangle and plane of a pentagon there is a similar associated  $\PP^8$-scrolls  over $\PP^1$ that contains $W_P$, each of codimension $2$, so of degree $3$.
% and for each line in $C$ of intersection between two planes of a quadrangles there is a similar associated  $\PP^9$-scrolls  over $\PP^1$ that contains $W_P$, each of codimension $2$, so of degree $3$. 
%It is natural to suggest that $W_P=Y_1\cap Y_2\cap Y_3\cap Y_4$, i.e. that each intersection $Y_i\cap Y_j\cap Y_k$ is a scroll that intersects each $\PP^6$ of either $Y_i, Y_j$ or $Y_k$ in a threefold.  
The adjoint surface $A_P$ is a smooth quartic surface that contains $\RP$.
  It is mapped by $\omega_P$ to a surface of degree $12$ and genus $7$ in a $\PP^7\subset \PP^{11}$.  
Moreover, the image $\bar D$ of a general polytopal surface $D \in \Gamma_P$ 
 is a non-minimal $K3$-surface with four $(-1)$-lines of degree $24$  and genus $15$.
%with Picard rank $?+4$. 
The surface $D$ is the projection of $\bar D$ from $\mathrm{span}\lbrace \bar A_P\rbrace$, a $\PP^7$ that is spanned by the four exceptional lines and intersects $\bar D$ in four more points.
 \hfill$\diamondsuit$
\end{example}

\bigskip

\paragraph{Acknowledgements.}
This material is based upon work supported by the National Science Foundation under Grant No. DMS-1439786 while the authors were in residence at the Institute for Computational and Experimental Research in Mathematics in Providence, RI, during the Fall 2018 semester.
We are grateful to Rainer Sinn, Frank Sottile, Hal Schenck and Dmitrii Pasechnik for many insightful discussions.
We also thank the anonymous referees for helpful suggestions that improved our exposition. 

\bigskip




\end{document}